\documentclass[11pt,a4paper,leqno,noamsfonts]{amsart}
\usepackage{a4wide}
\numberwithin{equation}{section}
\allowdisplaybreaks
\usepackage[british]{babel}
\usepackage[dvipsnames]{xcolor}
\usepackage{graphicx,pifont,comment}
\usepackage[utopia]{mathdesign}
\usepackage[a4paper,top=25mm,bottom=25mm,left=25mm,right=25mm,marginparwidth=60pt]{geometry}
\usepackage[utf8]{inputenc}

\usepackage{braket,caption,comment,mathtools,stmaryrd,enumitem,bm,tensor}
\usepackage[usestackEOL]{stackengine}
\usepackage{multirow,booktabs,microtype,relsize,soul}
\usepackage[foot]{amsaddr}
\usepackage[colorlinks,bookmarks]{hyperref} %
      \hypersetup{colorlinks,%
            citecolor=britishracinggreen,%
            filecolor=black,%
            linkcolor=cobalt,%
            urlcolor=RawSienna}

\definecolor{cornellred}{rgb}{0.7, 0.11, 0.11}
\definecolor{britishracinggreen}{rgb}{0.0, 0.26, 0.15}
\definecolor{cobalt}{rgb}{0.0, 0.28, 0.67}

\usepackage[colorinlistoftodos]{todonotes} 
\usepackage{faktor}
\usepackage[all]{xy}
\usepackage{tikz}
\usepackage{tikz-cd}
\usepackage{adjustbox}
\usepackage{rotating}
\usepackage{pict2e}

\usetikzlibrary{matrix,shapes,intersections,arrows,decorations.pathmorphing,decorations.markings}
\tikzset{commutative diagrams/.cd,
mysymbol/.style={start anchor=center,end anchor=center,draw=none}}
\newcommand\MySymb[2][\square]{%
  \arrow[mysymbol]{#2}[description]{#1}}
\tikzset{
shift up/.style={
to path={([yshift=#1]\tikztostart.east) -- ([yshift=#1]\tikztotarget.west) \tikztonodes}
}
}

\newcommand{\adjunction}[4]{%
  \begin{tikzcd}[ampersand replacement=\&, column sep=4em]
    #1 \colon #2
      \arrow[r, shift left=0.8ex]
      \arrow[r, phantom, "\scriptstyle\perp"]
    \&
    #3 \colon #4
      \arrow[l, shift left=0.8ex]
  \end{tikzcd}%
}

\usepackage[capitalise]{cleveref}
\usepackage{etoolbox}
\newlist{defenum}{enumerate}{1} 
\setlist[defenum]{label=$(\arabic*)$, ref=\thedefinition$(\arabic*)$}
\crefalias{defenumi}{definition} 
\newlist{lemenum}{enumerate}{1} 
\setlist[lemenum]{label=$(\arabic*)$, ref=\thelemma$(\arabic*)$}
\crefalias{lemenumi}{lemma} 
\newlist{assumpenum}{enumerate}{1} 
\setlist[assumpenum]{label=$(\arabic*)$, ref=\theassumption$(\arabic*)$}
\crefalias{assumpenumi}{assumption} 
\newlist{propenum}{enumerate}{1} 
\setlist[propenum]{label=$(\arabic*)$, ref=\theprop$(\arabic*)$}
\crefalias{propenumi}{proposition}
\newlist{exmpenum}{enumerate}{1} 
\setlist[exmpenum]{label=$(\arabic*)$, ref=\theexample$(\arabic*)$}
\crefalias{exmpenumi}{example}
\newlist{thmenum}{enumerate}{1} 
\setlist[thmenum]{label=$(\arabic*)$, ref=\thetheorem$(\arabic*)$}
\crefalias{thmenumi}{theorem}
\newlist{corenum}{enumerate}{1} 
\setlist[corenum]{label=$(\arabic*)$, ref=\thecorollary$(\arabic*)$}
\crefalias{corenumi}{corollary}
\newlist{factenum}{enumerate}{1} 
\setlist[factenum]{label=$(\arabic*)$, ref=\thefact$(\arabic*)$}
\crefalias{factenumi}{fact}
\newlist{remarkenum}{enumerate}{1} 
\setlist[remarkenum]{label=$(\arabic*)$, ref=\theremark$(\arabic*)$}
\crefalias{remarkenumi}{remark}
\newlist{axiomenum}{enumerate}{1} 
\setlist[axiomenum]{label=$(\arabic*)$, ref=\theaxiom$(\arabic*)$}
\crefalias{axiomenumi}{axiom} 
\theoremstyle{definition}
\newtheorem*{lemma*}{Lemma}
\newtheorem*{theorem*}{Theorem}
\newtheorem*{example*}{Example}
\newtheorem*{fact*}{Fact}
\newtheorem*{notation*}{Notation}
\newtheorem*{definition*}{Definition}
\newtheorem*{prop*}{Proposition}
\newtheorem*{remark*}{Remark}
\newtheorem*{corollary*}{Corollary}

\newtheorem*{conventions*}{Conventions}

\newtheorem{definition}{Definition}[section]

\newtheorem{example}[definition]{Example}
\newtheorem{fact}[definition]{Fact}

\newtheorem{notation}[definition]{Notation}
\newtheorem{construction}[definition]{Construction}
\newtheorem{remark}[definition]{Remark}
\newtheorem{warning}[definition]{Warning}

\newtheoremstyle{thm} 
        {3mm}
        {3mm}
        {\slshape}
        {0mm}
        {\bfseries}
        {.}
        {1mm}
        {}
\theoremstyle{thm}
\newtheorem{theorem}[definition]{Theorem}
\newtheorem{corollary}[definition]{Corollary}

\newtheorem{lemma}[definition]{Lemma}
\newtheorem{prop}[definition]{Proposition}

\newtheorem{thm}{Theorem}

\newtheorem*{Acknowledgments*}{Acknowledgments}



\newcommand{\BA}{{\mathbb{A}}}
\newcommand{\BB}{{\mathbb{B}}}
\newcommand{\BC}{{\mathbb{C}}}
\newcommand{\BD}{{\mathbb{D}}}
\newcommand{\BE}{{\mathbb{E}}}

\newcommand{\BH}{{\mathbb{H}}}

\newcommand{\BK}{{\mathbb{K}}}
\newcommand{\BL}{{\mathbb{L}}}
\newcommand{\BM}{{\mathbb{M}}}
\newcommand{\BN}{{\mathbb{N}}}

\newcommand{\BP}{{\mathbb{P}}}
\newcommand{\BQ}{{\mathbb{Q}}}
\newcommand{\BR}{{\mathbb{R}}}
\newcommand{\BS}{{\mathbb{S}}}
\newcommand{\BT}{{\mathbb{T}}}

\newcommand{\BZ}{{\mathbb{Z}}}

\newcommand{\CC}{{\mathcal C}}

\newcommand{\CF}{{\mathcal F}}

\newcommand{\CM}{{\mathcal M}}

\newcommand{\CX}{{\mathcal X}}
\newcommand{\CY}{{\mathcal Y}}

\newcommand{\FM}{{\mathfrak{M}}}

\DeclareSymbolFont{cmarrows}{OMS}{cmsy}{m}{n}
\SetSymbolFont{cmarrows}{bold}{OMS}{cmsy}{b}{n}
\DeclareMathSymbol{\cmminus}{\mathbin}{cmarrows}{"00}
\DeclareRobustCommand\relbar{\mathrel{\smash\cmminus}}
\DeclareMathSymbol{\leftrightarrow}{\mathrel}{cmarrows}{"24}
\DeclareMathSymbol{\leftarrow}{\mathrel}{cmarrows}{"20}
   
\DeclareMathSymbol{\rightarrow}{\mathrel}{cmarrows}{"21}
   \let\to=\rightarrow
\DeclareMathSymbol{\mapstochar}{\mathrel}{cmarrows}{"37}
   \def\mapsto{\mapstochar\rightarrow}    

\DeclareSymbolFont{usualmathcal}{OMS}{cmsy}{m}{n}
\DeclareSymbolFontAlphabet{\mathcal}{usualmathcal}
\DeclareMathAlphabet\BCal{OMS}{cmsy}{b}{n} 

\makeatletter
\newcommand{\mylabel}[2]{#2\def\@currentlabel{#2}\label{#1}}
\makeatother

\newcommand*{\defeq}{\mathrel{\vcenter{\baselineskip0.5ex \lineskiplimit0pt
                     \hbox{\scriptsize.}\hbox{\scriptsize.}}}=}
                     
\renewcommand{\geq}{\geqslant}
\renewcommand{\le}{\leqslant}
\renewcommand{\ge}{\geqslant}

\DeclareMathOperator{\Gr}{Gr}

\DeclareMathOperator*{\colim}{colim}
\newcommand{\Mbar}{\overline{\CM}_{g,n}(Y,\beta)}
\newcommand{\MB}{\overline{\CM}_{g,n}}

\newcommand{\preSUSY}{{\Mf^{\mathsf{pre-}\mathrm{SUSY}}_{g,\nf}}}
\newcommand{\SUSY}{{\Mf^{\mathsf{st-}\hbox{\rm\tiny SUSY}}_{g,\mathfrak n}(\boldit Y,\beta) }}

\newcommand{\OO}{\mathscr{O}}
\newcommand{\HH}{\mathrm{H}}

\newcommand{\ch}{\mathsf{ch}}

\DeclareMathOperator{\Spec}{Spec}

\DeclareMathOperator{\rk}{rk}

\DeclareMathOperator{\RR}{\mathbf{R}}
\DeclareMathOperator{\LL}{\mathbf{L}}

\DeclareMathOperator{\vd}{vdim}

\DeclareMathOperator{\Ob}{Ob}

\newcommand{\simto}{\,\widetilde{\to}\,}
\newcommand{\into}{\hookrightarrow}

\DeclareMathOperator{\et}{\mathrm{\acute{e}t}}

\DeclareMathOperator{\pre}{\mathsf{pre}}

\DeclareMathOperator{\derived}{\mathbf{D}}

\DeclareMathOperator{\CCoh}{\mathscr{C}\kern-0.25em \mathit{o}\kern-0.1em\mathit{h}}
\DeclareMathOperator{\ncCCoh}{nc\kern-0.25em\CCoh}

\DeclareMathOperator{\Fun}{\mathsf{Fun}}

\DeclareMathOperator{\Sets}{\mathsf{Sets}}
\DeclareMathOperator{\Mod}{Mod}
\DeclareMathOperator{\opp}{op}

\DeclareMathOperator{\Schs}{\mathsf{Sch}}
\DeclareMathOperator{\SSchs}{\mathsf{SSch}}
\DeclareMathOperator{\SSt}{\mathsf{SSt}}
\DeclareMathOperator{\St}{\mathsf{St}}
\DeclareMathOperator{\Id}{\mathsf{Id}}

\DeclareMathOperator{\Hom}{Hom}

\DeclareMathOperator{\lHom}{\mathscr{H}\kern-0.3em\mathit{om}}
\DeclareMathOperator{\lEnd}{\mathscr{E}\kern-0.2em\mathit{nd}}
\DeclareMathOperator{\RRlHom}{\mathbf{R}\kern-0.025em\mathscr{H}\kern-0.3em\mathit{om}}

\DeclareMathOperator{\lExt}{{\mathscr{E}\kern-0.2em\mathit{xt}}}

\newcommand*{\boldone}{\text{\usefont{U}{bbold}{m}{n}1}}
\newcommand{\boldit}[1]{\boldsymbol{#1}}

\newcommand{\XX}{\boldit{X}}

\newcommand{\UU}{\boldit{U}}

\newcommand{\TT}{\boldit{T}}
\newcommand{\dS}{\boldit{S}}

\newcommand{\Z}{{\mathbb Z}}

\newcommand\Homst{{\mathfrak{Hom}}}

\newcommand\iso{\kern.35em{\raise3pt\hbox{$\sim$}\kern-1.1em\to}\kern.3em}

\newcommand{\DM}{Deligne--Mumford}

\makeatletter
\newcommand{\longrightharpoonup}{\relbar\joinrel\rightharpoonup}
\newcommand{\longleftharpoondown}{\leftharpoondown\joinrel\relbar}
\newcommand\noloc{%
  \nobreak
  \mspace{6mu plus 1mu}
  {:}
  \nonscript\mkern-\thinmuskip
  \mathpunct{}
  \mspace{2mu}
}
\newcommand{\smallbot}{%
  \begingroup\setlength\unitlength{.15em}%
  \begin{picture}(1,1)
  \roundcap
  \polyline(0,0)(1,0)
  \polyline(0.5,0)(0.5,1)
  \end{picture}%
  \endgroup
}

\let\Xcal\XX
\let\Ycal\YY
\newcommand\Sc{\boldit{S}}
\newcommand\Tc{\boldit{T}}
\let\Ucal\UU
\let\U\UU
\let\Zc\ZZ

\newcommand\nf{\mathfrak{n}}
\let\Oc\OO
\newcommand{\Wc}{\boldit{W}}

\newcommand{\Xf}{\mathfrak{X}}
\newcommand\Mf{\BCal{M}}

\newcommand{\Lcl}{\mathcal{L}}

\newcommand{\bos}{\operatorname{bos}}

\makeatother

\title[Derived supergeometry and perfect obstruction theories on algebraic superstacks]{Derived supergeometry  and perfect obstruction \\ theories on algebraic superstacks}
\author{Ugo Bruzzo, Daniel Hern\'andez Ruip\'erez, 
\\ Emanuele Pavia and Andrea T. Ricolfi}
\thanks{Research partly supported by Bolsa de Produtividade 305343/2025-4 of Brazilian CNPq   (UB) and by the  research projects ``Espacios finitos y functores integrales'', MTM2017-86042-P (Ministerio de Econom\'{\i}a, Industria y Competitividad) and GIR 4139 of the University of Salamanca (DHR)}

\begin{document}

\begin{abstract}
We develop a deformation-theoretic framework for algebraic superstacks and apply it to the construction of perfect obstruction theories on moduli spaces of stable supermaps. In order to achieve this we employ the tools and concepts from derived algebraic geometry, extending them to the supergeometric setting. Our main technical contributions are (i) the notions of derived superstacks, (ii) their cotangent supercomplexes and (iii) perfect obstruction theories for algebraic superstacks. We prove that a quasi-smooth derived enhancement of an algebraic superstack always yields a perfect obstruction theory: this generalizes a well-known principle from classical deformation theory. 
 
We then turn to the case of the Deligne--Mumford superstack $\Mf^{\mathsf{st-}\mathrm{SUSY}}_{g,\nf}(\Ycal)$ of stable supermaps of fixed genus, homology class and numbers of Neveu--Schwarz and Ramond--Ramond punctures, with target a smooth projective superscheme $\Ycal$. We exhibit a natural derived enhancement $\BR \Mf^{\mathsf{st-}\mathrm{SUSY}}_{g,\nf}(\Ycal)$ and compute its cotangent supercomplex; we prove that the induced obstruction theory on $\Mf^{\mathsf{st-}\mathrm{SUSY}}_{g,\nf}(\Ycal)$ is perfect. Finally, we compute its virtual dimension by super Grothendieck--Riemann--Roch. The resulting formula gives, in particular, a deformation-theoretic proof of the virtual dimension formula previously conjectured for stable supermaps and specializes to the usual virtual dimension of the classical moduli stack of stable maps.
\end{abstract}

\maketitle
{\hypersetup{linkcolor=black}\tableofcontents}

\section{Introduction}
\subsection{Overview and motivation}
Algebraic geometry can be seen as ``contained in'', or generalized by two modern theories: derived algebraic geometry (DAG, for short), and 
algebraic supergeometry (ASG, for short). In algebraic geometry, especially in enumerative geometry, one is interested in moduli stacks and in using deformation theory to extract invariants out of them. The corresponding enhanced spaces in both generalized theories exist: \emph{derived stacks} and \emph{superstacks} are fundamental objects constituting a very active field of research. While deformation theory is naturally embedded within DAG, via the theory of cotangent complexes, this is not the case for supergeometry. The first technical achievement of our paper is to fill this gap: it provides a derived-type enhancement of supergeometry with a built-in coherent and deformation-theoretic framework incorporating a functional theory of cotangent complexes. 

Our main motivation comes from enumerative geometry: the cotangent complex is used in the definition of an \emph{obstruction theory} on a moduli stack, and in turn a \emph{perfect} obstruction theory is the key technical tool used to extract invariants out of a moduli stack \cite{BFinc}. The prototype is the perfect obstruction theory of the moduli stack of stable maps $\overline{\CM}_{g,n}(Y,\beta)$ to a smooth projective variety $Y$ (recalled in \Cref{sec:classical-stable-maps}), which is at the foundation of Gromov--Witten theory \cite{Behrend_GW}.

At a visual, and highly informal level, the main contribution of this paper is to fill in the ``pushout diagram''
\[
\begin{tikzcd}
\textrm{Algebraic Geometry}\arrow[hook]{r}\arrow[hook]{d} & \textrm{DAG}\arrow[dotted]{d} \\
\textrm{ASG}\arrow[dotted]{r} & \textrm{Derived supergeometry}
\end{tikzcd}
\]
by constructing a ``derived-super framework'' that is flexible enough to incorporate the notion of a perfect obstruction theory in algebraic supergeometry: as briefly outlined in \Cref{sec:intro-derived-super-stuff}, we introduce the notion of a \emph{derived superstack} and attach to each derived superstack $\BCal{X}$ a \emph{superderived $\infty$-category} $\mathbf{SD}_\infty(\BCal{X})$ containing a distinguished object $\BL_{\BCal{X}} \in \mathbf{SD}_\infty(\BCal{X})$ which satisfies the universal property and has all the formal properties enjoyed by the cotangent complex in DAG (cf.~\Cref{Main-Thm-Intro-1} below). 

\smallbreak
There are two related issues in carrying out this program. First, one needs a coherent derived category of quasicoherent supermodules in which cotangent objects can be defined and handled. Although the elementary homological algebra of modules is familiar, in order to deal with supermodules one needs to take account of both the cohomological grading and the additional $\mathbb Z/2\mathbb Z$-grading in a compatible way. Secondly, a theory defined locally on affine {derived} superschemes has to be globalized to algebraic {derived} superstacks. In particular, one needs descent for the corresponding derived categories and a sufficiently functorial construction of cotangent supercomplexes: these are well-known issues even in classical algebraic geometry, and they are known to \textit{fail} if one sticks to the triangulated setting.

\smallbreak
Once the general theory is in place, we apply it in \Cref{sec:stable-supermaps} to the superstack 
\[
\Mf^{\mathsf{st}-\mathrm{SUSY}}_{g,\nf}(\Ycal,\beta)
\]
of stable supermaps with target a smooth projective superscheme $\Ycal$. The output is a perfect obstruction theory on this Deligne--Mumford superstack, generalizing nontrivially the classical construction (cf.~\Cref{sec:classical-stable-maps}). We also compute the (super) virtual dimension, confirming the expectation formulated in \cite[Conj.~5.1]{BHRSM25}, see \Cref{thm:intro-derived-enhancement-maps} below.

We spend the rest of the introduction discussing some of our constructions and main results.

\subsection{Derived superstacks, and the superderived \texorpdfstring{$\infty$}{∞}-category}
\label{sec:intro-derived-super-stuff}

The starting point of the present paper is the homotopy theory of supercommutative algebras and homological algebra for supermodules, in the sense of derived algebraic geometry \textit{à la} To\"en--Vezzosi \cite{hag_2} and Lurie \cite{sag}. Let us remark that, in recent years, other sources have tried to adapt homotopical and derived tools to differential \cite{Carchedi_Roytenberg,carchedi_derived_supergeometry} and algebraic \cite{hesselholt_pstragowski_I,hesselholt_pstragowski_II,dang_derived_supergeometry} supergeometry, in order to define stable enhancements of derived categories of supermodules. However, our approach differs in two main aspects.
\begin{enumerate}
\item Since our paper is mostly aimed at an audience of classical algebraic supergeometers, we do not derive our statements and constructions using other derived theories (e.g., Dirac geometry) as black boxes.
 \item For the same reason, our approach tries to emphasize the relation with the standard theory of derived categories in terms of localizations of cochain complexes of supermodules, using model categories and differential graded categories as our preferred models for $\infty$-categories (see \cref{sec:dg-supermodules}).
\end{enumerate}
We start from a commutative ring $R$ of characteristic different from $2$. Denote by $\mathrm{SMod}_R$ the abelian category of supermodules over $R$, equipped with the symmetric monoidal structure given by the supertensor product. In \cref{sec:dg-supermodules}, we construct a stable differential graded enhancement (hence, a stable $\infty$-category) of $\mathrm{SMod}_R$, using the homotopy theory of cochain complexes of supermodules as a building block: we call such $\infty$-category the \textit{superderived $\infty$-category}, and denote it as $\mathbf{SD}_\infty(R)$. The derived supertensor product endows $\mathbf{SD}_\infty(R)$ with a symmetric monoidal structure, with a symmetry constraint governed by the super sign rule: therefore, we can consider commutative algebra objects inside the superderived $\infty$-category: we define the $\infty$-category of
\textit{derived supercommutative $R$-algebras} as
\[
\mathrm{dSCAlg}_R\defeq \mathrm{CAlg}(\mathbf{SD}_{\infty}(R)).
\]
In \cref{def:super-derived}, for an arbitrary derived supercommutative algebra $\BA$, we define its own superderived $\infty$-category
\[
\mathbf{SD}_\infty(\BA)\defeq\mathrm{Mod}_{\BA}(\mathbf{SD}_\infty(R)),
\]
which again inherits a symmetric monoidal structure given by the relative derived supertensor product $\otimes^{\mathrm{super}}_{\BA}$, and so allows us to work in a relative setting as well: indeed, we can define the $\infty$-category of \textit{derived supercommutative $\BA$-algebras} as
\[
\mathrm{dSCAlg}_{\BA}\defeq\mathrm{CAlg}(\mathbf{SD}_\infty(\BA))\simeq\mathrm{CAlg}(\mathbf{SD}_\infty(R))_{\BA/}.
\]
Just like in the non-super setting, the perfect complexes inside $\mathbf{SD}_\infty(\BA)$ are defined as the fully dualizable objects under the supertensor product, yielding a stable full sub-$\infty$-category $\mathrm{SPerf}(\BA)\subseteq\mathbf{SD}_\infty(\BA)$ (cf.~\cref{subsec:perf_objects}). Just like in the non-super setting, this is identified as the full sub-$\infty$-category spanned by compact objects, and with the smallest stable $\infty$-category containing the compact generators $\BA$ and $\Pi\BA$ (cf.~\cref{prop:superf}).

Morevoer, as we prove in \cref{prop:t-structure}, the superderived $\infty$-category $\mathbf{SD}_\infty(\BA)$ of any connective derived supercommutative $R$-algebra $\BA\in\mathrm{dSCAlg}_R^{\leqslant0}$ admits a very well-behaved natural $t$-structure, whose heart recovers the ordinary abelian categories of supermodules over the underived superalgebra $\BH^0(\BA)$. In particular, any supercommutative algebra can be seen as a derived supercommutative algebra concentrated in cohomological degree $0$.

Moreover, in \cref{sec:grothendieck_topologies} we show that $\infty$-categories of derived supercommutative algebras admit Grothendieck topologies that naturally generalize the relevant (fpqc, fppf, étale, Zariski) topologies on underived supercommutative algebras. In this way, the theory of algebraic supergeometry can be embedded in its derived counterpart.

Derived superstacks are then defined in \cref{sec:dasg} following a functor-of-points approach: given a connective derived supercommutative algebra $\BA$, we denote by $\mathrm{dSAff}_{\BA}$ the opposite $\infty$-category of $\mathrm{dSCAlg}_{\BA}^{\leqslant0}$ and call it the $\infty$-category of \textit{derived affine superschemes over $\BA$}. We then define (cf.~\Cref{def:derived-super-prestack}) a \textit{derived pre-superstack over $\BA$} to be a presheaf 
\[
\begin{tikzcd}
\BCal{X}\colon \mathrm{dSAff}^{\mathrm{op}}_{\BA}\arrow{r} & \mathrm{Spaces}
\end{tikzcd}
\]
that is \textit{accessible}: there exists a small regular cardinal $\kappa$ such that $\BCal{X}$ preserves filtered $\kappa$-colimits. For a connective derived supercommutative $\BA$-algebra $\BS$, we denote by $\Spec \BS$ the corresponding derived pre-superstack.  Then, in \Cref{def:derived-superstack} we define a \emph{derived superstack} as a derived pre-superstack that is a sheaf with respect to the fpqc, or to the \'etale topology.

With this technology in place, we can define the $\infty$-category which will host the cotangent complex $\BL_{\BCal X}$ of a derived superstack: we define it in \Cref{def:derived_qcoh} as the limit of presentable symmetric monoidal $\infty$-categories
\[
\mathbf{SD}_\infty(\BCal{X})\defeq\lim_{\Spec\BS\in\mathrm{dSAff}_{/\BCal{X}}^{\mathrm{op}}}\mathbf{SD}_\infty(\BS),
\]
using the fact that superderived $\infty$-categories glue along faithfully flat covers (\cref{thm:flat_descent}). 

\subsection{Cotangent complexes on derived superstacks and obstruction theories on superstacks}
\label{sec:intro-super-pot}

Let $\varphi\colon\BCal X\to\BCal Y$ be a morphism of derived superstacks. The main technical novelty of this work is the definition of the relative cotangent complex $\BL_{\BCal X/\BCal Y}$ associated to the morphism $\varphi$, and the study of the connectivity properties of $\BL_{\BCal X/\BCal Y}$ that are detected from the connectivity properties of $\mathrm{cone}(\varphi)$. Our definition in \cref{sec:cotangent_pot} is, just as in DAG, completely functorial and in terms of a deformation-theoretic universal property, cf.~\Cref{def:cotangent-global} for full details. 

\begin{thm}[{\Cref{thm:cotg-exists-and-induced-pot}\ref{cotg-exists}}]
\label{Main-Thm-Intro-1}
Let $\BCal X$ be a locally geometric\footnote{The definition of local geometricity is parallel to the one in DAG, see \Cref{def:geometricity}.} derived superstack over a base $\BA$. Then there exists a distinguished object $\BL_{\BCal X} \in \mathbf{SD}_\infty(\BCal{X})$, called the \emph{cotangent complex of} $\BCal X$, uniquely defined (up to a contractible space of choices) by the following properties.
\begin{enumerate}
\item [\mylabel{cotg-1}{\normalfont{(1)}}] For any derived supercommutative $\BA$-algebra $\BS$ and for any~~$\BS$-rational point $x\colon\Spec\BS\to\BCal X$, the pullback $x^\ast\BL_{\BCal X}$ classifies higher $\BA$-linear derivations from $\Spec\BS$ to $\BCal X$.
\item [\mylabel{cotg-2}{\normalfont{(2)}}] For any~~$\BS$-rational point $x\colon\Spec\BS\to\BCal X$, and for any map $\varphi\colon \Spec \BT\to\Spec\BS$ of derived affine schemes, there is a natural equivalence $\varphi^\ast x^\ast\BL_{\BCal X}\simeq(x\circ\varphi)^\ast\BL_{\BCal X}$.
\end{enumerate}
\end{thm}

The cotangent complex $\BL_{\BCal X}$ is characterized locally by its universal property and is compatible with base change --- see \Cref{def:super_cotangent} for the local definition and \Cref{remark:cotangent_universal_property} for the universal property in terms of higher $\BA$-linear super-derivations. Thus $\BL_{\BCal X}$ has the expected deformation-theoretic interpretation, in that it classifies infinitesimal deformations and governs the lifting problem. This leads naturally to a definition of obstruction theory for an algebraic superstack $\BCal X$, parallel to the original Behrend--Fantechi definition \cite{BFinc}, as follows.

\begin{definition}[\Cref{def:superPOT}]
\label{def:INTRO-superPOT}
Let $f\colon \BCal X \to \BCal Y$ be a morphism of superstacks admitting cotangent complexes. An \emph{obstruction theory} relative to $f$ is the datum of an object $\BE\in \mathbf{SD}_\infty(\BCal X)$ along with a morphism $\psi\colon \BE \to \BL_{f}$ such that
\begin{enumerate}
    \item [\mylabel{intro-OT-i}{(i)}] $\BH^0(\psi)$ is an isomorphism of $\OO_{\BCal X}$-modules, and
    \item [\mylabel{intro-OT-ii}{(ii)}] $\BH^{-1}(\psi)$ is a surjection of $\OO_{\BCal X}$-modules.
\end{enumerate}
An obstruction theory is said to be \emph{perfect} if, moreover, $\BE$ is perfect of cohomological amplitude $[-1,0]$. If $\BCal{X}$ is defined over an underived base commutative superalgebra $\BA$, an obstruction theory on the structure morphism $\BCal{X} \to \Spec \BA$ is simply called an obstruction theory on $\BCal X$.

Finally, the \emph{virtual dimension} of a perfect obstruction theory $\BE \to \BL_{\BCal X}$ is defined to be the rank of $\BE$, in the sense of \Cref{free-rank-global}.
\end{definition}

In the definition, flat descent of super derived $\infty$-categories (cf.~\Cref{thm:flat_descent}) allows us to define perfect objects in $\mathbf{SD}_\infty(\BCal Y)$ in terms of local charts, cf.~\Cref{Dperf} for the global case, building upon the local one, developed in \Cref{subsec:perf_objects}.

This technical machinery, developed in \cref{sec:homological_super_algebra,sec:dasg}, implies that perfect obstruction theories on superstacks can be systematically produced from derived enhancements, generalizing a well-known principle from classical algebraic geometry \cite{DerivedQuot,DerivedHilbert,Schurg_Toen_Vezzosi_Determinants_of_perfect_complexes}.  

\begin{thm}[{\Cref{thm:cotg-exists-and-induced-pot}\ref{induced-POT}}]
\label{Main-Thm-Intro-2}
Assume $\BCal{X}$ is an algebraic superstack over an underived supercommutative algebra $\BA$, and let $\BR\BCal X$ be a locally geometric derived enhancement of $\BCal X$ with closed immersion $\iota \colon \BCal X\into \BR\BCal X$. Then, the canonical morphism $\iota^*\BL_{\BR\BCal X}\to \BL_{\BCal X}$ is an obstruction theory on $\BCal X$.
\end{thm}

Consequently, as pointed out in \Cref{cor:quasi-smooth-pot}, the problem of constructing a \emph{perfect} obstruction theory on an algebraic superstack $\BCal X$ can be turned into the problem of constructing a suitable derived enhancement and proving that it is quasi-smooth (cf.~\Cref{def:quasi-smooth}).

\subsection{Application to stable supermaps}
\label{sec:intro-application-stablemaps}

In \Cref{sec:stable-supermaps} we turn to the geometric application that mainly motivates our paper: the moduli of stable supermaps. We work with a smooth projective superscheme $\boldit{Y}$, fix an arithmetic genus $g$, a class $\beta$ of $1$-supercycles, and a pair of nonnegative integers $\mathfrak n = (n_{\mathrm{NS}},n_{\mathrm{RR}})$ of Neveu--Schwarz and Ramond--Ramond punctures. A \textit{prestable SUSY curve} over a superscheme $\boldit{S}$ consists of a supercurve $f\colon\boldit{C}\to\boldit{S}$  of relative dimension $1|1$ together with the relevant punctures and the superconformal structure encoded by a surjective morphism $\delta\colon\Omega^1_f\to\omega_f$. Stability is imposed by the ampleness condition involving the relative dualizing sheaf and the punctures. We have an open immersion of smooth algebraic superstacks 
\[
\Mf^{\mathsf{st-}\mathrm{SUSY}}_{g,\nf}\subset \Mf_{g,\nf}^{\mathsf{\pre-}\mathrm{SUSY}},
\]
where, in addition, $\Mf^{\mathsf{st-}\mathrm{SUSY}}_{g,\nf}$ is proper and of Deligne--Mumford type.

A stable supermap to $\boldit{Y}$ is obtained by adding to such a prestable SUSY curve a morphism to   $\boldit{Y}$, together with the condition that the corresponding class in $\boldit{Y}$ is $\beta$ and an appropriate ampleness condition involving a polarization of $\boldit{Y}$. The resulting moduli problem defines a Deligne--Mumford superstack of stable supermaps, denoted
\[
\Mf^{\mathsf{st}-\mathrm{SUSY}}_{g,\nf}(\Ycal,\beta). 
\]
This is the natural supergeometric analog of the classical moduli stack of stable maps, and it is the setting in which a virtual fundamental class and related enumerative constructions are expected to arise from a perfect obstruction theory.

We obtain the following result.

\begin{thm}[\Cref{thm:pot-on-maps}, \Cref{thm:vdim-maps}]
\label{thm:intro-derived-enhancement-maps}
Let $\boldit Y$ be a smooth projective superscheme of dimension $\dim \boldit Y = r\vert s$. 
The Deligne--Mumford superstack $\Mf^{\mathsf{st}-\mathrm{SUSY}}_{g,\nf}(\Ycal,\beta)$ admits a natural derived enhancement 
\begin{equation}
\label{eqn:der-enhance-maps}
\BR\Mf^{\mathsf{st}-\mathrm{SUSY}}_{g,\nf}(\Ycal,\beta).
\end{equation} 
Moreover, $\Mf^{\mathsf{st}-\mathrm{SUSY}}_{g,\nf}(\Ycal,\beta)$ carries a canonical perfect obstruction theory. The even and odd parts of the virtual dimension are given by the formulas
\begin{align*}
   \mathrm{even\ vdim \ } \SUSY &= (r-3)(1-g) + \nf_{\mathrm{NS}}+ \nf_{\mathrm{RR}} (1+ s/2)
   + \int_{\beta_0} [ \ch_1(T_Y)-\ch_1(\mathcal F_{\boldit Y})], \\
   \mathrm{odd\ vdim\  } \SUSY &=  (1-g)(s-2)+ \nf_{\mathrm{NS}}+ (\nf_{\mathrm{RR}}/2) (r+1)
   + \int_{\beta_0} [ \ch_1(T_Y)-\ch_1(\mathcal F_{\boldit Y})].
   \end{align*}
Here, $\mathcal F_{\boldit Y}=\Pi \mathcal{J}/\mathcal{J}^2$, where $\mathcal{J}$ is the ideal generated by the odd elements of $\Oc_{\boldit{Y}}$.
\end{thm}
The enhancement \eqref{eqn:der-enhance-maps} is described as a derived open substack of the derived mapping stack
\[
\BR\Mf_{g,\nf}^{\pre-\mathrm{SUSY}}(\Ycal) \defeq \BR\underline{\Hom}_{/\BCal{M}^{\pre-\mathrm{SUSY}}_{g,\nf}}(\BCal{C}^{\pre-\mathrm{SUSY}}_{g,\nf},\BCal{M}^{\pre-\mathrm{SUSY}}_{g,\nf} \times \boldit{Y}),
\]   
where
\[
\begin{tikzcd}
\BCal{C}^{\pre-\mathrm{SUSY}}_{g,\nf} \arrow{r} & \BCal{M}^{\pre-\mathrm{SUSY}}_{g,\nf}
\end{tikzcd}
\]
is the universal curve.
Its cotangent supercomplex can then be computed locally: given a point $x$ corresponding to a stable supermap $\varphi \colon \boldit{C}\to \boldit{Y}$, the cotangent complex 
\[
x^\ast\BL_{\BR\BCal{M}^{\mathsf{st}-\mathrm{SUSY}}_{g,\nf}(\boldit{Y},\beta)}
\]
is computed by the cone
\[
\begin{tikzcd}
\mathrm{Cone}(\Gamma_\ast(\varphi^\ast T_{\boldit{Y}})^\vee\arrow{r} & \Gamma_\ast(\mathcal{A})^{\vee})[-1],
\end{tikzcd}
\]
where the source of the morphism governs the problem of deforming the underlying pre-stable SUSY curve  $\boldit{C}\to \boldit{S}$ with its punctures and superconformal structure, while the target governs the problem of deforming the map $\varphi$ while holding the curve fixed. When $\boldit{Y}$ is smooth and projective,   the derived global sections in the above formula are perfect and the cone lies in cohomological degrees $[-1,0]$. Restricting the cotangent supercomplex of $\BR\Mf^{\mathsf{st}-\mathrm{SUSY}}_{g,\nf}(\Ycal,\beta)$ to the classical moduli superstack $\Mf^{\mathsf{st}-\mathrm{SUSY}}_{g,\nf}(\Ycal,\beta)$ produces a perfect obstruction theory by \Cref{Main-Thm-Intro-2}.

The virtual dimension computation uses the super Grothendieck--Riemann--Roch theorem established in \cite{VMP} and naturally involves both the ordinary tangent bundle of the bosonic reduction of the target $\boldit{Y}$ and the odd directions of the superscheme. The latter are encoded by the sheaf $\mathcal F_{\boldit{Y}}=\Pi\mathcal J/\mathcal J^2$, where $\mathcal J$ is the ideal generated by the odd elements of the structure sheaf of $\boldit{Y}$. The odd contribution is therefore intrinsic to the supergeometry of the target rather than being an auxiliary correction appended to the classical formula. 
In the bosonic specialization, this computation recovers the corresponding classical deformation theory and, ultimately, the usual virtual dimension of the moduli of stable maps.

\subsection{Structure of the paper}
The paper is organized as follows. In \Cref{sec:background} we recall the classical background: perfect obstruction theories in the sense of Behrend--Fantechi \cite{BFinc}, with the guiding example of the classical moduli stack of stable maps \cite{Behrend_GW}. \Cref{sec:superstacks} reviews the basic theory of superstacks, including algebraic superspaces, Deligne--Mumford and algebraic superstacks, and the bosonic reduction.

\Cref{sec:homological_super_algebra} develops the homological foundations. We begin with introducing the differential graded enhancement of the abelian category of supermodules, construct the super derived $\infty$-category, and equip it with its symmetric monoidal structure. We then develop the theory of derived superalgebras and supermodules, establish the relevant $t$-structure and perfectness criteria, and define the (affine) cotangent supercomplex by the universal property associated with square-zero extensions. The final part of the section treats the Grothendieck topologies on derived superalgebras and proves descent for derived $\infty$-categories of supermodules. These results provide the technical foundation needed for the global theory.

\Cref{sec:dasg} passes from affine derived supergeometry to derived superstacks. We introduce geometric derived superstacks as sheaves for the étale topology over connective derived supercommutative algebras; we then introduce their superderived $\infty$-categories, prove the relevant locality properties, and globalize the cotangent supercomplex. We then introduce perfect obstruction theories for algebraic superstacks (cf.~\Cref{def:superPOT}) and explain their relation with derived enhancements and quasi-smoothness. 

Finally, \cref{sec:stable-supermaps} is devoted to the study of the moduli superstack of stable supermaps. After recalling prestable and stable SUSY curves and the definition of stable supermaps, we construct the derived enhancement $\BR\BCal{M}^{\mathrm{SUSY}}_{g,\nf}(\boldit{Y},\beta)$ of the moduli superstack $\BCal{M}^{\mathrm{SUSY}}_{g,\nf}(\boldit{Y},\beta)$. We compute its cotangent supercomplex locally, prove that the induced obstruction theory on the classical moduli superstack is perfect, and identify the corresponding deformation complex. The final subsection applies super Grothendieck--Riemann--Roch to compute the virtual dimension and verifies its compatibility with the previously known special cases.

\subsection*{Conventions and notations}
Classical schemes and algebraic spaces are written in standard font $X$. Classical Deligne--Mumford and algebraic stacks are written in calligraphic font $\CX$. Their super versions are denoted, respectively, $\boldit X$ and $\BCal{X}$. Derived superstacks are also denoted $\BCal{X}$ --- after all, ASG embeds in its derived counterpart --- but when we want to emphasize the presence of a derived enhancement of a superstack $\BCal{X}$, we denote it by $\BR\BCal{X}$. Sometimes, when no confusion is likely to arise, we denote by $X$, resp.~$\CX$, the bosonic reduction of a superscheme $\boldit X$, resp.~superstack $\BCal{X}$.
 
\section{Classical theory: obstruction theory on moduli of stable maps}
\label{sec:background}
In this section we recall from \cite{BFinc} the classical notion of a perfect obstruction theory on a Deligne--Mumford type morphism of algebraic stacks, and we sketch the well-known construction of the perfect obstruction theory on the moduli stack of stable maps $\Mbar$ with target a smooth projective complex variety.

\subsection{Perfect obstruction theories}
\label{sec:classical-POTs}
We review the classical theory of perfect obstruction theories in the setup of Behrend--Fantechi \cite{BFinc}. Let $\CY$ be a pure dimensional algebraic stack. Let $f\colon \CX \to \CY$ be a morphism of algebraic stacks with unramified diagonal (such morphisms are called \emph{of Deligne--Mumford type}), and let $\BL_f$ be the relative cotangent complex. Then one has $h^1(\BL_f) = 0$, so that $\BL_f$ is an object of the derived category $\derived^{\leqslant 0}(\CX_{\et})$. 

\begin{definition}[{\cite[Deff.~4.4,~5.1]{BFinc}}]
\label{def:pot}
An \emph{obstruction theory} relative to a morphism $f\colon \CX \to \CY$ of Deligne--Mumford type is a pair $(\BE,\varphi)$, where 
\begin{itemize}
    \item [$\circ$] $\BE \in \derived^{\leqslant 0}(\CX_{\et})$ is a complex with coherent cohomology in degrees $-1,0$ and vanishing cohomology in positive degrees, and
    \item [$\circ$] $\varphi\colon \BE \to \BL_f$ is a morphism in $\derived^{\leqslant 0}(\CX_{\et})$ such that $h^0(\varphi)$ is an isomorphism and $h^{-1}(\varphi)$ is surjective.
\end{itemize} 
We say that $(\BE,\varphi)$ is a \emph{perfect obstruction theory} if $\BE$ is of perfect amplitude contained in $[-1,0]$.
\end{definition}

\begin{definition}
The \emph{virtual dimension} of a perfect obstruction theory $(\BE,\varphi)$ is the integer $\vd = \rk \BE$. The \emph{obstruction sheaf} of $(\BE,\varphi)$ is the cohomology sheaf $\Ob = h^1(\BE^\vee)$.
\end{definition}

\subsection{The classical moduli stack of stable maps}
\label{sec:classical-stable-maps}
Fix integers $g,n \in \BZ_{\geqslant 0}$. We start recalling the definition of the stack $\FM_{g,n}$ of $n$-\emph{pointed prestable curves} of genus $g$. Recall that a purely 1-dimensional finite type $\BC$-scheme $C$ is a \emph{prestable curve} if it is proper, reduced, connected and has at worst nodal singularities. The arithmetic genus of a prestable curve $C$ is the integer $g = 1-\chi(\OO_C)$. For a $\BC$-scheme $B$, the groupoid of $B$-valued points of $\FM_{g,n}$ consists of tuples
\begin{equation}
\label{family-prestable-curves}
(\pi\colon \CC \to B,\sigma_1,\ldots,\sigma_n)
\end{equation}
where
\begin{itemize}
    \item [$\circ$] $\pi$ is a flat, proper morphism of schemes, of finite presentation and relative dimension 1, whose geometric fibers are prestable curves of arithmetic genus $g$,
    \item [$\circ$] $\sigma_1,\ldots,\sigma_n \colon B \to \CC$ are pairwise disjoint sections of $\pi$, landing in its smooth locus.
\end{itemize}

Two objects $(C,p_1,\ldots,p_n)$ and $(C',p_1',\ldots,p_n')$ of $\FM_{g,n}(\BC)$ are isomorphic is there is an isomorphism $C \simto C'$ sending $p_i$ to $p_i'$ for every $i$. Thus one sees that the automorphism group of a prestable $n$-pointed curve may be infinite. 

\begin{theorem}[{\cite[Prop.~2]{Behrend_GW}}]
The groupoid $\FM_{g,n}$ is a smooth algebraic stack of dimension $3g-3+n$.
\end{theorem}

Next, we recall the definition of the classical moduli stack of stable maps. Fix a smooth complex projective variety $Y$, along with integers $g,n \in \BZ_{\geqslant 0}$ and a homology class $\beta \in \HH_2(Y,\BZ)$. The moduli stack 
\[
\Mbar
\]
of \emph{stable maps} in class $\beta$ is defined as follows. Recall that, if $(C,p_1,\ldots,p_n)$ is an $n$-pointed prestable curve, a regular map $f \colon (C,p_1,\ldots,p_n) \to Y$ is \emph{stable} if whenever $f$ contracts a component $E \subset C$, such component contains at least 3 special points if it is rational and at least 1 special point if it has genus 1. Here, a point $p \in C$ is said to be \emph{special} if it is either a singular point of $C$ or a marked point $p_i$ for some $i$. By definition, the objects of the groupoid $\Mbar(B)$, for a $\BC$-scheme $B$, are the tuples
\[
(\pi\colon \CC \to B, \sigma_1,\ldots,\sigma_n,f\colon \CC \to Y),
\]
where
\begin{itemize}
\item [$\circ$] $(\pi\colon \CC \to B,\sigma_1,\ldots,\sigma_n)$ is a family of $n$-pointed prestable curves of genus $g$ as in \eqref{family-prestable-curves}, and
\item [$\circ$] $f\colon \CC \to Y$ is a morphism such that, for every geometric point $b$ of $B$, the induced map $f_b\colon (\CC_b,\sigma_1(b),\ldots,\sigma_n(b)) \to Y$ is a stable map such that $f_{b\ast}[\CC_b] = \beta$. 
\end{itemize}

\begin{remark}
Let $(C,p,f)$ and $(C',p',f')$ be two $\BC$-valued points of $\Mbar$. An isomorphism between them is an isomorphism between the underlying prestable curves, commuting with the stable maps $f$ and $f'$. A crucial observation is that the stability condition for a map says precisely that the automorphism group of the $\BC$-valued point $(C,p,f)$ of $\Mbar$ is finite. In particular, there is no genus 1 unpointed stable map hitting $\beta=0$, although this is a bit of a pathology, as the  condition on contracted elliptic components is always satisfied for $(g, n, \beta) \neq (1, 0, 0)$.
\end{remark}

\begin{example}
\label{example:classical-stable-maps}
We have the following main examples of the moduli stack of stable maps.
\begin{itemize}
\item [\mylabel{SM-1}{(1)}] If $2g-2+n>0$, we obtain $\overline{\CM}_{g,n}(Y,0) = \MB \times Y$, where $\MB$ is the classical Deligne--Mumford moduli space of $n$-pointed stable curves of genus $g$.
\item [\mylabel{SM-2}{(2)}] We have $\overline{\CM}_{0,0}(\BP^r,1) = \Gr(1,r)$, the classical Grassmannian. The map 
\[
\begin{tikzcd}
\overline{\CM}_{0,n}(\BP^r,1) \arrow{r} & \overline{\CM}_{0,0}(\BP^r,1)
\end{tikzcd}
\]
forgetting all markings is a Zariski locally trivial fibration with fiber $\BP^1[n]$, the Fulton–MacPherson configuration space.
\item [\mylabel{SM-3}{(3)}] $\overline{\CM}_{0,0}(\BP^2,2)$ is the moduli stack of \emph{complete conics}, whose coarse moduli space if the classical moduli scheme of complete conics, isomorphic to the blow-up of $\BP^5$ along the Veronese surface of double lines.
\end{itemize}
\end{example}

There is a canonical morphism
\[
\begin{tikzcd}
\Mbar \arrow{r}{\tau} & \FM_{g,n},
\end{tikzcd}
\]
forgetting the stable map. This morphism is used to construct the perfect obstruction on $\Mbar$ which gives rise, via integration against the induced virtual fundamental class, to the \emph{Gromov--Witten invariants} of $Y$. We briefly recall the construction, following \cite{Behrend_GW,Ricolfi-Book}.

Consider the universal structure
\[
\begin{tikzcd}
\CC\arrow{d}{\pi}\arrow{r}{f} & Y \\
\Mbar &
\end{tikzcd}
\]
and the transitivity triangle
\[
\begin{tikzcd}
\tau^\ast \BL_{\FM_{g,n}} \arrow{r} & \BL_{\Mbar}\arrow{r} & \BL_\tau \arrow{r}{\alpha} & \tau^\ast \BL_{\FM_{g,n}}[1]
\end{tikzcd}
\]
attached to the forgetful morphism $\tau\colon \Mbar \to \FM_{g,n}$. One can prove that there is a relative obstruction theory 
\[
\begin{tikzcd}
\BE_\tau = (\RR\pi_\ast f^\ast T_Y)^\vee \arrow{r}{\phi_\tau} & \BL_\tau.
\end{tikzcd}
\]
This is proved by embedding $\Mbar$ as an open substack of the $\mathfrak M_{g,n}$-relative stack of morphisms into $Y$, and restricting the natural obstruction theory there; of course, $\phi_\tau$ can be seen as a globalisation of the classical fact that tangents and obstructions for the deformation problem of a map $C \to Y$ with fixed domain are given by $\HH^i(C,f^\ast T_Y)$, for $i=0$ and $i=1$ respectively. Now, define
\[
\begin{tikzcd}
\BE \defeq \mathrm{Cone}(\alpha \circ \phi_\tau \colon \BE_\tau \arrow{r} & \tau^\ast \BL_{\FM_{g,n}}[1])[-1].
\end{tikzcd}
\]
Using the smoothness of $\mathfrak M_{g,n}$ and the perfectness of $\BE_\tau$ in $[-1,0]$, one deduces that in the morphism of exact triangles
\[
\begin{tikzcd}
\tau^\ast \BL_{\FM_{g,n}}\arrow[equal]{d} \arrow{r} & \BE \arrow{r}\arrow[dotted]{d}{\phi} & \BE_\tau \arrow{r}\arrow{d}{\phi_\tau} & \tau^\ast \BL_{\FM_{g,n}}[1]\arrow[equal]{d} \\
\tau^\ast \BL_{\FM_{g,n}} \arrow{r} & \BL_{\Mbar}\arrow{r} & \BL_\tau \arrow{r}{\alpha} & \tau^\ast \BL_{\FM_{g,n}}[1]
\end{tikzcd}
\]
the\footnote{There might be more than one, but all possible arrows $\BE \to \BL_{\Mbar}$ are isomorphisms on $h^0$ and surjections on $h^{-1}$ by a standard diagram chasing argument.} induced morphism $\phi$ is an absolute perfect obstruction theory. The induced virtual fundamental class lives in Chow-dimension
\[
\vd = \rk \BE = (\dim Y-3)(1-g) - \int_\beta K_Y + n.
\]
The adaptation of this construction to the super setting is not quite straightforward; it will be performed with full details in \Cref{sec:stable-supermaps}. 

\section{Basic theory of superstacks}
\label{sec:superstacks}
\subsection{General theory of superstacks}

We recall the basics of the theory of superstacks, mainly taken from \cite{BHRSS25}.
Superstacks had already been defined for the \'etale site of smooth superschemes in \cite{CodViv17}. However, since the definition of stack over a site is pretty general (see, for instance, \cite{CasWis17}), we can just use the general theory applied to the site of superschemes with respect to the \'etale topology. A more involved question is related to the notion of algebracity of superstacks. 

We will denote by $\SSchs$ the category of superschemes.  We refer to \cite{BrHRMa23,BrHRPo20} for all the definitions of the different types of morphisms of superschemes. For instance, smooth and \'etale morphisms of superschemes are defined in  \cite[Def.~2.7 and 2.13]{BrHRMa23} (see also \cite{BrHRPo20}); equivalent definitions are provided in \cite{BHRSS25}. 
The \'etale (pre)site of superschemes for the \'etale topology is given in \cite[Def.~2.14]{BrHRMa23}. We denote this site by $\SSchs_{\et}$. Notice that $\SSchs_{\et}$ is \emph{subcanonical} \cite[Def.~3.12]{CasWis17}, that is, every functor  $\SSchs^{\opp}\to\Sets$ that is representable by a superscheme, is an \'etale sheaf.

A pre-superstack (resp.~a superstack) is a prestack (resp.~a stack) over $\SSchs_{\et}$. This can be spelled out as follows.
 
\begin{definition} 
A \emph{pre-superstack} (resp.~a \emph{superstack}) is a category fibered in groupoids 
\[
\begin{tikzcd}
p_\BCal{X}\colon \BCal{X} \arrow{r} & \SSchs_{\et} 
\end{tikzcd}
\]
over the site $\SSchs_{\et}$ such that the isomorphisms between two objects in $\BCal{X}(\dS)\defeq p_\BCal{X}^{-1}(\dS)$ form a sheaf for every superscheme $\dS$ (resp.~ditto and that every descent data is effective). A morphism of pre-superstacks or superstacks is a morphism of categories fibered in groupoids over $\SSchs_{\et}$. If $\dS$ is a superscheme, changing $\SSchs$ to the category $\SSchs/\dS$ of superschemes over $\dS$, we obtain the notions of pre-superstack and superstack over $\dS$. 
\end{definition}

The first example of a superstack is the one associated to a sheaf on $\SSchs_{\et}$. To see this, consider a presheaf of  $\CF\colon \SSchs^{\opp}\to\Sets$. We define a category fibered in groupoids $p_\BCal{F}\colon \BCal{F} \to \SSchs$ by setting $\BCal{F}(\dS)$ as the set $\CF(\dS)$ considered as a category; morphisms in $\BCal{F}$ are defined in the obvious way. Notice that $p_\BCal{F}\colon\BCal{F}\to\SSchs$ is a category fibered in setoids, namely, groupoids whose objects have no other morphisms than the identity. If $\BCal{X}$ is a superstack equivalent to the superstack $\BCal{F}$ associated to a sheaf $\CF$ we say that $\BCal{X}$ is \emph{representable} by a sheaf, or simply, that it is a sheaf.

Since for every superscheme $\dS$ the functor of points is an \'etale sheaf, it gives rise to
 a superstack. If $\BCal{X}$ is a superstack equivalent to the superstack $\BCal{S}$ associated to a superscheme $\dS$, we say that \emph{$\BCal{X}$ is representable by the superscheme $\dS$} or simply that it is a superscheme. One easily sees that, as a fibered category, the superstack $\BCal{S}$ is equivalent to the category fibered in groupoids $\SSchs/\dS\to \SSchs$.

There is a version of the Yoneda Lemma for superstacks \cite[Lemma 2.4]{BHRSS25}. To formulate it, let $p_\BCal{X}\colon \BCal{X}\to\SSchs$, $p_\BCal{Y}\colon \BCal{Y}\to\SSchs$ be superstacks. We denote by $\Homst_{\SSt}(\BCal{X},\BCal{Y})$ the fibered category whose objects are the morphisms $\BCal{X} \to \BCal{Y}$ of superstacks and whose morphisms are the natural isomorphisms. 

\begin{lemma}[2-Yoneda] Let $p_\BCal{X}\colon\BCal{X}\to\SSchs$ be a superstack. For every superscheme $\dS$ the natural transformation
\[
\begin{tikzcd}[row sep=tiny]
\Homst_{\SSt}(\SSchs/\dS,\BCal{X}) \arrow{r} & \BCal{X}(\dS) \\
F \arrow[mapsto]{r} & F(\dS\xrightarrow{\Id_{\dS}}\dS)
\end{tikzcd}
\]
with the similar definition for morphisms, is an equivalence of categories.
\qed
\end{lemma}

By virtue of the 2-Yoneda Lemma, specifying a superstack morphism
\[
\begin{tikzcd}
x\colon \SSchs/\dS \arrow{r} & \BCal{X}
\end{tikzcd}
\]
is the same as specifying an object $x\in \BCal{X}(\dS)$. We will make this identification without further mention.

The fiber product of two superstacks is its fiber product as fibered categories; it is also a superstack. This is a particular case of \cite[Def. 4.2 and Lemma 4.4]{CasWis17}. 

A morphism $f\colon \BCal{X} \to \BCal{Y}$ of superstacks is \emph{superschematic} (resp.~\emph{representable by sheaves}), if for every superscheme $\dS$ and every superstack morphism $g\colon\dS\to \BCal{Y}$ the fiber product superstack $\dS\times_{g,\BCal{Y} ,f} \BCal{X}$ is (representable by) a superscheme (resp.\! a sheaf). One has that a morphism $f\colon \BCal{X} \to \BCal{Y}$ of superstacks is \! representable by sheaves, if for every sheaf $\CF$ and every superstack morphism $g\colon\CF\to \BCal{Y}$ the fiber product superstack $\CF\times_{g,\BCal{Y} ,f} \BCal{X}$ is (representable by) a sheaf. 

The diagonal morphism $\Delta_{\BCal{X}} \colon \BCal{X}\to \BCal{X}\times\BCal{X}$ of a superstack $\BCal{X}$ has the important property that it codifies the isomorphisms between the objects of the category $\BCal{X}$, namely, for every superscheme $\dS$ and objects $x_1$, $x_2$ in $\BCal{X}(\dS)$, if we consider them as morphisms $x_1\colon \dS\to \BCal{X}$, $x_2\colon \dS \to \BCal{X}$ by the 2-Yoneda Lemma, there is a 2-cartesian diagram
\begin{equation}
\xymatrix{ 
\mathfrak{Isom}_{\BCal{X}(\dS)}(x_1,x_2) \ar[r]\ar[d] & \dS \ar[d]^{(x_1,x_2)} \\
\BCal{X}\ar[r]^{\Delta_{\BCal{X}}}& \BCal{X}\times\BCal{X}
}
\end{equation}

More generally, if $f\colon \BCal{X}\to \BCal{Y}$ is a superstack morphism, and $(x_1,x_2,\alpha)$ (where $\alpha\colon f(x_1) \simto f(x_2)$ is an isomorphism in $\BCal{Y}(\dS)$) is an object of $(\BCal{X}\times_{\BCal{Y}} \BCal{X})(\dS)$,  we have a 2-cartesian diagram
\begin{equation}
\xymatrix{ 
\mathfrak{Isom}_{\BCal{X}(\dS),\alpha}(x_1,x_2) \ar[r]\ar[d] & \dS \ar[d]^{(x_1,x_2,\alpha)} \\
\BCal{X}\ar[r]^{\Delta_f}& \BCal{X}\times_{\BCal{Y}}\BCal{X}
}
\end{equation}
where $\mathfrak{Isom}_{\BCal{X}(\dS),\alpha}(x_1,x_2)$ is the pre-image of $\alpha$ under the natural map 
\[
\begin{tikzcd}
f_\ast \colon \mathfrak{Isom}_{\BCal{X}(\dS)}(x_1,x_2)\arrow{r} & \mathfrak{Isom}_{\BCal{Y}(\dS)}(f(x_1),f(x_2)).
\end{tikzcd}
\]
The injectivity of the diagonal characterizes sheaves, that is, a superstack morphism $f\colon \BCal{X}\to \BCal{Y}$ is representable by sheaves if and only if the diagonal $\Delta_f$ is a monomorphism.

A  subsuperstack $\BCal{X}\subseteq \BCal{Y}$ is a monomorphism $\BCal{X}\hookrightarrow \BCal{Y}$ of superstacks. A subsuperstack $\BCal{X}\subseteq \BCal{Y}$ is open (resp.\! closed), if for every superscheme $\dS$ and every superstack morphism $\dS\to \BCal{Y}$, the superstack $\dS\times_{\BCal{Y}}\BCal{X}$ is representable by an open (resp.\! closed) superscheme of $\dS$.

\subsection{Algebraic superspaces. \DM\ and algebraic superstacks}

Though algebraic superspaces are particular instances of algebraic stacks, for operative reasons it is better to define them first.

If we have a property $\mathbf P$ of morphisms of superschemes that is stable under base change (e.g., surjective, proper, smooth, \'etale, unramified, formally smooth, formally \'etale, formally unramified), we can say that a superschematic superstack morphism $\BCal{X} \to \BCal{Y}$ has property $\mathbf P$ if for every morphism $\TT\to\BCal{Y}$ from a superscheme, the base-change morphism $\BCal{X}\times_{\BCal{Y}}\TT \to \TT$ of superschemes has property $\mathbf P$.

\begin{definition}
An \emph{algebraic superspace} is a sheaf $\XX$ such that there exist a superscheme $\UU$ and a superschematic surjective \'etale  morphism $p\colon \UU\to\XX$. A morphism of algebraic superspaces is a morphism of superspaces.
We say that $p\colon \UU\to\XX$ is an \emph{\'etale presentation} of $\XX$.
\end{definition}

We now say that a morphism $f\colon\BCal{X}\to\BCal{Y}$ of superstacks is \emph{representable} if for every morphism $\TT\to\BCal{Y}$ from a superscheme $\TT$, the fiber product superstack $\BCal{X}\times_\BCal{Y} \TT$ is an algebraic superspace. 

One easily sees that a superschematic morphism is representable and that any representable morphism, it is also representable by algebraic superspaces, namely,  for every morphism $\BZ\to\BCal{Y}$ from an algebraic superspace $\BZ$, the fiber product $\BCal{X}\times_{\BCal{Y}} \BZ$ is an algebraic superspace. 

For the properties of morphisms of superschemes that are \emph{local on the source} for the \'etale topology, we can make sense of the analogous properties for representable morphisms of superstacks. Namely, let  $\mathbf P$ be a property of morphisms of superschemes that is stable under base change and  local to the source for the \'etale topology. Then, a \emph{representable  morphism $\BCal{X}\to\BCal{Y}$ of superstacks has $\mathbf P$} if for  every morphism $\TT\to\BCal{Y}$ from a superscheme $\TT$ and \'etale presentation $\UU \to \BCal{X}\times_\BCal{Y} \TT$ by a superscheme, the composition morphism of superschemes  $\UU \to \BCal{X}\times_\BCal{Y} \TT \to \TT$ has property $\mathbf P$.

We can now give the more relevant geometric types of superstacks:

\begin{definition} 
A \DM\  (resp.\! algebraic) superstack is a superstack $\BCal{X}$ such that there exists a superscheme $\UU$ and a representable surjective \'etale (resp.\! smooth) morphism $p\colon\UU\to \BCal{X}$. A morphism of \DM\  or algebraic superstacks is a morphism of superstacks.
The morphism $p\colon\UU\to \BCal{X}$  is called an \'etale (resp.\! smooth) presentation. 
\end{definition}

One  has the following results whose proof is straightforward:

\begin{prop}
Superschemes and  algebraic superspaces are Deligne--Mumford superstacks.  De\-ligne--Mumford  superstacks are algebraic superstacks.
Open and closed substacks of an algebraic superspace (resp.\! a \DM\  superstack, resp.\! an algebraic superstack) are algebraic superspaces (resp.\! \DM\  superstacks, resp.\! algebraic superstacks). Fiber products exist for algebraic superspaces, \DM\  superstacks and algebraic superstacks.\qed
\end{prop}

It is clear that algebraic superspaces are algebraic stacks and sheaves. It can be proved \cite[Prop.\! 6.5]{BHRSS25} that an algebraic stack that is a sheaf is actually an algebraic superspace.

As for algebraic stacks, one has \cite[Prop.~3.27 and Cor.~3.28]{BHRSS25}:

\begin{prop}[Representability of the diagonal] 
The following statements hold true.
\begin{enumerate}
\item [\mylabel{ss1}{\normalfont{(1)}}]
The diagonal of an algebraic superspace is superschematic;
\item [\mylabel{ss2}{\normalfont{(2)}}]
the diagonal of an algebraic superstack is representable;
\item [\mylabel{ss3}{\normalfont{(3)}}]
the diagonal $\Delta_f\colon \BCal{X} \to \BCal{X}\times_\BCal{Y} \BCal{X}$ of a representable morphism $f\colon\BCal{X} \to\BCal{Y}$ of algebraic superstacks is superschematic.
\end{enumerate}
\end{prop}

Let $k$ be a field and  $x\colon \Spec k \to \BCal{X}$ a $k$-valued point of an algebraic superstack $\BCal{X}$. The stabilizer of $x$ is the \'etale sheaf 
\[
\mathbf{Aut}_{\BCal{X}(k)}(x)\defeq \mathfrak{Isom}_{\BCal{X}(\Spec k)}(x,x).
\]
Then, there is a 2-cartesian diagram
\begin{equation}
\xymatrix{ 
\mathbf{Aut}_{\BCal{X}(k)}(x) \ar[r]\ar[d] & \Spec k \ar[d]^{(x,x)} \\
\BCal{X}\ar[r]^{\Delta_\BCal{X}}& \BCal{X}\times\BCal{X}
}
\end{equation}
so that $\mathbf{Aut}_{\BCal{X}(k)}(x)$ is a \emph{group-algebraic superspace} by the representability of the diagonal. 
Also, since the diagonal of an algebraic superspace is superschematic, the stabilizers of the field-valued points of an algebraic superspace are actually group-superschemes or \emph{supergroups}.

The characterization of \DM\ superstacks is similar to the one for stacks:

\begin{prop}[{\cite[Prop.~6.3]{BHRSS25}}] Let $\BCal{X}$ be an algebraic superstack.
\begin{enumerate}
\item [\mylabel{diag}{\normalfont{(1)}}]
The diagonal $\Delta_\BCal{X}\colon \BCal{X} \to \BCal{X}\times\BCal{X}$ is of finite type,
\item [\mylabel{DM1}{\normalfont{(2)}}]
$\BCal{X}$ is a \DM\ superstack if and only if the diagonal $\Delta_\BCal{X}$ is unramified,
\item [\mylabel{DM2}{\normalfont{(3)}}] $\BCal{X}$ is a \DM\ superstack if and only if for every point $x$, the stabilizer $\mathbf{Aut}_{\BCal{X}(k)}(x)$ is a discrete and reduced ordinary group scheme.
\end{enumerate}
\end{prop}

Notice that, by the representability of the diagonal, it makes sense to say that it is unramified. Notice also that the stabilizers of the points are (ordinary) group schemes. This is due to the fact that every formally unramified morphism of superschemes is actually bosonic, that is, its fibers are ordinary schemes.

\subsection{Bosonic reduction of a superstack}
In the same way that any superscheme has a bosonic reduction, which is the underlying ordinary scheme, every superstack has a natural underlying (ordinary) stack, which we will call its \emph{bosonic reduction}; moreover, every stack can be understood as a superstack (see \cite{CodViv17}).

The \'etale site $\Schs_{\et}$ of ordinary schemes (i.e. even superschemes) is a subsite of $\SSchs_{\et}$. There is a functor $i\colon \Schs_{\et} \hookrightarrow \SSchs_{\et}$ defined by associating with every scheme $S$ the scheme itself $i(S)$ considered as a superscheme, that is actually a superstack.

\begin{definition}\label{def:bosred} The bosonic reduction of a superstack  $p_{\BCal{X}}\colon \BCal{X} \to \SSchs_{\et}$ is the fiber product of categories fibered in groupoids
\[
\begin{tikzcd}
\BCal{X}_{\bos} \defeq \BCal{X}\times_{p_{\BCal{X}},\SSchs_{\et},i}\Schs_{\et}\arrow{rr}{p_{\BCal{X}_{\bos}}} && \Schs_{\et}.
\end{tikzcd}
\]
For every scheme $S$, there is an equivalence of categories $\BCal{X}_{\bos}(S)\simeq \BCal{X}(S)$ and $\BCal{X}_{\bos}$ is a stack.
\end{definition}

In the same way that a scheme can be considered as a (even) superscheme, every stack can be consider as a superstack. We can formalize this by first considering the functor $p_+\colon\SSchs_{\et}\to \Schs_{\et}$ that maps a superscheme $\dS$ to the even superscheme $\dS_{\mathrm{even}}=\dS/\tau$, where $\tau$ is the natural involution that acts as the identity on the even sections and as multiplication by $-1$ on the odd ones. The functor $p_+\colon\SSchs_{\et}\to \Schs_{\et}$ is actually a stack. This is formalized in the next definition.

\begin{definition} 
\label{def:superstack-ass-to-stack}
The superstack associated to a stack $p_{\CX}\colon\CX \to \Schs_{\et}$ is defined as the fiber product
\[
\begin{tikzcd}
i(\CX)\defeq \CX\times_{p_{\CX},\Schs_{\et},p_+} \SSchs_{\et}\arrow{r} & \SSchs_{\et}.
\end{tikzcd}
\]
\end{definition}
One verifies immediately that $i(\CX)$ is indeed a superstack.
As for schemes, we usually write also $\CX$ for the superstack $i(\CX)$.
\begin{definition}
We say that a superstack $\BCal{X}$ is \emph{bosonic} when the natural superstack morphism $\BCal{X} \to i(\BCal{X}_{\bos})$ is an isomorphism.
\end{definition}
The functors $(-)_{\bos}$ and $i$ are actually adjoint, in the sense that, for every stack $\dS$ and every superstack $\BCal{X}$, there is an equivalence of categories
\[
\Homst_{\SSt}(i(\dS), \BCal{X})\simeq \Homst_{\St}(\dS, \BCal{X}_{\bos}),
\]
see \cite[Equation 2.2]{BHRSS25}.

\section{Foundations of homological superalgebra}
\label{sec:homological_super_algebra}
In order to develop a consistent theory of cotangent supercomplexes for algebraic superstacks, which is essential for a good definition of (perfect) obstruction theories, we need to construct a suitable derived ($\infty$-)category of quasicoherent supersheaves over algebraic superstacks: our cotangent supercomplex will be then naturally defined as an object in such $\infty$-category. Notice that even if the cotangent supercomplex has never been explicitly studied in the literature, its existence is somewhat assumed in the proof of \cite[Thm.~3.3]{Felder-Kazhdan-Polishchuk}, which is itself elaborated from an argument due to Deligne \cite{Deligne-letter}. There, the definition of the cotangent supercomplex of a supercommutative ring mirrors the one by Illusie in \cite{Illusie_I,Illusie_II} in terms of (super)polynomial simplicial resolutions. However, the relevant properties relating the cotangent supercomplex and deformation theory are not discussed, and it is not even proved to be uniquely defined up to canonical quasi-isomorphism.

Here, we propose a somewhat more esoteric approach following the construction of topological André--Quillen homology for an arbitrary commutative algebra object in a presentable ($\infty$-)category $\mathscr{C}$, as presented in \cite[$\S$7.3]{ha}. This requires a more convoluted higher categorical formalism, but it has the following advantages.
\begin{enumerate}
    \item All constructions make sense for an arbitrary (derived) supercommutative ring, superscheme or superstack. Moreover, they are naturally well-defined up to quasi-isomorphism.
    \item The cotangent supercomplex is determined by local data: that is, a derived superstack $\bm{\CX}$ admits a global cotangent supercomplex provided that there exists a cotangent supercomplex at any point $\eta\colon \Spec\BA\to \bm{\CX}$, and for any map of affine derived superschemes $f\colon \Spec\BB\to\Spec\BA$ the cotangent supercomplex at the point $\eta\circ f$ is the base change of the cotangent supercomplex at the point $\eta$ along $f$.
    \item By its very definition, the cotangent supercomplex naturally classifies superderivations and controls the deformation/obstruction theory of a superscheme.
\end{enumerate}
For the sake of clarity, we start by recalling the main basic notions of supermodules over ordinary commutative rings, which readily generalize the relevant definitions from ordinary commutative algebra and classical scheme theory, and then proceed to extend the theory to the derived setting. Most of the results in this appendix follow from formal rephrasings of the arguments for the analogous statements in the non-super case (see for example \cite{ha,sag}): we still provide proofs for the reader's convenience. Proofs of longer or more technical results can be found in the appendix.
\subsection{Enhancements of categories of supermodules}
\label{sec:dg-supermodules}
Let $R$ be an ordinary commutative ring. We recall that a \textit{$R$-supermodule} is just a $\BZ/2\BZ$-graded $R$-module
\[
\BM\defeq\BM^{\mathrm{even}}\oplus\BM^{\mathrm{odd}}.
\]
$R$-supermodules are naturally gathered in the category $\mathrm{SMod}_R$, where a morphism $\BM\to\BN$ is required to satisfy the parity of the grading. Via the projections $\BM\mapsto(\BM^{\mathrm{even}},\BM^{\mathrm{odd}})$, the category $\mathrm{SMod}_R$ is readily seen to decompose as
\[
\mathrm{SMod}_R\simeq\mathrm{Mod}_R\times\mathrm{Mod}_R.
\]
This presentation allows us to deduce straightforwardly some structural properties of the category of $R$-supermodules.
\begin{fact}
\label{fact:super_modules_good_abelian} The category   $\mathrm{SMod}_R$ enjoys the following properties: 
  \begin{factenum}
    \item\label{fact:SMod_lim_colim}  it  admits all small limits and colimits, and they are computed degree-wise;
    \item\label{fact:SMod_abelian} it is abelian; 
    \item\label{fact:SMod_presentable} it is locally presentable. Let
    \[
    \begin{tikzcd}
\Pi\colon\mathrm{SMod}_R\arrow{r} & \mathrm{SMod}_R
    \end{tikzcd}
    \]
    denote the parity exchange functor (sending a $\mathbb{Z}/2\mathbb{Z}$-graded $R$-module $(M,N)$ to $(N,M)$). Then the even and odd lines $R\defeq(R,0)$ and $\Pi R$ are compact generators;
    \item\label{fact:SMod_Grothendieck}   in particular, it is Grothendieck. 
    \item\label{fact:SMod_enough_inj_proj} it  has enough projectives and injectives. 
\end{factenum}  
\end{fact}
This, however, is only part of the story. The actual difference between ordinary $\BZ/2\BZ$-graded $R$-modules and $R$-supermodules is provided by the \textit{braiding} of the symmetric monoidal structure. Namely, let $\Mod^{\BZ/2\BZ}_R$ denote the category of $\BZ/2\BZ$-graded $R$-modules: then we have a natural symmetric monoidal structure described as
\[
\BM\otimes^{\BZ/2\BZ}_R\BN\defeq \left((\BM^{\mathrm{even}}\otimes_R\BN^{\mathrm{even}})\oplus(\BM^{\mathrm{odd}}\otimes_R\BN^{\mathrm{odd}})\right)\bigoplus\left((\BM^{\mathrm{even}}\otimes_R\BN^{\mathrm{odd}})\oplus(\BM^{\mathrm{odd}}\otimes_R\BN^{\mathrm{even}})\right).
\]
The \textit{braiding} is a natural isomorphism $\tau\colon \BM\otimes^{\BZ/2\BZ}_R\BN\cong\BN\otimes^{\BZ/2\BZ}_R\BM$, satisfying $\tau^2\simeq\mathrm{id}$ and some additional coherence properties, which guarantee the commutativity of all possible diagrams where factors in a tensor product are swapped. Such braiding is naturally described as
\[
\begin{tikzcd}[row sep=tiny]
\BM\otimes^{\BZ/2\BZ}_R\BN \arrow{r}{\sim} & \BN\otimes^{\BZ/2\BZ}_R\BM\\
m\otimes n \arrow[mapsto]{r} & n\otimes m.
\end{tikzcd}
\]
When the characteristic of $R$ is different from $2$, the same tensor product (which we shall now denote as $\otimes^{\mathrm{super}}_R$) equips the category $\mathrm{SMod}_R$ with another symmetric monoidal structure, and as plain monoidal categories we have an equivalence $\Mod^{\BZ/2\BZ}_R\simeq\mathrm{SMod}_R$. Yet, $\mathrm{SMod}_R$ comes equipped with a different braiding which involves a Koszul sign. Namely, let $p_m$ denote the parity of an element $m$ in a $R$-supermodule $\BM$. Then the braiding on $\mathrm{SMod}_R$ is described as
\[
\begin{tikzcd}[row sep=tiny]
\BM\otimes^{\mathrm{super}}_R\BN
\arrow{r}{\sim} & 
\BN\otimes^{\mathrm{super}}_R\BM\\
m\otimes n \arrow[mapsto]{r} & (-1)^{p_m\cdot p_n} n\otimes m.
\end{tikzcd}
\]
As \textit{symmetric} monoidal categories, it is no longer true that $\Mod^{\BZ/2\BZ}_R\simeq\mathrm{SMod}_R$. We call a commutative algebra object $\BA$ in $\mathrm{SMod}_R$ a \textit{supercommutative $R$-algebra} (or \textit{supercommutative ring} if $R$ is assumed to be the ring of integers $\BZ$): we say that a $R$-supermodule $\BM$ is a \textit{$\BA$-supermodule} over a supercommutative ring $\BA$ if it comes equipped with a map
\[
\begin{tikzcd}
\alpha\colon\BM\otimes^{\mathrm{super}}_R\BA\arrow{r} &  \BM
\end{tikzcd}
\]
satisfying the obvious diagrammatic axioms.
\begin{fact}
Fix an ordinary commutative ring $R$.
\begin{factenum}
\item \label{fact:internal_hom}The supertensor product $\otimes^{\mathrm{super}}_R$ commutes with arbitrary colimits. In fact, for any $R$-supermodule $\BM$ the functor
\[
\begin{tikzcd}
(-)\otimes^{\mathrm{super}}_R\BM\colon\mathrm{SMod}_R\arrow{r} & \mathrm{SMod}_R
\end{tikzcd}
\]
admits a right adjoint
\[
\begin{tikzcd}
\BH\mathrm{om}_R(\BM,-) \colon\mathrm{SMod}_R\arrow{r} & \mathrm{SMod}_R.
\end{tikzcd}
\]
Given an arbitrary $R$-supermodule $\BN$, the $R$-supermodule $\BH\mathrm{om}_R(\BM,\BN)$ is the $R$-supermodule of homomorphisms between $\BM$ and $\BN$, and it is described by the formula
\begin{align*}
\BH\mathrm{om}(\BM,\BN)^{\mathrm{even}}&\defeq\Hom_R(\BM^{\mathrm{even}},\BN^{\mathrm{even}})\oplus\Hom_R(\BM^{\mathrm{odd}},\BN^{\mathrm{odd}}),\\
\BH\mathrm{om}(\BM,\BN)^{\mathrm{odd}}&\defeq\Hom_R(\BM^{\mathrm{even}},\BN^{\mathrm{odd}})\oplus\Hom_R(\BM^{\mathrm{odd}},\BN^{\mathrm{even}}).
\end{align*}
\item Since the supertensor product commutes with colimits, it follows that for any supercommutative $R$-algebra $\BA$ limits and colimits in the category $\mathrm{SMod}_{\BA}$ are computed at the level of the underlying $R$-supermodules: one can see it by computing explicit left and right adjoints to the forgetful functor $\mathrm{SMod}_{\BA}\to\mathrm{SMod}_R$ given respectively by $(-)\otimes^{\mathrm{super}}_R\BA$ and $\BH\mathrm{om}_R(\BA,-)$.
\item For any supercommutative $R$-algebra $\BA$, the category $\mathrm{SMod}_{\BA}$ is moreover locally presentable: once again, the objects $\BA$ and $\Pi\BA\defeq \BA\otimes^{\mathrm{super}}_R\Pi R$ are compact generators.
\item In particular, $\mathrm{SMod}_{\BA}$ admits both enough injectives (being a Grothendieck abelian category) and projectives. The latter can be deduced from the fact that $\mathrm{SMod}_{\BA}$ can be realized as the category of $\mathrm{SMod}_R$-enriched functors from the $\mathrm{SMod}_R$-enriched category $\mathbf{B}\BA$, having only one object $\left\{\ast\right\}$ with automorphisms given by $\BA$, to $\mathrm{SMod}_R$. 
\end{factenum}
\end{fact}
In order to define a derived stable $\infty$-category of supermodules, we start by producing a stable/ differential graded enhancement of the derived category of $\BA$-supermodules for an arbitrary (supercommutative) super-ring $\BA$. This is obtained by putting a \textit{model structure} on the category of cochain complexes of $\BA$-supermodules, which amounts to the datum of a distinguished class of morphisms that we want to formally invert, and the datum of distinguished classes of morphisms and objects that are well-behaved from a homotopical perspective. We will use the notion of model categories as a black box, so it is not really important to dive into deeper detail: the curious reader can read the relevant definitions in \cite{hovey}.

We start by establishing the existence of the \textit{projective} model structure on the category of chain complexes of $R$-supermodules. This can be easily deduced from the existence of the relative model structure on the category of cochain complexes of ordinary $R$-modules (see for example \cite[Prop.~7.1.2.8]{ha}) together with \cref{fact:super_modules_good_abelian}, so we leave the proof to the reader.
\begin{lemma}
\label{lemma:projective_model_structure}
Let $R$ be an ordinary commutative ring, and let $\mathrm{SMod}_R$ denote the Grothendieck abelian category of $R$-supermodules. Then the abelian category $\mathrm{Ch}(\mathrm{SMod}_R)$ of cochain complexes of $R$-supermodules admits a left proper combinatorial model structure, called the \emph{projective model structure}, which can be described as follows. 
\begin{lemenum}
    \item\label{lemma:projective_model_structure_fib} A map $\BM^{\bullet}\to\BN^{\bullet}$ is a fibration if and only if for all $k\in\BZ$ the map $\BM^k\to\BN^k$ is a surjective map of $R$-supermodules.
    \item\label{lemma:projective_model_structure_weak_eq}A map $\BM^{\bullet}\to\BN^{\bullet}$ is a weak equivalence if and only if it is a quasi-isomorphism, i.e., if and only if the map $\BH^k(\BM^{\bullet})\to\BH^k(\BN^{\bullet})$ is an isomorphism for all integers $k\in\BZ$.
    \item\label{lemma:projective_model_structure_cofib} A map $\BM^{\bullet}\to\BN^{\bullet}$ is a cofibration if and only if it has the right lifting property with respect to all maps which are simultaneously a fibration and a weak equivalence.
\end{lemenum}
\end{lemma}

\begin{remark}
\label{remark:projective_fibrant_cofibrant}
We recall that an object $X$ in a model category is \textit{cofibrant} (respectively \textit{fibrant}) if the map from the initial object $0\to X$ is a cofibration (respectively if the map to the terminal object $X\to1$ is a fibration). In particular, it follows immediately from \cref{lemma:projective_model_structure_fib} that every cochain complex of $R$-supermodules is fibrant with respect to the projective model structure. However, cofibrant objects are more difficult to characterize. For example, while it is true that any fibrant cochain complex necessarily consists of projective $R$-supermodules in every cohomological degree, the converse implication holds only when a cochain complex is cohomologically bounded above \cite[Prop.~$1.3.5.6$]{ha}.
\end{remark}
\begin{construction}
\label{constr:dg-nerve}
Let $R$ be a commutative ring, and let $\mathrm{Ch}(\mathrm{SMod}_R)$ be the model category of \cref{lemma:projective_model_structure}. We can extract a differential graded category $\mathrm{Ch}_{\mathrm{dg}}(\mathrm{SMod}_R)$, which we can describe as follows.
\begin{enumerate}
    \item The objects of $\mathrm{Ch}_{\mathrm{dg}}(\mathrm{SMod}_R)$ are fibrant--cofibrant objects in $\mathrm{Ch}(\mathrm{SMod}_R)$.
    \item The hom-complexes $\smash{\underline{\Hom}}(\BM^{\bullet},\BN^{\bullet})$ are described in degree $n$ by the formula
    \[
    \smash{\underline{\Hom}}(\BM^{\bullet},\BN^{\bullet})^n\defeq\prod_{k\in\BZ}\Hom_{\mathrm{SMod}_R}(\BM^k,\BN^{n+k}).
    \]
    The differential on a element $f$ in degree $k$ is described by the assignment $f\mapsto d\circ f-(-1)^kf\circ d$.
\end{enumerate}
\end{construction}
\begin{remark}
There is a natural notion of homotopy category of a differential graded category: this is the category having the same class of objects, but whose hom-set $\Hom(X,Y)$ is given by the $0$-th cohomology of $\smash{\underline{\Hom}(X,Y)}$. In the case of $\mathrm{Ch}_{\mathrm{dg}}(\mathrm{SMod}_R)$, for two cochain complexes of $R$-supermodules $\BM^{\bullet}$ and $\BN^{\bullet}$ the fact that a map $f\in\smash{\underline{\Hom}}(\BM^{\bullet},\BN^{\bullet})^0$ is a cocycle means that $f$ is a honest map of cochain complexes; the fact that it is a coboundary means that it is null-homotopic. In particular, \[
\BH^0(\smash{\underline{\Hom}}(\BM^{\bullet},\BN^{\bullet}))\cong\faktor{\Hom_{\mathrm{Ch}(\mathrm{SMod_R})}(\BM^{\bullet},\BN^{\bullet})}{\simeq_{\mathrm{homotopy}}}.
\] 
Recall that fibrant and cofibrant objects in a model category are precisely the ones for which the Whitehead theorem applies: this means that any map between fibrant--cofibrant objects is a weak equivalence if and only if it admits a homotopy inverse. This implies that the hom-set $\BH^0(\smash{\underline{\Hom}}(\BM^{\bullet},\BN^{\bullet}))$ is precisely the hom-set between $\BM^{\bullet}$ and $\BN^{\bullet}$ in the ordinary derived category $\derived(\mathrm{SMod}_R)$ defined via usual homological algebra methods. In particular, the homotopy category of $\mathrm{Ch}_{\mathrm{dg}}(\mathrm{SMod}_R)$ is naturally identified with $\derived(\mathrm{SMod}_R)$.

As proved in \cite{CNS}, \cref{constr:dg-nerve} is essentially the unique way to produce a \textit{differential graded enhancement} of $\derived(\mathrm{SMod}_R)$---that is, to build a pretriangulated differential graded category $\mathrm{Ch}_{\mathrm{dg}}(\mathrm{SMod}_R)$ whose homotopy category is equivalent to $\derived(\mathrm{SMod}_R)$. Applying the correspondence between pretriangulated differential graded categories and $R$-linear stable $\infty$-categories (see \cite{Cohn-DG-cat}) we obtain the \textit{derived $\infty$-category of $R$-supermodules}. For simplicity, we shall call this $\infty$-category the \textit{superderived category of $R$}, and denote it as $\mathbf{SD}_{\infty}(R)$.
\end{remark}
\begin{fact}
We collect here some important properties of the superderived $\infty$-category $\mathbf{SD}_\infty(R)$.
\begin{factenum}
\item\label{fact:presentable} \cite[Prop.~1.3.4.22]{ha} The superderived $\infty$-category of $R$ is presentable: this means that it admits all limits and colimits, and it is generated under colimits by two compact objects (namely: $R$ and $\Pi R$).
\item\label{fact:stable} \cite[Prop.~1.3.5.9]{ha} As already mentioned, the superderived $\infty$-category of a commutative ring $R$ is stable: it has a zero object and a square in $\mathbf{SD}_\infty(R)$ is a pullback if and only if it is a pushout. This implies in particular that the homotopy category is triangulated, with the exact triangles $\BM^{\bullet}\to\BN^{\bullet}\to\BP^{\bullet}$ in the homotopy category arising from the images of pullback/pushout diagrams of the form
\[
 \xymatrix{
\BM^{\bullet} \ar[r] \ar[d] & \BN^{\bullet}\ar[d]
 \\
 0\ar[r] & \BP^{\bullet}.
 }
\]
\end{factenum}
\end{fact}
\subsection{Derived supercommutative algebras and supermodules}
Our goal is to equip the stable derived $\infty$-category $\mathbf{SD}_{\infty}(R)$ with a symmetric monoidal structure, in order to consider derived supercommutative algebras and derived modules over them. Morally, such tensor product should be obtained by left-deriving the bifunctor $\otimes^{\mathrm{super}}_R$, which is right exact separately in each variable. We first define a symmetric monoidal structure on the category of cochain complexes of $R$-supermodules as follows.
\begin{enumerate}
    \item The underlying monoidal structure intertwines the K\"unneth tensor product of cochain complexes and the tensor product of $R$-supermodules. In other words, for two cochain complexes of $R$-supermodules $\BM^{\bullet}$ and $\BN^{\bullet}$ their tensor product is described in degree $n$ by the formula
    \begin{align*}    \left(\BM^{\bullet}\otimes^{\mathrm{super}}_R\BN^{\bullet}\right)^n\defeq\bigoplus_{i+j=n}\BM^i\otimes^{\mathrm{super}}_R\BN^j.
    \end{align*}
    The homological differential does not involve the $\BZ/2\BZ$-grading: namely, for the differential on a homogeneous element $m\otimes n$ acts as
    \[
    \mathrm{d}(m\otimes n)=\mathrm{d}(m)\otimes n+(-1)^{|n|}m\otimes d(n),
    \]
    where $|n|$ denotes the \textit{cohomological} grading of $n$ (independent from its parity).
    \item The braiding intertwines the Koszul sign rule of cochain complexes and the Koszul sign rule of $R$-supermodules. In other words, the braiding $\BM^{\bullet}\otimes^{\mathrm{super}}_R\BN^{\bullet}\simeq\BN^{\bullet}\otimes_R^{\mathrm{super}}\BM^{\bullet}$ is defined on a homogeneous element $m\otimes n$ as
    \[
    \begin{tikzcd}
    m\otimes n\arrow[mapsto]{r} &  (-1)^{|m||n|+p_mp_n}n\otimes m
    \end{tikzcd}
    \]
    where $m$ and $n$ respectively have cohomological degree $|m|$ and $|n|$ and parity $p_m$ and $p_n$. (Notice that we are using the Deligne sign convention, see \cite[Appendix to $\S1$]{Deligne-Supersymmetry} for a discussion on the matter.)
\end{enumerate}
The above rules do define a symmetric monoidal structure on the abelian category $\mathrm{Ch}(\mathrm{SMod}_R)$; but we need some compatibilities between the monoidal and the model structure to transfer such symmetric monoidal structure to the superderived $\infty$-category. We collect the following definition for the sake of the reader.

The category of cochain complexes of $R$-supermodules can be endowed also with an injective model structure; the advantage of the projective model structure of \cref{lemma:projective_model_structure} is that it is compatible with the monoidal structure. This allows us to build a derived tensor superproduct in the superderived $\infty$-category. We recall the following definition.
\begin{definition}
\label{def:monoidal_model_category}
Let $\mathscr{C}$ be a combinatorial model category equipped with a (symmetric) monoidal tensor product $\otimes$. We say that $\mathscr{C}$ is a \textit{(symmetric) monoidal model category} if the following axioms hold.
\begin{axiomenum}
\item\label{axiom:monoidal_closed}The (symmetric) monoidal structure is \textit{closed}, i.e., for any object $X$ in $\mathscr{C}$ the functor $(-)\otimes X$ admits a right adjoint $\smash{\underline{\Hom}}(X,-)$.
\item \label{axiom:unit_cofibrant}The unit for the (symmetric) monoidal structure is cofibrant. 
\item\label{axiom:quillen_bifunctor}Let $f\colon X\to Y$ and $g\colon U\to V$ be cofibrations. Then the map
\[
\begin{tikzcd}
f\wedge g\colon X\otimes V\coprod_{X\otimes U}Y\otimes U\arrow{r} &  Y\otimes V
\end{tikzcd}
\]
is a cofibration. If moreover either $f$ or $g$ is a weak equivalence, so is $f\wedge g$.
\end{axiomenum}
\end{definition}
\begin{remark}[{\cite[Ex.~4.1.7.6]{ha}}]
\label{remark:symmetric_monoidal_infty_cat}
Let $\mathscr{C}$ be a (symmetric) monoidal model category, in the sense of \cref{def:monoidal_model_category}. Then its $\infty$-categorical localization is once again (symmetric) monoidal.
\end{remark}
\begin{example}
Let $\mathrm{Ch}(\Mod_R)$ be the category of cochain complexes of $R$-modules, equipped with the projective model structure. Then it is a symmetric monoidal model category \cite[Prop.~$7.1.2.11$]{ha}. This allows us to deduce that also $\mathrm{Ch}(\mathrm{SMod}_R)$ is a symmetric monoidal model category, because under the equivalence of plain abelian categories 
\[
\mathrm{Ch}(\mathrm{SMod}_R)\simeq\mathrm{Ch}(\mathrm{Mod}_R)\times\mathrm{Ch}(\mathrm{Mod}_R)
\]
the projective model structure of \cref{lemma:projective_model_structure} is precisely the product of the projective model structures on both factors. In particular, we have the following.
\begin{enumerate}
    \item The unit $R^{1\mid0}$ is cofibrant, because it corresponds to the product of the cofibrant objects $R$ and $0$ in $\mathrm{Ch}(\mathrm{Mod}_R)$.
    \item The symmetric monoidal structure is closed: let $\BM^{\bullet}$ and $\BN^{\bullet}$ be two cochain complexes of $R$-supermodules. Then we define the cochain complex $\smash{\underline{\BH\mathrm{om}}}_R(\BM^{\bullet},\BN^{\bullet})^{\bullet}$ via the formula
    \[
    \smash{\underline{\BH\mathrm{om}}}_R(\BM^{\bullet},\BN^{\bullet})^n\defeq\prod_{k\in\BZ}\BH\mathrm{om}_R(\BM^{k},\BN^{n+k}),
    \]
    where $\BH\mathrm{om}_R$ denotes the $R$-supermodule of homomorphisms between ordinary $R$-supermodules constructed in \cref{fact:internal_hom}. Again, the differential does not involve the $\BZ/2\BZ$-grading: it acts on an element $f$ via the assignment 
\[
\begin{tikzcd}
f\arrow[mapsto]{r} &  \mathrm{d}\circ f -(-1)^{|f|}f\circ \mathrm{d}.
\end{tikzcd}
\]
\item \cref{axiom:quillen_bifunctor} holds. Indeed, this follows from the fact that a cofibration (resp. a weak equivalence) in $\mathrm{Ch}(\mathrm{SMod}_R)$ is the product of two cofibrations (resp. of two weak equivalences) in $\mathrm{Ch}(\Mod_R)$. Thus, by decomposing the tensor super-product into direct sums of tensor products of plain cochain complexes, the result follows from the fact that \cref{axiom:quillen_bifunctor} holds in $\mathrm{Ch}(\Mod_R)$. 
\end{enumerate}
Therefore, \cref{remark:symmetric_monoidal_infty_cat} implies that $\mathbf{SD}_{\infty}(R)$ is a symmetric monoidal $\infty$-category with respect to the derived tensor product of $R$-supermodules. Moreover, such monoidal structure preserves colimits separately in each variable. In fact, one can prove more: let $\derived_{\infty}(R)$ denote the stable derived $\infty$-category of cochain complexes of $R$-modules. Then, as a plain monoidal $\infty$-category, $\mathbf{SD}_{\infty}(R)$ is equivalent to the functor $\infty$-category $\Fun(\BZ/2\BZ,\derived_{\infty}(R))$ equipped with the Day convolution product induced by the abelian group structure on $\BZ/2\BZ$ (see \cite[Remark $4.8.1.13$]{ha}). As a corollary, we have that the derived tensor product is exact and colimit-preserving separately in each variable.
\end{example}
\begin{remark}
Our construction is very model-dependent, since we rely on presentations of $\infty$-categories via model categories and differential graded categories. However, we can use \cite{keidar2025twistedgradedcategories} to deduce that whenever $\mathscr{C}$ is a symmetric monoidal $\infty$-category, with monoidal unit $\boldone_{\mathscr{C}}$, which admits a minus one (i.e., a morphism of monoids $\phi$ from $\mathbb{Z}/2\mathbb{Z}$ to the monoid of units $\mathrm{Map}_{\mathscr{C}}(\boldone_{\mathscr{C}},\boldone_{\mathscr{C}})^{\times}$) then the functor $\infty$-category $\Fun\left(\BZ/2\BZ,\mathscr{C}\right)$ admits a symmetric monoidal structure which is obtained by the Day convolution after twisting the braiding with $\phi$ \cite[Prop.~2.3.4]{keidar2025twistedgradedcategories}. This allows us to define a symmetric monoidal $\infty$-category of $R$-supermodules for any commutative ring spectrum admitting a minus one (in particular, any discrete commutative ring of characteristic different from $2$) internally to $\infty$-category theory. 
\end{remark}
Since $\mathbf{SD}_{\infty}(R)$ is now a symmetric monoidal $\infty$-category, it makes sense to define monoids and commutative monoids with respect to its derived tensor product. The definition is a bit convoluted because it relies on operads that encode all possible diagrammatic homotopy coherencies: for example, if $\BA$ is a commutative monoid in $\mathbf{SD}_\infty(R)$, we need to specify all homotopies between $a_1\otimes \cdots \otimes a_n$ and $a_{\sigma(1)}\otimes\cdots \otimes a_{\sigma(n)}$, for any permutation $\sigma\in\Sigma_n$ and for any integer $n\geqslant2$. Nevertheless, we have the following.
\begin{definition}
\label{def:super-derived}
Let $R$ be a commutative ring.
\begin{defenum}
\item\label{def:superalg} The $\infty$-category of \textit{derived $R$-superalgebras} is the $\infty$-category $\mathrm{Alg}(\mathbf{SD}_{\infty}(R))$. We will denote it as $\mathrm{dSAlg}_R$.
\item\label{def:supercalg} The $\infty$-category of \textit{derived supercommutative $R$-algebras} is the $\infty$-category $\mathrm{CAlg}(\mathbf{SD}_{\infty}(R))$. We will denote it as $\mathrm{dSCAlg}_R$.
\end{defenum}
Let now $\BA$ be a derived $R$-superalgebra.
\begin{defenum}
\setcounter{defenumi}{2}
\item\label{def:supermod}The $\infty$-category of \textit{derived $\BA$-supermodules} is the $\infty$-category $\Mod_{\BA}(\mathbf{SD}_{\infty}(R))$. We shall call it the \textit{superderived $\infty$-category of $\BA$}, and denote it by $\mathbf{SD}_{\infty}(\BA)$.
\end{defenum}
Suppose that $\BA$ is supercommutative. Then the superderived $\infty$-category $\mathbf{SD}_{\infty}(\BA)$ inherits a relative derived supertensor product $\otimes_{\BA}^{\mathrm{super}}$, so we can consider algebras inside it.
\begin{defenum}
\setcounter{defenumi}{3}
\item\label{def:superalg-2}The $\infty$-category of \textit{derived $\BA$-superalgebras} is the $\infty$-category $\mathrm{Alg}(\mathbf{SD}_{\infty}(\BA))$. We will denote it as $\mathrm{dSAlg}_{\BA}$.
\item\label{def:supercalg-2}The $\infty$-category of \textit{derived supercommutative $\BA$-algebras} is the $\infty$-category $\mathrm{CAlg}(\mathbf{SD}_{\infty}(\BA))$. We will denote it as $\mathrm{dSCAlg}_{\BA}$.
\end{defenum}
\end{definition}

\begin{remark}
Let us delve deeper into \cref{def:super-derived}.
\begin{enumerate}
    \item A $R$-superalgebra in the sense of \cref{def:superalg} has to be interpreted as a cochain complex of $R$-supermodules $\BA$ with a multiplication
    \[
\begin{tikzcd}    \BA\otimes_R^{\mathrm{super}}\BA\arrow{r} & \BA
\end{tikzcd}
    \]
    which is required to be associative and unital only up to a coherent system of homotopies. Since the braiding is not involved in the definition of associative algebra objects, it is clear that a derived $R$-superalgebra is no different from a $\BZ/2\BZ$-graded derived $R$-algebra.
    \item A supercommutative $R$-algebra in the sense of \cref{def:supercalg} must be interpreted as a $R$-superalgebra $\BA$ that is moreover equipped with a coherent system of homotopies testifying to the commutativity of \textit{all} possible diagrams involving the multiplication of $n$ elements in $\BA$ and swapping factors. In particular, this datum is a piece of structure, and not a property.
    \item When $\BA$ is a $R$-superalgebra, a $\BA$-supermodule is a cochain complex of $R$-supermodules with an action
    \[
    \begin{tikzcd}
\BA\otimes_R^{\mathrm{super}}\BM\arrow{r} & \BM
    \end{tikzcd}
\]
    which is required to satisfy the usual associativity and compatibility properties only up to a coherent system of homotopies. By general features of module categories, for any $R$-superalgebra $\BA$ the $\infty$-category $\mathbf{SD}_{\infty}(\BA)$ is once again stable, presentable and $R$-linear.
    \item If $\BA$ is a derived supercommutative $R$-algebra, the $\infty$-category $\mathbf{SD}_{\infty}(\BA)$ admits a relative (derived) tensor product $\otimes^{\mathrm{super}}_{\BA}$ which turns $\mathbf{SD}_{\infty}(\BA)$ into a symmetric monoidal $\infty$-category (see \cite[$\S4.8$]{ha}). Again, this tensor product commutes with colimits separately in each variable, so for any derived $\BA$-supermodule $\BM$ we have an adjunction
    \[
    \adjunction{(-)\otimes^{\mathrm{super}}_{\BA}\BM}{\mathbf{SD}_{\infty}(\BA)}{\mathbf{SD}_{\infty}(\BA)}{\BR\BH\mathrm{om}_{\BA}(\BM,-).}
    \]
    \item In ordinary commutative algebra, it is known that a commutative algebra $S$ over a commutative ring $R$ is the same as a homomorphism of rings $R\to S$. This also holds in the $\infty$-categorical setting: for any derived supercommutative $R$-algebra $\BA$, we have
    \[
    \mathrm{dSCAlg}_{\BA}\simeq(\mathrm{dSCAlg}_R)_{\BA/},
    \]
    see for example \cite[Cor.~3.4.1.7]{ha}.
    \item Given a morphism of derived supercommutative $R$-algebras $\BA\to\BB$, there is a couple of adjunctions
    \[
    \adjunction{(-)\otimes^{\mathrm{super}}_{\BA}\BB}{\mathbf{SD}_{\infty}(\BA)}{\mathbf{SD}_{\infty}(\BB)}{\mathrm{forget}}\]
    and
    \[
    \adjunction{\mathrm{forget}}{\mathbf{SD}_{\infty}(\BB)}{\mathbf{SD}_{\infty}(\BA)}{\BR\BH\mathrm{om}_{\BA}(\BB,-),}
    \]
    generalizing the classical picture in commutative algebra.
    \item Using \cite[Thm.~4.1.8.4 and Prop.~7.1.4.3]{ha}, it is easy to prove that derived $R$-superalgebras can be described in terms of a model structure on differential grader $R$-superalgebras, having quasi-isomorphisms as weak equivalences. In other words, the $\infty$-category $\mathrm{dSAlg}_R$ can be interpreted as the $\infty$-category of differential graded $R$-superalgebras up to quasi-isomorphism. 
\item Similarly, when $R$ is a $\BQ$-algebra, then we can use \cite[Thm.~4.5.4.7 and Prop.~7.1.4.7]{ha} to deduce an analogous result regarding the $\infty$-category $\mathrm{dSCAlg}_R$ of derived \textit{supercommutative} $R$-algebras, allowing us to interpret a derived supercommutative $R$-algebra as a honest supercommutative differential graded $R$-superalgebra up to a quasi-isomorphism. However, when $R$ is not defined over the rationals one needs to use simplicial $R$-superalgebras instead: we do not delve deeper into the matter, being out of the scope of the present work.
\end{enumerate}
\end{remark}
\subsection{The \texorpdfstring{$t$}{t}-structure}
\label{subsec:t-structure}
One key feature of derived ($\infty$-)categories of modules over commutative rings is that they are equipped with a $t$-structure: objects in the derived category can be understood in terms of their cohomologies, which are naturally graded $R$-modules. Moreover the $t$-structure is very well-behaved: if a cochain complex has trivial cohomology in each degree, then it is a zero object in the derived category. This section generalizes this picture to the super setting.

We first observe that, for any commutative ring $R$, we have a functor $\mathbf{SD}_{\infty}(R)\to\derived_{\infty}(R)$ to the derived $\infty$-category of cochain complexes of $R$-modules described as follows: the image of a cochain complex of $R$-supermodules $\mathbb{M}^{\bullet}$ is $\mathbb{M}^{\bullet,\mathrm{even}}\oplus\mathbb{M}^{\bullet,\mathrm{odd}}$. We interpret this functor to take the \textit{underlying derived $R$-module} of $\BM$. In particular, it makes sense to consider the \textit{cohomology} of a derived $R$-supermodule to be defined as the cohomology of its underlying derived $R$-module (which agrees with the natural notion of cohomology of a cochain complex).

For a general derived $R$-superalgebra $\BA$ and a derived $\BA$-supermodule $\BM$, its \textit{underlying derived $R$-module} is the image of $\BM$ under the composition
\begin{equation}
\label{functor:underlyingmodule}
\begin{tikzcd}
\mathbf{SD}_{\infty}(\BA)\arrow{r} & \mathbf{SD}_{\infty}(R)\arrow{r} & \derived_{\infty}(R)
\end{tikzcd}
\end{equation}
where the first functor is the natural forgetful functor and the second is the underlying derived $R$-module functor defined above.
\begin{definition}
Let $\BM$ be a derived $R$-supermodule. We say that $\BM$ is \textit{connective} (resp. \textit{coconnective}) if the cohomology of its underlying derived $R$-module is concentrated in non-positive (resp. non-negative) cohomological degrees. 

If $\BA$ is a derived $R$-superalgebra, we say that $\BA$ is \textit{connective} if it is connective as a derived $R$-supermodule.

The $\infty$-category of connective derived supercommutative $R$-algebras will be denoted as $\mathrm{dSCAlg}_R^{\leqslant0}$.
\end{definition}
\begin{theorem}
\label{prop:t-structure}
Let $\BA\in\mathrm{dSCAlg}_R^{\leqslant0}$ be a connective derived $R$-superalgebra. Then the $\infty$-category $\mathbf{SD}_{\infty}(\BA)$ admits a left and right complete $t$-structure described as follows.
\begin{enumerate}
    \item [\mylabel{t-str1}{\normalfont{(1)}}] The subcategory $\mathbf{SD}_{\infty}(\BA)^{\leqslant0}$ consists of the full subcategory of connective objects.
    \item [\mylabel{t-str2}{\normalfont{(2)}}] The subcategory $\mathbf{SD}_{\infty}(\BA)^{\geqslant0}$ consists of the full subcategory of coconnective objects.
\end{enumerate}
Let $\iota^{\leqslant n}\dashv \tau^{\leqslant n}$ and $\tau^{\geqslant n}\dashv \iota^{\geqslant n}$ denote the truncation-inclusion adjunctions defined by this $t$-structure, and let $\BH^0\defeq \tau^{\geqslant0}\circ\tau^{\leqslant0}\colon\mathbf{SD}_{\infty}(\BA)\to\mathbf{SD}_{\infty}(\BA)^{\heartsuit}$ denote the natural functor to the heart. Then for any derived $\BA$-supermodule $\BM$ we have that $\BH^0(\BM)\cong\mathrm{H}^0(\BM^{\mathrm{even}})\oplus\mathrm{H}^0(\BM^{\mathrm{odd}})$, and $\mathbf{SD}_{\infty}(\BA)^{\heartsuit}$ is naturally equivalent to the ordinary abelian category of $\BH^0(\BA)$-supermodules. This turns $\mathbf{SD}_\infty(\BA)$ into a \emph{projectively rigid ttt-$\infty$-category}, in the sense of \cite{liang_higher_algebra}.
\end{theorem}
See \cref{sec:t-structure_proof} for the proof. 
\begin{remark}
\label{remark:on_projective_rigidity}
The projective rigidity condition in the statement of \cref{prop:t-structure} is quite technical in nature, but it allows us to harmlessly generalize some useful and fundamental results in homological algebra to the setting of derived supermodules. Indeed, \cite{liang_higher_algebra} immediately implies that Postnikov towers in $\mathbf{SD}_\infty(\BA)$ converge and consist of square-zero extensions; that the $\infty$-category of derived étale supercommutative $\BA$-algebras is equivalent to the ordinary category of étale supercommutative $\BH^0(\BA)$-algebras; Lazard's theorem on flat derived $\BA$-supermodules; and more.
\end{remark}

\begin{remark}
\label{remark:bosonic_reduction}
Given an ordinary $R$-superalgebra $\BA$, its bosonic reduction is defined as the quotient $\BA/\BA^{\mathrm{odd}}$. In the context of derived supergeometry, this quotient should be replaced by the mapping cone of the inclusion $\BA^{\mathrm{odd}}\to\BA$, computed in the derived $\infty$-category of $\BA$-supermodules. However, since all functors appearing in \eqref{functor:underlyingmodule} are exact, and since connective cochain complexes are stable under mapping cones, it follows that the bosonic reduction of a connective derived $R$-superalgebra is again connective. Moreover, if $\BA$ is an ordinary $R$-superalgebra, the cohomology of this mapping cone is isomorphic to the ordinary bosonic reduction of $\BA$.
\end{remark}
\subsection{Perfect supermodules}
\label{subsec:perf_objects}
Given a commutative ring $R$, the stable $\infty$-category $\mathrm{Perf}(R)$ is the full sub-$\infty$-category of the derived $\infty$-category $\mathbf{D}_\infty(R)$ spanned by fully dualizable objects for the (derived) tensor product over $R$. This means that this is the full sub-$\infty$-category spanned by cochain complexes $M$ such that there exists an object $M^{\vee}\defeq{\mathbf{R}{\Hom}}_{R}(M,R)$ and maps
\[
\begin{tikzcd}
\mathrm{ev}\colon M\otimes_RM^{\vee}\arrow{r} & R,\qquad\mathrm{coev}\colon M^{\vee}\otimes_RM\arrow{r} & R
\end{tikzcd}
\]
satisfying the usual diagrammatic identities; being $\mathbf{D}_\infty(R)$ a closed symmetric monoidal $\infty$-category, the dual $M^{\vee}$ is the internal homomorphism object to the unit, i.e.,
\[
M^{\vee}\defeq\mathbf{R}\mathrm{Hom}_R(M,R).
\]This leads naturally to the following piece of terminology.
\begin{definition}
\label{def:superf}
Let $\BA$ be a derived supercommutative $R$-algebra. The $\infty$-category $\mathrm{SPerf}(\BA)\subseteq\mathbf{SD}_\infty(\BA)$ is the full sub-$\infty$-category spanned by fully dualizable objects for the derived supertensor product over $\BA$.
\end{definition}
\begin{remark}
Since the derived supertensor product $\otimes^{\mathrm{super}}_{\BA}$ always commutes with colimits separately in each variable, it is clear that $\mathrm{SPerf}(\BA)$ contains the zero object and is closed under pushouts inside $\mathbf{SD}_\infty(\BA)$. In particular, $\mathrm{SPerf}(\BA)$ is stable and hence its homotopy category is triangulated.
\end{remark}
When working in the classical realm of derived algebraic geometry, is known that the $\infty$-category of perfect complexes admits two further equivalent characterizations.
\begin{enumerate}
    \item It is the smallest stable sub-$\infty$-category of $\mathbf{D}_\infty(R)$ containing $R$ and closed under retracts.
    \item It is the $\infty$-category spanned by \textit{compact} objects: that is, the $\infty$-category spanned by cochain complexes $M$ such that the mapping space functor
    \[
    \begin{tikzcd}
\mathrm{Map}_R(M,-)\defeq \Omega^\infty{\mathbf{R}{\Hom}}_R(M,-)\colon\mathbf{D}_\infty(R)\arrow{r} & \mathrm{Spaces}
    \end{tikzcd}
\]
    preserves filtered colimits.
\end{enumerate}
This straightforwardly generalizes to the super setting (see \cref{sec:perf_proof} for the proof).
\begin{prop}
\label{prop:superf}
Let $\BA$ a connective derived supercommutative $R$-algebra $\BA$ and let $\BM$ be a derived $\BA$-supermodule. The following are equivalent.
\begin{propenum}
    \item \label{prop:def} The derived $\BA$-supermodule $\BM$ is perfect.
    \item\label{prop:compact}The functor
\[
\begin{tikzcd}
\mathrm{Map}_{\BA}(\BM,-)^{\mathrm{even}} \colon\mathbf{SD}_\infty(\BA)\arrow{r} & \mathrm{Spaces}
\end{tikzcd}
\]
    commutes with filtered colimits.
    \item\label{prop:thick}The derived $\BA$-supermodule $\BM$ belongs to the smallest stable sub-$\infty$-category of $\mathbf{SD}_\infty(\BA)$ containing $\BA$ and $\Pi\BA$ and closed under retracts.
\end{propenum}
If $\BA$ is purely bosonic, then the above are equivalent to the following further condition.
\begin{propenum}
    \setcounter{propenumi}{3}
    \item \label{prop:degreewise}Both derived $\BA$-supermodules $\BM^{\mathrm{even}}$ and $\BM^{\mathrm{odd}}$ are perfect, in the usual sense of homological algebra.
\end{propenum}
\end{prop}
\begin{remark}
The connectivity assumption can be avoided, using a variant of the argument in \cite[Prop.~7.2.4.3]{ha}. Since we are not interested in derived supercommutative algebras which are not connective, we impose this further condition since it shortens the proof of the equivalence \ref{prop:def} $\Leftrightarrow$ \ref{prop:compact}.
\end{remark}
Let us finish this section by introducing a special class of derived supermodules over a derived supercommutative ring $\BA$.
\begin{definition}
\label{def:free-rank}
Let $\BM$ be a derived supermodule over a derived supercommutative algebra $\BA$. Fix sets of indices $I,J$. We say that $\BM$ is \textit{free} over $\BA$ if there exists an isomorphism
\[
\BM\simeq\BA^{\oplus I}\oplus \Pi\BA^{\oplus J}.
\]
If $I$ and $J$ are finite sets of cardinality $m$ and $n$ respectively, then we say that $\BM$ is free of rank $m|n$.  

We say that a derived supermodule is \textit{locally free of rank $m|n$} if it is a retract of a free derived supermodule of finite rank such that for any map of derived supercommutative rings $\BA\to \Bbbk$, where $\Bbbk$ is a field, the classical $\Bbbk$-supervector space $\BH^0(\Bbbk\otimes^{\mathrm{super}}_\BA\BM)$ has rank $m|n$.
\end{definition}
\begin{remark}
Our definition of locally free supermodule is inspired by the one in derived algebraic geometry (see \cite[Definition 2.9.2.1]{sag}). This recovers the notion of projective module when considering a discrete bosonic commutative ring.
\end{remark}
Obviously, free derived modules are always perfect; but this does not hold for a general locally free derived supermodule. Moreover, they do not form a stable sub-$\infty$-category of $\mathbf{SD}_\infty(\BA)$. Still, they behave nicely with respect to base change, as the following proposition shows.
\begin{prop}
\label{prop:rank-descent}
Let $\varphi\colon \BA\to\BB$ be a morphism of derived supercommutative algebras, and let $\BM$ be a derived $\BA$-supermodule that is (locally) free of rank $m|n$ over $\BA$. Then the base change $\BM\otimes^{\mathrm{super}}_\BA\BB$ is (locally) free of rank $m|n$ over $\BB$.
\end{prop}
\begin{proof}
First, assume that $\BM$ is free. By assumption, we can write
\[
\BM\simeq\BA^{\oplus m}\oplus \Pi\BA^{\oplus n}\simeq \BA^{\oplus m}\oplus\left(\BA\otimes_R^{\mathrm{super}}\Pi R\right)^{\oplus n},
\]
where $R$ is the (classical) base commutative ring. Therefore, using the compatibility of the supertensor product with colimits, we have
\begin{align*}
\BM\otimes^{\mathrm{super}}_\BA\BB&\simeq\left(\BA^{\oplus m}\oplus \Pi\BA^{\oplus n}\right)\otimes^{\mathrm{super}}_\BA\BB\\
&\simeq\BB^{\oplus m}\oplus \left(\Pi\BA^{\oplus n}\otimes^{\mathrm{super}}_\BA\BB\right).
\end{align*}
Since by definition $\Pi\BA\defeq \BA\otimes^{\mathrm{super}}_R\Pi R$, thanks to the associativity of the supertensor product we conclude that
\[
\BM\otimes^{\mathrm{super}}_\BA\BB\simeq \BB^{\oplus m}\oplus\Pi\BB^{\oplus n}.
\]
Finally, for the statement concerning locally free derived supermodules, we just observe that finitely generated free supermodules and retractions are obviously preserved by the base change functor.
\end{proof}
\subsection{Cotangent supercomplexes of superalgebras}
\label{section:cotangentsupercomplex}
In \cite[$\S7.3$]{ha}, the cotangent complex of an object $X$ in any presentable $\infty$-category $\mathscr{C}$ is defined in two steps. First, one proves that $\mathscr{C}$ admits a \textit{tangent bundle}
\[
\begin{tikzcd}
p\colon T_{\mathscr{C}}\arrow{r}{G} & \mathrm{Arr}(\mathscr{C})\arrow{r}{t} & \mathscr{C}.
\end{tikzcd}
\]
Here, $t$ denotes the evaluation on the target of a morphism in $\mathscr{C}$. Roughly, this has to be understood as a fibration $p$ of $\infty$-categories such that for any object $X$ of $\mathscr{C}$ the fiber $p^{-1}(X)$ is presentable and is equivalent to the $\infty$-category $\mathrm{Sp}(\mathscr{C}_{X/})$ of \textit{spectrum objects} in $\mathscr{C}_{/X}$; morally, these are derived analogs of abelian group objects of $\mathscr{C}_{/X}$. In particular, objects of $T_{\mathscr{C}}$ consist of pairs $(X,E)$ where $X$ is an object of $\mathscr{C}$ and $E$ is a spectrum object of $\mathscr{C}_{/X}$. The $\infty$-functor $G$ maps an object $(X,E)$ to a certain morphism $G(X)\to X$, hence $p$ maps the object $(X,E)$ to $X$.

Finally, one defines the cotangent complex of an object $X$ as the object $\LL_X$ of $\mathrm{Sp}(\mathscr{C}_{X/})$ corresponding to the image of $X$ under the left adjoint of the $\infty$-functor
\[
\begin{tikzcd}
q\colon T_{\mathscr{C}}\arrow{r}{G} & \mathrm{Arr}(\mathscr{C})\arrow{r}{s} & \mathscr{C},
\end{tikzcd}
\]
where now $s$ evaluates the morphism $G(X,E)\to X$ on the \textit{source}.

Luckily, this formalism is general and abstract enough so that we can apply it without harm to the supergeometric setting, with little modifications. Let $\BA$ any derived supercommutative $R$-algebra over a commutative ring $R$. Then the $\infty$-category $\mathrm{dSCAlg}_{\BA}$ of derived supercommutative $\BA$-algebras is presentable, because it is by definition the $\infty$-category of commutative algebras in the presentable $\infty$-category of derived $\BA$-supermodules. Therefore, we can apply the machinery described above and obtain our tangent bundle
\[
\begin{tikzcd}
p\colon T_{\mathrm{dSCAlg}_{\BA}}\arrow{r} & \mathrm{dSCAlg}_{\BA}.
\end{tikzcd}
\]
But since $\mathrm{dSCAlg}_{\BA}$ is the $\infty$-category of commutative algebra objects in the stable presentable $\infty$-category $\mathbf{SD}_{\infty}(\BA)$, \cite[\S7.3.4]{ha} allows us to give a more concrete description of such tangent bundle. We collect the basic results that we need in the following statement.
\begin{theorem}[{\cite[Thm.~7.3.4.18]{ha}}]
Let $\BA$ be any derived supercommutative $R$-algebra over a commutative ring $R$. Then the $\infty$-functor \eqref{functor:tangent_fibration} exhibits $\smash{\underline{\mathbf{SD}}}_{\infty}(\BA)$ as a tangent bundle to $\mathrm{dSCAlg}_{\BA}$. In particular:
\begin{thmenum}
\item\label{thm:spectra}Given a derived supercommutative $\BA$-algebra $\BS$, a spectrum object in the $\infty$-category $\mathrm{dSCAlg}_{\BA//\BS}$ of derived supercommutative $\BA$-algebras augmented over $\BS$ uniquely corresponds to a derived $\BS$-supermodule \cite[Thm.~7.3.4.13 and Cor.~7.3.4.14]{ha}.
\item \label{thm:augmentation_ideal}The equivalence $\mathrm{Sp}(\mathrm{dSCAlg}_{\BA//\BS})\overset{\simeq}{\to}\mathbf{SD}_{\infty}(\BS)$ is described by sending an object $\varepsilon\colon\BT\to\BS$ to $\mathrm{cone}(\varepsilon)[-1]$, seen as a derived $\BS$-supermodule \cite[Thm.~7.3.4.7]{ha}.
\item\label{thm:square_zero}The inverse $\mathbf{SD}_{\infty}(\BS)\overset{\simeq}{\to}\mathrm{Sp}(\mathrm{dSCAlg}_{\BA//\BS})$ is described by sending a derived $\BS$-supermodule $\BM$ to a spectrum object whose underlying augmented derived $\BA$-superalgebra is \emph{homotopy square-zero}. This means that its underlying derived $\BA$-supermodule is $\BS\oplus \BM$, and the restriction of the multiplication map $\BM\otimes_{\BA}^{\mathrm{super}}\BM\to\BM$ is null-homotopic \cite[Rem.~$7.3.4.16$]{ha}.
\end{thmenum}
\end{theorem}
As a consequence of \cref{thm:square_zero}, we obtain that the $\infty$-functor
\[
\begin{tikzcd}
q\colon\smash{\underline{\mathbf{SD}}}_{\infty}(\BA)\arrow{r}{G} & \mathrm{Arr}(\mathrm{dSCAlg}_{\BA})\arrow{r}{s} & \mathrm{SCAlg}_{\BA}
\end{tikzcd}
\]
is described by first sending a pair $(\BS,\BM)$ to the morphism $\BS\oplus\BM\to\BS$, and then by evaluating it on the source, obtaining the square-zero extension $\BS\oplus\BM$. Its left adjoint 
\[
\begin{tikzcd}
\BL\colon \mathrm{dSCAlg}_{\BA}\arrow{r} & \smash{\underline{\mathbf{SD}}}_{\infty}(\BA)
\end{tikzcd}
\]
takes a derived supercommutative $\BA$-algebra $\BS$ and produces an object $(\BS,\BL_{\BS/\BA})$, where $\BL_{\BS/\BA}$ is a derived $\BS$-supermodule.
\begin{definition}
\label{def:super_cotangent}
Let $\BA$ be any derived supercommutative $R$-algebra over an arbitrary commutative ring $R$, and let $\BS$ be a derived supercommutative $\BA$-algebra. The \textit{cotangent supercomplex} of $\BS$ is the derived $\BS$-supermodule $\BL_{\BS/\BA}$.
\end{definition}
\begin{remark}
\label{remark:cotangent_universal_property}
From \cite[\S7.3.2]{ha}, it follows that $\BL_{\BS/\BA}$ is the derived $\BS$-supermodule enjoying the following universal property: for all $\BS$-modules $\BN$ we have a homotopy equivalence of spaces
\[
\mathrm{Map}_{\mathbf{SD}_{\infty}(\BS)}(\BL_{\BS/\BA},\BN)\simeq\mathrm{Map}_{\mathrm{dSCAlg}_{\BA//\BS}}(\BS,\BS\oplus\BN).
\]
In other words: let $\BS$ be a derived supercommutative $\BA$-algebra, and let $\BM$ be a $\BS$-supermodule. We define the space of \textit{higher $\BA$-linear super-derivations} from $\BS$ to $\BM$ to be the space of morphisms from $\BS$ to $\BS\oplus\BM$ as derived supercommutative $\BA$-algebras augmented over $\BS$:
\begin{align*}
\mathrm{Der}_{\BA}(\BS,\BM)&\defeq\mathrm{Map}_{\mathrm{dSCAlg}_{\BA//\BS}}(\BS,\BS\oplus \BM)
\\&\simeq\mathrm{fib}_{\mathrm{id}_{\BA}}\left(\mathrm{Map}_{\mathrm{dSCAlg}_{\BA}}(\BS,\BS\oplus \BM)\longrightarrow\mathrm{Map}_{\mathrm{dSCAlg}_{\BA}}(\BS,\BS)\right).
\end{align*}
Then $\BL_{\BS/\BA}$ corepresents the $\infty$-functor
\[
\begin{tikzcd}
\mathrm{Der}(\BS,-)\colon\mathbf{SD}_{\infty}(\BS)\arrow{r} & \mathrm{Spaces}.
\end{tikzcd}
\]
Similarly, if $\varphi\colon\BS\to\BT$ is a morphism of derived supercommutative $\BA$-algebras, the functor $\BL$ yields a morphism
\[
\begin{tikzcd}
\BL_{\BS/\BA}\otimes^{\mathrm{super}}_{\BS}\BT \arrow{r} & \BL_{\BT/\BA}.
\end{tikzcd}
\]
For a derived $\BT$-supermodule $\BN$, we write
\[
\mathrm{Der}_{\BS}(\BT,\BN)\defeq\mathrm{Map}_{\mathrm{dSCAlg}_{\BA//\BT}}(\BS,\BT\oplus\BN)
\]
for the space of \textit{$\BS$-linear} super-derivations from $\BT$ to $\BN$. Then the object $\BL_{\BS/\BA}\otimes^{\mathrm{super}}_{\BS}\BT$ is seen to corepresent the functor
\[
\begin{tikzcd}
\mathrm{Der}_{\BS}(\BT,-)\colon\mathbf{SD}_{\infty}(\BS)\arrow{r} & \mathrm{Spaces}.
\end{tikzcd}
\]
\end{remark}
Using \cite[$\S$ 7.3.3]{ha}, this picture can be further generalized to the relative setting.
\begin{definition}
\label{def:relative_cotangent}
Let $\BA$ a ground derived supercommutative $R$-algebra, and let $\varphi\colon \BS\to\BT$ be a morphism between derived supercommmutative $\BA$-superalgebras. Let $\BL_{\BS/\BA}$, $\BL_{\BT/\BS}$ denote their cotangent supercomplexes. The \emph{relative cotangent supercomplex} $\BL_{\BT/\BS}$ is the mapping cone of the natural morphism of derived $\BT$-supermodules
\[
\begin{tikzcd}
\BL_{\BS/\BA}\otimes^{\mathrm{super}}_{\BS}\BT\arrow{r} & \BL_{\BT/\BA}.
\end{tikzcd}
\]
\end{definition}
\begin{remark}
\label{remark:relative_cotangent_functorial}
\cref{def:relative_cotangent} is not quite the precise formalization of the relative cotangent supercomplex. Indeed, in the formalism of \cite[$\S$ 7.3.3]{ha} the relative cotangent complex is defined to be a \textit{functor}
\[
\begin{tikzcd}
\mathrm{Arr}(\mathscr{C})\arrow{r} &  T_{\mathscr{C}}.
\end{tikzcd}
\]
from the $\infty$-category of morphisms in a presentable $\infty$-category $\mathscr{C}$ to its tangent bundle. Point-wise, this is proved to be computed as a mapping cone \textit{a posteriori} \cite[Prop.~7.3.3.5]{ha}. In particular, the relative cotangent supercomplex enjoys the following functorial property with respect to pushouts: for any pair of morphisms $\BS\to\BT$ and $\BS\to\BT'$ of derived supercommutative $\BA$-algebras, we have an equivalence of $(\BT\otimes^{\mathrm{super}}_{\BS}\BT')$-supermodules
\[
\BL_{\BT/\BS}\otimes_{\BT}^{\mathrm{super}}\left(\BT\otimes^{\mathrm{super}}_{\BS}\BT'\right)\simeq\BL_{\BT\otimes^{\mathrm{super}}_{\BS}\BT'/\BT'}.
\]
\end{remark}
We finish this section by collecting some important results concerning the relationship between cotangent supercomplexes and superderivations. We first recall the following pieces of terminology from \cite{ha}.
\begin{definition}
Let $\BA\in\mathrm{dSCAlg}_R^{\leqslant0}$ be a connective supercommutative $R$-algebra. 
\begin{defenum}
    \item\label{def:super-derivations}The $\infty$-category $\BD\mathrm{er}$ of \emph{$\BA$-linear super-derivations} is the $\infty$-category consisting of triples $(\BS,\BM,\eta)$ where $\BS$ is a connective derived supercommutative $\BA$-algebra, $\BM$ is a derived $\BS$-supermodule, and $\eta\colon\BL_{\BS/\BA}\to \BM$ is a derivation (equivalently, a morphism of derived supercommutative $\BA$-algebras $\BS\to\BS\oplus\BM$ which is a section of the projection). We say that a super-derivation is \emph{$n$-small} if $\BM\in\mathbf{SD}_\infty(\BS)^{\leqslant-n}\cap\mathbf{SD}_\infty(\BS)^{\geqslant-2n}$: we denote by $\BD\mathrm{er}_{n\mathrm{-sm}}$ the full sub-$\infty$-category of $\BD\mathrm{er}$ spanned by $n$-small super-derivations. 
    \item\label{def:square-zero}A morphism $\varphi\colon\BT\to\BS$ of connective derived supercommutative $\BA$-algebras is a \emph{square-zero extension} of $\BS$ by a derived $\BS$-supermodule $\BM$ if there exists a derivation $\eta\colon \BS\to\BS\oplus\BM[1]$ and a pullback diagram of derived supercommutative $\BA$-algebras augmented over $\BS$ of the form
    \[
 \xymatrix{
\BT \ar[r]^{\varphi} \ar[d]^{\varphi} & \BS \ar[d]^{\eta}
 \\
\BS\ar[r]^{\iota} & \BS\oplus\BM[1].
 }
\]
If that is the case, we shall denote $\BT$ as $\BS^{\eta}$.
\item The $\infty$-category $\mathrm{Arr}(\mathrm{dSCAlg}^{\leqslant0}_{\BA})_{n\text{-sm}}$ of \emph{$n$-small extensions} is the full sub-$\infty$-category of $\mathrm{Arr}(\mathrm{dSCAlg}^{\leqslant0}_{\BA})$ spanned by morphisms $\varphi\colon \BT\to\BS$ for which $\mathrm{cone}(\varphi)[-1]\in\mathbf{SD}_\infty(\BS)^{\leqslant -n}\cap\mathbf{SD}_\infty(\BS)^{\geqslant-2n}$ and for which the multiplication
\[
\mathrm{cone}(\varphi)[-1]\otimes^{\mathrm{super}}_{\BT}\mathrm{cone}(\varphi)[-1]\to\mathrm{cone}(\varphi)[-1]\]is null-homotopic.
\end{defenum}
\end{definition}
\begin{theorem}[{\cite[Thm.~6.3]{liang_higher_algebra}}]
\label{thm:cotangent-square-zero}
There is a functor
\[
\begin{tikzcd}[row sep=tiny]
\BD\mathrm{er}\arrow{r} & \mathrm{Arr}(\mathrm{dSCAlg}^{\leqslant0}_{\BA})
\\
(\BS,\BM,\eta)\arrow[mapsto]{r} &  \left\{\BS\oplus\BM\to\BM\right\}
\end{tikzcd}
\]
which induces an equivalence
\[
\BD\mathrm{er}_{n\text{-}\mathrm{sm}}\simeq \mathrm{Arr}(\mathrm{dSCAlg}^{\leqslant0}_{\BA})_{n\text{-}\mathrm{sm}}.
\]
In particular, all $n$-small extensions of connective derived supercommutative $\BA$-algebras are square-zero extensions.
\end{theorem}
Informally, the meaning of \cref{thm:cotangent-square-zero} is that even if not all arbitrary square-zero extensions $\varphi\colon\BT\to\BS$ can be recovered by the $\BS$-supermodule $\mathrm{cone}(\varphi)[-1]$ together with the super-derivation $\BS\to\BS\oplus\mathrm{cone}(\varphi)[-1]$, this holds when restricting ourselves to $n$-small square-zero extensions and derivations. 
\begin{corollary}
\label{corollary:postnikov_tower}
Let $\tau^{\geqslant-n}$ denote the truncation functor for the $t$-structure of \cref{prop:t-structure}. Then for every connective derived supercommutative $R$-algebra $\BA$, every connective derived supercommutative $\BA$-algebra $\BS$ is a limit of its Postnikov tower
\[
\begin{tikzcd}
\cdots\arrow{r} & \tau^{\geqslant-2}\BS\arrow{r} & \tau^{\geqslant-1}\BS\arrow{r} & \tau^{\geqslant0}\BS\,\cong\,\BH^0(\BS).
\end{tikzcd}
\]
Moreover, each morphism in the Postnikov tower is a square-zero extension.
\end{corollary}
The following statements can be deduced straight-forwardly from \cite[Thm.~7.4.3.12]{ha}, generalizing the classical picture for cotangent complexes of spectra.
\begin{theorem}
\label{thm:connectivity-cotangent}
Let $\varphi\colon\BS\to\BT$ be a morphism of connective derived supercommutative algebras.
\begin{thmenum}
\item\label{thm:obstruction-theory-on-truncation}If $\mathrm{cone}(\varphi)$ is $(-n)$-connective, then $\BL_{\BB/\BA}$ is $(-n)$-connective. The converse holds if $\BH^0(\varphi)$ is an isomorphism.
\item The morphism $\varphi$ is a quasi-isomorphism if and only if $\BH^0$ is an isomorphism and $\BL_{\BB/\BA}$ vanishes.
\item The cotangent supercomplexes $\BL_{\BA}$ and $\BL_{\BB/\BA}$ are connective.
\item The $0$-th cohomology $\BH^0(\BL_{\BA})$ is isomorphic to $\BH^0(\BL_{\BH^0(\BA)})$, which in turn is isomorphic to $\Omega^1_{\BH^0(\BA)}$.
\item The $0$-th cohomology $\BH^0(\BL_{\BB/\BA})$ is isomorphic to $\Omega^1_{\BH^0({\BB/\BA})}$.
\end{thmenum}
\end{theorem}
\begin{remark}
Let $\BS$ be a derived supercommutative $\BA$-algebra over a derived supercommutative $R$-algebra $\BA$. Then the cotangent supercomplex of $\BS$ as a $\BA$-superalgebra is the same as the relative cotangent supercomplex of $\BS$ over $\BA$ as $R$-superalgebras \cite[$\S$ 7.3.3]{ha}. Therefore, the statements of \cref{thm:connectivity-cotangent} imply analogous statements on any connective base superalgebra $\BA$.
\end{remark}
\begin{warning}
The theory developed here generalizes the \textit{topological} cotangent complex of ring spectra due to André--Quillen. This means that, for $R$ an ordinary commutative ring and $A$ a derived commutative $R$-algebra seen as a even derived supercommutative $R$-algebra, our cotangent complex will recover the cotangent complex $\BL_{A/R}$ defined in \cite{ha} and this, in general, differs from Illusie's cotangent complex $\BL_{A/R}^{\mathrm{alg}}$ from algebraic geometry. However, for any commutative ring $R$, the complex $\BL_{A/R}^{\mathrm{alg}}$ is obtained by base change over $A$ from $\BL_{A/R}$, and if $A\to B$ is a morphism of \textit{connective} derived $R$-superalgebras, then we have a morphism $\BL_{B/A}\to\BL_{B/A}^{\mathrm{alg}}$ which induces an isomorphism in cohomology in degrees $\geqslant-1$, a surjection in degree $-2$, and is a rational isomorphism in all degrees. In particular, if $R$ is a $\BQ$-algebra then $\BL_{B/A}\simeq\BL_{B/A}^{\mathrm{alg}}$. See \cite[$\S$ 25.3]{sag} for more on the matter.

While an analogous theory of \textit{algebraic} cotangent complexes can be developed in the super-setting, we are not interested in developing it in the present work.
\end{warning}

\subsection{Grothendieck topologies on derived superalgebras}
\label{sec:grothendieck_topologies}
Recall (see for example \cite{BrHR21}) that a morphism $\varphi\colon\BS\to\BT$ of ordinary $\BA$-superalgebras is said to be 
\begin{enumerate}
    \item \textit{flat} if the functor $(-)\otimes^{\mathrm{super}}_{\BS}\BT\colon\mathrm{SMod}_{\BS}\to\mathrm{SMod}_{\BT}$ is exact,
    \item \textit{faithfully flat} if the functor $(-)\otimes^{\mathrm{super}}_{\BS}\BT\colon\mathrm{SMod}_{\BS}\to\mathrm{SMod}_{\BT}$ is exact and faithful (equivalently, if the morphism is flat and induces a surjection between the corresponding affine superschemes),
    \item \emph{smooth of relative dimension $m|n$} if it is flat, locally of finite presentation, $\Omega_{\BT/\BS}$ is locally free of rank $m|n$, and for every irreducible components $\boldit V \subset \Spec \BT$ and $\boldit W \subset \Spec \BS$ such that $\Spec \varphi$ maps $\boldit V$ into $\boldit W$, one has $\dim \boldit V=\dim \boldit W + m|n$,
    \item \textit{étale} if it is smooth of relative dimension $0|0$. 
\end{enumerate}

\begin{definition}
\label{def:etale_topology_derived}
Let $\BA$ be any derived supercommutative $R$-algebra.
\begin{defenum}
\item \label{def:flatness} We say that a derived $\BA$-supermodule $\BM$ is \textit{flat} if $\BH^0(\BM)$ is a flat $\BH^0(\BA)$-supermodule and for all $n\in\BZ$ the natural morphism
\[
\begin{tikzcd}
\BH^n(\BA)\otimes_{\BH^0(\BA)}\BH^0(\BM)\arrow{r} & \BH^n(\BM)
\end{tikzcd}
\]
is an isomorphism of $\BH^0(\BA)$-supermodules.
\item \label{def:faithfully_flatness}  We say that a derived $\BA$-supermodule $\BM$ is \textit{faithfully flat} if it is flat as a derived $\BA$-supermodule and $\BH^0(\BM)$ is a faithfully flat $\BH^0(\BA)$-supermodule.
\item \label{def:etale} We say that a morphism $\varphi\colon\BS\to\BT$ of derived supercommutative $\BA$-algebras is \textit{étale} if the morphism $\BH^0(\varphi)\colon\BH^0(\BS)\to\BH^0(\BT)$ of ordinary supercommutative $\BH^0(\BA)$-algebras is étale and $\varphi$ exhibits $\BT$ as a flat derived $\BS$-supermodule.
\item \label{def:smooth} We say that a morphism $\varphi\colon\BS\to\BT$ of derived supercommutative $\BA$-algebras is \textit{smooth} if the morphism $\BH^0(\varphi)\colon\BH^0(\BS)\to\BH^0(\BT)$ of ordinary supercommutative $\BH^0(\BA)$-algebras is smooth and $\varphi$ exhibits $\BT$ as a flat derived $\BS$-supermodule.
\item\label{def:zariski}We say that a morphism $\varphi\colon\BS\to\BT$ of derived supercommutative $\BA$-algebras is a \textit{Zariski open immersion} if the morphism $\BH^0(\varphi)\colon\BH^0(\BS)\to\BH^0(\BT)$ of ordinary supercommutative $\BH^0(\BA)$-algebras induces a Zariski open immersion at the level of superschemes, and $\varphi$ exhibits $\BT$ as a flat derived $\BS$-supermodule.
\end{defenum}
\end{definition}
If $\BA$ is an underived supercommutative $R$-algebra, the above definitions agree with the usual definitions from algebraic supergeometry. In fact, we can state a stronger \textit{topological invariance} result, when we consider étale morphisms and Zariski open immersions.
\begin{theorem}
\label{thm:topological_invariance}
Let $\BA$ be a connective derived supercommutative algebra. The functor
\[
\begin{tikzcd}
\BH^0\colon\mathrm{dSCAlg}_{\BA}^{\leqslant0,\mathrm{\acute{e}t}}\arrow{r}&\mathrm{CAlg}^{\mathrm{\acute{e}t}}_{\BH^0(\BA)}
\end{tikzcd}
\]
is an equivalence of $\infty$-categories between derived étale supercommutative $\BA$-algebras and underived étale supercommutative $\BH^0(\BA)$-algebras. Moreover, it restricts to an equivalence of $\infty$-categories between derived Zariski open immersions from $\BA$ and open immersions from $\BH^0(\BA)$.

In particular, all these $\infty$-categories are $1$-categories.
\end{theorem}
\begin{proof}
The étale case is precisely \cite[Thm.~6.26]{liang_higher_algebra}. The Zariski open immersion case is then a formal consequence from the étale case and the definition of derived Zariski open immersion.
\end{proof}
We recall the following useful characterization of derived flat supermodules.
\begin{prop}[{\cite[Prop.~2.15]{liang_higher_algebra}}]
\label{prop:flatcharacterization}
Let $\BA\in\mathrm{dSCAlg}_R^{\leqslant0}$ a connective derived supercommutative $R$-algebra and let $\BM$ be a connective derived $\BA$-supermodule. The following are equivalent.
\begin{propenum}
\item \label{prop:flat1}The supermodule $\BM$ is flat over $\BA$.
\item \label{prop:flat2}If $\BN$ is a $n$-coconnective module (i.e., $\BH^k(\BN)\cong0$ for $k>n$) then $\BM\otimes^{\mathrm{super}}_{\BA}\BN$ is $n$-coconnective.
\item\label{prop:flat3}If $\BN$ is a discrete flat $\BA$-supermodule (i.e., $\BN\simeq\BH^0(\BN)$) then $\BM\otimes^{\mathrm{super}}_{\BA}\BN$ is discrete.
\end{propenum}
\end{prop}

\begin{definition}
Let $\BA$ be a connective derived supercommutative $R$-algebra. We can define two Grothendieck topologies on the $\infty$-category of derived affine $\BA$-superschemes $\mathrm{dSAff}_{\BA}\defeq\smash{(\mathrm{dSCAlg}^{\leqslant0}_{\BA})^{\mathrm{op}}}$.
\begin{defenum}
\item\label{def:fpqc topology} The \textit{faithfully flat topology} is the Grothendieck topology whose coverings are collections of  morphisms $\left\{\BS\to\BT_i\right\}_{i\in I}$ for which there exists a finite subset of indices $\left\{i_1,\ldots,i_n\right\}$ such that the morphism
\[
\begin{tikzcd}
\BS\arrow{r} &\displaystyle \prod_{j=1}^n\BT_{i_j}
\end{tikzcd}
\]
is a faithfully flat morphism of derived supercommutative $\BA$-algebras.
\item \label{def:etale topology}The \textit{étale topology} whose coverings are collections of \textit{étale} morphisms $\left\{\BS\to\BT_i\right\}_{i\in I}$ for which there exists a finite subset of indices $\left\{i_1,\ldots,i_n\right\}$ such that the morphism
\[
\begin{tikzcd}
\BS\arrow{r} & \displaystyle\prod_{j=1}^n\BT_{i_j}
\end{tikzcd}
\]
is an \'etale morphism of derived supercommutative $\BA$-algebras and such that
\[
\begin{tikzcd}
\Spec \BH^0(\prod_{j}\BT_{i_j}) \arrow{r} & \Spec \BH^0(\BS)
\end{tikzcd}
\]
is a surjection between ordinary topological spaces.
\end{defenum}
\end{definition}

\begin{remark}
We observe the following.
\begin{enumerate}
\item \cref{def:fpqc topology,def:etale topology} do define Grothendieck topologies, because faithfully flat and étale morphisms define Grothendieck topologies on the ordinary categories of $\BH^0(\BA)$-superalgebras and the derived tensor super-product behaves well under composition and base change. 
\item It is clear from the definition that the faithfully flat topology is finer than the étale topology.
\end{enumerate}
\end{remark}
\subsection{Descent for derived \texorpdfstring{$\infty$}{oo}-categories of supermodules}
In derived algebraic geometry, the cotangent complex is first defined over affine derived schemes in terms of (higher) derivations. Then, this construction is extended to arbitrary derived (pre)stacks, in virtue of the fact that the global cotangent complex of a derived (pre)stack $X$ is uniquely determined (up to a contractible space of ambiguities) by the following properties.
\begin{enumerate}
    \item For any morphism $x\colon\Spec R\to X$, there exists a cotangent complex at $x$.
    \item For any morphisms of derived affine schemes $f\colon\Spec S\to\Spec R$, the cotangent complex at $x\circ f$ is the base change of the cotangent complex at $x$ along $f$.
\end{enumerate}
In fact, under these assumptions the local cotangent complexes can be glued to obtain a global object living over $X$. This discussion makes it clear that a key result in establishing the definition of a global cotangent complex is that \textit{quasicoherent sheaves of derived stacks satisfy descent}. The goal of this section is to prove a super-analog of this statement, in order to apply the usual machinery to introduce a notion of \textit{derived super-derivation} and finally obtain our desired global cotangent supercomplex. 
\begin{construction}
Let $\BA$ be connective derived supercommutative $R$-algebra over a classical commutative ring $R$. Let $\mathrm{Pr}^{\mathrm{L},\otimes}$ denote the $\infty$-category of presentable symmetric monoidal $\infty$-categories: then the derived supercategory $\mathbf{SD}_{\infty}(\BA)$ is an object in $\mathrm{Pr}^{\mathrm{L},\otimes}$, and we can write
\[
\mathrm{Pr}^{\mathrm{L},\otimes}_{\BA}\defeq \Mod_{\mathbf{SD}_{\infty}(\BA)}(\mathrm{Pr}^{\mathrm{L},\otimes})
\]
to mean the \textit{$\infty$-category of presentable $\BA$-linear symmetric monoidal $\infty$-categories}: roughly, these are pretriangulated differential graded categories which are $\BA$-linear, in a homotopy-coherent way. The technical machinery of \cite[Ch.~$4.8.5$]{ha} allows to define a $\infty$-functor
\begin{equation}
\label{functor:derived_mod}
\begin{tikzcd}
\mathbf{SD}_{\infty}(-)\colon\mathrm{dSCAlg}^{\leqslant0}_{\BA}\arrow{r} & \mathrm{Pr}^{\mathrm{L},\otimes}_{\BA},
\end{tikzcd}
\end{equation}
where a morphism $\BS\to\BT$ is sent to the adjunction $(-)\otimes^{\mathrm{super}}_{\BS}\BT\dashv\mathrm{forget}$. In virtue of the $\infty$-categorical Grothendieck construction \cite[$\S3.2$]{htt}, this is equivalently described in terms of a presentable (both Cartesian and coCartesian) fibration of $\infty$-categories
\begin{equation}
\label{functor:tangent_fibration}
\begin{tikzcd}
\smash{\underline{\mathbf{S}\derived}}_{\infty}(\BA)\arrow{r} & \mathrm{dSCAlg}^{\leqslant0}_{\BA}.
\end{tikzcd}
\end{equation}
Roughly, $\smash{\underline{\mathbf{S}\derived}}_{\infty}(\BA)$ is the $\infty$-category consisting of pairs of objects $(\BS,\BM)$, where $\BS$ is a connective derived supercommutative $\BA$-algebra and $\BM$ is a derived $\BS$-supermodule. Morphisms $(\BS,\BM)\to(\BT,\BN)$ are given by a morphism of connective derived supercommutative $R$-algebras $\BS\to\BT$, together with a morphism of derived $\BT$-supermodules $\BM\otimes^{\mathrm{super}}_{\BS}\BS\to\BN$.
\end{construction}
\begin{theorem}[{See also \cite[$\S$3.4]{dang_derived_supergeometry}}]
\label{thm:flat_descent}
The $\infty$-functor \eqref{functor:tangent_fibration} is a sheaf for the faithfully flat (hence, for the étale) topology.
\end{theorem}
The proof is essentially the same as the proof of faithfully flat descent for derived $\infty$-categories in \cite[\S D.6]{sag} with ``super'' added where needed. We still revisit it in \cref{sec:flat_descent_proof} for the sake of the reader.
\section{Derived algebraic supergeometry}
\label{sec:dasg}
One of the leading motivations for introducing derived algebraic geometry is that, in many cases of interest, obstruction theories (and thus virtual fundamental classes, enumerative invariants, and so forth) arise from \textit{derived enhancements} of classical moduli spaces/stacks. This is what happens for the usual perfect obstruction theory on the moduli stack of stable maps yielding Gromov--Witten invariants. This is why we need to introduce the formalism of \textit{derived algebraic supergeometry}, which is a key input in order to adapt many constructions and results from derived algebraic geometry to supergeometry.

Let us emphasize that the idea of extending homotopical and derived tools to the super setting fits into a currently quite lively field of inquiry. In a differential context, derived supergeometry was first introduced in \cite{carchedi_derived_supergeometry}; during the final stages of preparation of this current paper, derived \textit{algebraic }methods were introduced in \cite{dang_derived_supergeometry}. In both cases, the main focus seemed to be the theory of D-modules and Higgs sheaves. It is also worth mentioning the foundational papers \cite{hesselholt_pstragowski_I,hesselholt_pstragowski_II} that introduced derived methods in the closely related (but different) field of \textit{Dirac geometry}, concerning algebras in the $\mathbf{Z}$-graded (rather than in the $\mathbf{Z}/2$-graded) setting. 

In this section, we describe the general characters and features of derived algebraic supergeometry and relate them to the analogs from usual derived algebraic geometry. The well-definedness of these concepts can be deduced from the affine case together with appropriate descent arguments.

\subsection{Accessible presheaves and derived superstacks}
Let $\mathscr{C}$ be a small $\infty$-category. The $\infty$-category of presheaves of spaces (or homotopy types, or anima) over $\mathscr{C}$ is the functor $\infty$-category
\[
\mathcal{P}(\mathscr{C})\defeq \mathrm{Fun}(\mathscr{C}^{\mathrm{op}},\mathrm{Spaces}).
\]
If moreover $\mathscr{C}$ (or equivalently, its homotopy category $\mathrm{h}\mathscr{C}$) is equipped with a Grothendieck topology $\tau$, then one can consider the full sub-$\infty$-category $\mathrm{Shv}_\tau(\mathscr{C})$ of $\mathcal{P}(\mathscr{C})$ spanned by \textit{sheaves} of spaces. Assuming for simplicity that $\mathscr{C}$ admits all colimits, sheaves can be described as those functors $F\colon\mathscr{C}^{\mathrm{op}}\to\mathrm{Spaces}$ satisfying the following property: let $\left\{\varphi_i\colon U_i\to X\right\}_{i\in I}$ be a $\tau$-covering in $\mathscr{C}$, and let $\varphi\colon\coprod_i U_i\to X$ be the resulting morphism. Let $\mathrm{\check{C}}(\varphi)$ denote the \v{C}ech nerve of $\varphi$. Then the natural comparison map
\[
\begin{tikzcd}
F(X)\arrow{r} & \lim_{[n]\in\bm{\Delta}} F(\mathrm{\check{C}}(\varphi)_n)
\end{tikzcd}
\]
is a homotopy equivalence. 

The inclusion $\mathrm{Shv}_\tau(\mathscr{C})\subseteq\mathcal{P}(\mathscr{C})$ always commutes with all limits; in fact, $\mathrm{Shv}_\tau(\mathscr{C})$ is a \textit{topological localization} of $\mathcal{P}(\mathscr{C})$, in the sense that there exists a sheafification functor 
\[
\begin{tikzcd}
(-)^\dagger\colon\mathcal{P}(\mathscr{C})\arrow{r} & \mathrm{Shv}_\tau(\mathscr{C})
\end{tikzcd}
\]
behaving as a left adjoint to the inclusion functor, which moreover preserves finite colimits.

However, for a bigger $\infty$-category such as $\mathrm{dSCAlg}^{\leqslant0}_{\BA}$, where $\BA$ is a connective derived supercommutative $R$-algebra, size issues provide obstructions to the existence of a sheafification functor for the fpqc topology. This is why we restrict ourselves to \textit{accessible presheaves}. We first set some terminology.
\begin{notation}
Let $\BA$ be an arbitrary connective derived supercommutative $R$-algebra. We denote by $\mathrm{dSAff}_{\BA}$ the opposite $\infty$-category of $\mathrm{dSCAlg}_{\BA}^{\leqslant0}$, and call it the $\infty$-category of \emph{derived affine superschemes}. For a derived supercommutative $\BA$-algebra $\BS$, we denote by $\Spec \BS$ the corresponding object in $\mathrm{dSAff}_{\BA}$.
\end{notation}
\begin{definition}
\label{def:derived-super-prestack}
Let $\BA$ be an arbitrary connective derived supercommutative $R$-algebra. A \textit{derived pre-superstack over $\BA$} is a presheaf $F\colon \mathrm{dSAff}^{\mathrm{op}}_{\BA}\to\mathrm{Spaces}$ that is \textit{accessible}: there exists a small regular cardinal $\kappa$ such that $F$ preserves filtered $\kappa$-colimits.

We denote by $\mathrm{dSPSt}_{\BA}$ the full sub-$\infty$-category of $\mathcal{P}(\mathrm{dSAff}_{\BA})$ spanned by derived pre-superstacks over $\BA$.
\end{definition}
The $\infty$-category $\mathrm{dSAff}_{\BA}$ admits Grothendieck topologies $\tau_{\text{fpqc}}$ and $\tau_{\text{ét}}$ that are the super-analog of the fpqc and étale topologies of commutative rings; see \cref{sec:grothendieck_topologies} for definitions and further details. In particular, it makes sense to consider sheaves with respect to those topologies.

\begin{definition}[Derived superstack]
\label{def:derived-superstack}
Let $\BA$ be a connective derived supercommutative algebra. A \textit{derived $\tau_\ast$-superstack} over $\BA$ is a derived pre-superstack which is a sheaf for the Grothendieck topology $\tau_\ast$, where $\ast\in\left\{\text{fpqc},\text{ét}\right\}.$ For the rest of the paper, we shall assume the choice of one of these two topologies and leave $\tau_\ast$ implicit; in particular, we shall call a derived $\tau_\ast$-superstack simply a \textit{derived superstack}.

We denote by $\mathrm{dSSt}_{\BA}\subseteq\mathrm{dSPSt}_{\BA}$ the full sub-$\infty$-category spanned by derived superstacks over $\BA$. 
\end{definition}
\begin{remark}
The following facts are well-known to practitioners of derived algebraic geometry; we still revisit them for the sake of the less accustomed reader.
\begin{remarkenum}
\item In the classical (not super/not derived) setting, the underlying functor of a stack is accessible if and only if the stack is locally of finite presentation. So, in some way, we are imposing further smallness conditions on our prestack so as to ensure the existence of a stackification functor.
\item Another way to ensure the existence of a fpqc stackification functor is to consider the $\infty$-category of presheaves of \textit{large} spaces $\widehat{\mathcal{P}}(\mathrm{dSAff}_{\BA})\defeq\mathrm{Fun}(\mathrm{dSAff}_{\BA}^{\mathrm{op}},\widehat{\mathrm{Spaces}})$. Accessible presheaves then embed fully faithfully inside $\widehat{\mathcal{P}}(\mathrm{dSAff}_{\BA})$, and their image is contained in the sub-$\infty$-category of large presheaves whose fpqc-stackification is a sheaf of \textit{small} spaces.
\end{remarkenum}
For further inquiries on the matter, we redirect the curious reader to \cite[App.~A]{hesselholt_pstragowski_II} and to the more concise account in \cite[Remarks 3.1.6 and 3.1.7]{gregoric2025periodizationspectralstacks}.
\end{remark}
\begin{example}We provide some natural examples of derived superstacks.
\begin{exmpenum}
\item Let $\BS$ be a derived supercommutative $\BA$-algebra. Then the derived affine scheme $\Spec \BS$ represents a derived pre-superstack
\[
\begin{tikzcd}
\mathrm{Map}_{\mathrm{dSAff}_{\BA}}(-,\Spec\BS)\colon\mathrm{dSAff}^{\mathrm{op}}_{\BA}\arrow{r} & \mathrm{Spaces},
\end{tikzcd}
\]
that we shall call again $\Spec\BS$. This is a derived superstack for the fpqc topology (hence, for the étale topology): in particular, they are both subcanonical topologies. We argue as in \cite[Prop.~A.5.7.2]{sag}. The fact that $\Spec\BS$ sends colimits of derived affine superschemes to limits of spaces is clear, because of the natural equivalence
\[
\mathrm{Map}_{\mathrm{dSAff}_{\BA}}(-,\Spec\BS)\simeq\mathrm{Map}_{\mathrm{dSCAlg}_{\BA}}(\BS,-).
\]
In particular, $\Spec\BS$ sends finite coproducts of derived affine superschemes to products. Therefore, we are left to prove that any flat cover of derived affine superschemes $\varphi\colon\Spec \widetilde{\BT}\to\Spec\BT$ identifies $\Spec\BT$ as the colimit of the \v{C}ech nerve of $\varphi$. In other words, we need to prove that the morphism of derived supercommutative algebras
\[
\begin{tikzcd}
\BT\arrow{r} & \displaystyle \lim_{[n]\in\bm{\Delta}}\widetilde{\BT}\otimes^{\mathrm{super}}_{\BT}\cdots\otimes^{\mathrm{super}}_{\BT}\widetilde{\BT}
\end{tikzcd}
\]
is an equivalence if $\varphi$ is faithfully flat. Since the forgetful functor $\mathrm{dSCAlg}_{\BA}\to\mathbf{SD}_\infty(\BA)$ is conservative and preserves all limits, we can reduce ourselves to check that the above map is an equivalence of derived supermodules; but this claim is an immediate consequence of \cref{thm:flat_descent}.
\item As one could expect, any superstack as in \cref{sec:superstacks} is naturally a derived superstack. Indeed, the $\infty$-categorical Grothendieck construction of \cite[Thm.~2.2.1.2]{htt} provides an equivalence between fibrations in groupoids with descent data, and sheaves for the \'etale topology with values in $1$-truncated spaces. In particular, any superscheme or algebraic superspace is also a derived superstack.
\item If $\BCal X$ and $\BCal Y$ are derived superstacks, then their product $\BCal X\times\BCal Y$ is a derived superstack. Indeed, sheaves for a Grothendieck topology on a small $\infty$-site are always closed under arbitrary limits inside all presheaves.
\item The $\infty$-category $\mathrm{dSSt}_{\BA}$ is not an $\infty$-topos in general (it is ``too large'', being not presentable); but it is up to enlarging our universe of choice. In particular, if $\BCal X$ and $\BCal Y$ are derived superstacks, we have a derived superstack of morphisms
\[
\begin{tikzcd}
\smash{\underline{\mathbb{R}\mathrm{Hom}}}(\BCal X,\BCal Y)\colon\mathrm{dSAff}^{\mathrm{op}}_{\BA}\arrow{r} & \mathrm{Spaces}.
\end{tikzcd}
\]
This property is inherited by any slice $\infty$-category of $\mathrm{dSSt}_{\BA}$ as well. Namely, let $\BCal{W}$ be an arbitrary derived superstack. Then, for any $\BCal X\to\BCal{W}$ and $\BCal Y\to\BCal{W}$ derived superstacks over $\BCal{W}$ there exists a derived superstack of \textit{relative} morphisms
\[
\begin{tikzcd}
\smash{\underline{\mathbb{R}\mathrm{Hom}}}_{/\BCal{W}}(\BCal X,\BCal Y)\colon(\mathrm{dSAff}_{/\BCal{W}})^{\mathrm{op}}\arrow{r} & \mathrm{Spaces}.
\end{tikzcd}
\]
\end{exmpenum}
\end{example}
We now turn to the relationship between derived stacks and derived superstacks over ordinary commutative rings.
\begin{construction}[{See \cite[$\S3.6$]{dang_derived_supergeometry}}]
Fix a (not super) derived commutative ring $R$. Then we have an ambidextrous adjunction
\[
\adjunction{(-)_0}{\mathbf{D}_\infty(R)}{\mathbf{SD}_\infty(R)}{(-)^{\mathrm{even}},}
\]
where $(-)_0$ is a fully faithful embedding. This functor views a derived $R$-module as a $R$-supermodule in pure even degree, so it preserves all limits and colimits. Its right adjoint considers only the even part of a derived $R$-supermodule. Since $(-)_0$ is strongly monoidal, such right adjoint is lax monoidal, so the above adjunction lifts to the level of commutative algebras
\[
\adjunction{(-)_0}{\mathrm{dCAlg}_R}{\mathrm{dSCAlg}_R}{(-)^{\mathrm{even}}.}
\]
The inclusion still preserves all limits and colimits, so it admits a left adjoint itself. Passing to the $\infty$-categories of prestacks (and, then, of stacks), we obtain adjunctions
\[
\adjunction{\iota}{\mathrm{dSt}_R}{\mathrm{dSSt}_R}{(-)_{\mathrm{bos}}}\quad\text{ and }\quad\adjunction{\mathrm{L}}{\mathrm{dSSt}_R}{\mathrm{dSt}_R}{\iota.}
\]
The functor $\iota$ is, again, a fully faithful inclusion; its right adjoint agrees with the bosonic reduction; and the left adjoint takes the even part of a derived superstack. If $X\simeq\Spec\BA$ is affine, then $(\Spec\BA)^{\mathrm{bos}}\simeq\Spec \BA^{\mathrm{bos}}$, where $\BA^{\mathrm{bos}}$ is the derived commutative ring of \cref{remark:bosonic_reduction}.
\end{construction}
We finish this small section with the definition of \textit{truncated} and \textit{derived enhancements} of superstacks, which will play a key role in our main result. This is a generalization of \cite[Ch.~$2$, $\S1.3$]{studyindag1} to the super-setting.
\begin{definition}
Let $\Spec\BS$ be a derived affine superscheme over an underived supercommutative ring $\BA$. We say that $\Spec\BS$ is $n$-coconnective for some integer $n\geqslant0$ if $\BH^k(\BS)\cong0$ for all $k\leqslant -n-1$. We let $\mathrm{dSAff}^{\geqslant-n}_{\BA}$ denote the full sub-$\infty$-category of derived affine superschemes spanned by $n$-coconnective derived superschemes.
\end{definition}
By their very definition, the Grothendieck topologies on derived affine superschemes restrict well to $n$-coconnective derived affine superschemes; this yields an $\infty$-category of \textit{$n$-truncated derived superstacks}, that we denote as $\mathrm{dSSt}_{\BA}^{\geqslant -n}$. For any integer $n\geqslant0$, we have an adjunction
\[
\adjunction{\iota_n}{\mathrm{dSSt}_{\BA}^{\geqslant-n}}{\mathrm{dSSt}_{\BA}}{\mathfrak{t}_n.}
\]
The left adjoint is fully faithful and produces a left Kan extension along the inclusion $\mathrm{dSCAlg}^{[-n,0]}_{\BA}\subseteq\mathrm{dSCAlg}^{\leqslant0}_{\BA}$.
\begin{definition}
Let $\BCal X$ be a derived superstack over an underived supercommutative algebra $\BA$. Its \textit{classical truncation} is the higher superstack $\mathfrak{t}_0\BCal X$.

Given a classical superstack $\BCal{X}$, we say that $\BCal X$ admits a \textit{derived enhancement} if there exists a derived superstack $\BR\BCal X$ and an equivalence $\mathfrak{t}_0\BR\BCal X\simeq \BCal X$. 
\end{definition}
\begin{remark}
\label{rmk:truncated-derived-enhanc}
We collect some useful facts regarding truncations and derived enhancements of superstacks, generalizing well-known results in derived algebraic geometry (see for example \cite[$\S4.3$]{MR2483943}). The proofs are formally identical, hence they are omitted.
\begin{remarkenum}
\item If a derived superstack $\BCal{X}\simeq\Spec\BS$ is affine, its truncation $\mathfrak{t}_n\BCal{X}$ is very easy to compute: it is just the affine superscheme corresponding to $\tau^{\geqslant-n}\BA$, where $\tau^{\geqslant-n}$ is the coconnective cover of the $t$-structure of \cref{prop:t-structure}. For $n=0$, in particular, it is an affine superscheme corresponding to the underived supercommutative algebra $\BH^0(\BS)$.
\item For any higher superstack $\BCal{X}\colon\mathrm{SCAlg}_{\BA}\to\mathrm{Spaces}$, its left Kan extension $\iota_0\BCal{X}$ is a trivial derived enhancement of $\BCal X$.
\item Since $\mathfrak{t}_0$ is a right adjoint, it follows that if we have underived superstacks $\left\{ \BCal{X}\right\}_{i\in I}$ then the natural map
\[
\begin{tikzcd}
\mathfrak{t}_0\iota_0\left(\lim_{i\in I}\BCal{X}_i\right)\arrow{r} & \mathfrak{t}_0\left(\lim_{i\in I}\iota_0(\BCal{X}_i)\right)
\end{tikzcd}
\]
is an equivalence; therefore, the homotopy limit $\lim_i \iota_0(\BCal{X}_i)$ is a derived enhancement of the ordinary limit $\iota_0(\lim_i \BCal{X}_i)$. Notice however that these two limits often fail to be equivalent: for example, in the case of a diagram of affine superschemes $\Spec\BA\to\Spec \BB\leftarrow\Spec \BC$, the homotopy limit computes the spectrum of the derived tensor product $\BA\otimes^{\mathrm{super}}_{\BB}\BC$, while the ordinary limit is the ordinary spectrum of the (underived) tensor product. The above morphism between limits is readily identified as the closed immersion corresponding to the projection
\[
\begin{tikzcd}
\BA\otimes^{\mathrm{super}}_{\BB}\BC\arrow{r} & \BH^0(\BA\otimes^{\mathrm{super}}_{\BB}\BC).
\end{tikzcd}
\]
\item \label{flat-implies-Hom} Let $\BCal{X}$ and $\BCal{Y}$ be underived superstacks. The superstack of morphisms $\iota_0\smash{\underline{\mathrm{Hom}}}(\BCal{X},\BCal{Y})$ is not equivalent to the \textit{derived} superstack $\smash{\underline{\mathbb{R}\mathrm{Hom}}}(\iota_0(\BCal{X}),\iota_0(\BCal{Y}))$. Since the former is defined in terms of classical (underived) fiber products while the latter is defined in terms of derived pullbacks, for general $\BCal{X}$ and $\BCal{Y}$ it is not even true that  $\smash{\underline{\mathbb{R}\mathrm{Hom}}}(\iota_0(\BCal{X}),\iota_0(\BCal{Y}))$ is a derived enhancement of $\iota_0\smash{\underline{\mathrm{Hom}}}(\BCal{X},\BCal{Y})$. However, this is true if $\BCal{X}$ is flat over the base. 
\end{remarkenum}
\end{remark}
\subsection{Geometric derived superstacks and representable morphisms}
Given a derived supercommutative algebra $\BA$, the $\infty$-category $\mathrm{dSSt}_{\BA}$ is still very large and---for most purposes---as wild as $\mathrm{dSPSt}_{\BA}$. In particular, for an arbitrary derived superstack its superderived $\infty$-category may be too large, or it may fail to have a well-defined cotangent complex, and so forth.

In \cite{hag_2}, To\"en and Vezzosi single out a class of \textit{geometric} derived stacks, generalizing the notion of Artin stacks from algebraic geometry. More recently, this notion has been extended to the supergeometric realm in \cite{hesselholt_pstragowski_II,dang_derived_supergeometry}, mirroring the original definition in derived algebraic geometry by To\"en--Vezzosi. We record the following piece of terminology, following \cite{studyindag2}.

\begin{definition}
Let $k\geqslant0$ be an integer, and let $\mathcal{P}$ be a property of morphisms of derived supercommutative algebras which is stable under composition, base change, and equivalences (e.g., $\mathcal{P}$ is flat, étale, smooth, disjoint union of Zariski open immersions, surjective...). Let $\BA$ be a derived supercommutative algebra, and let $f\colon\BCal X\to\BCal Y$ be a morphism of derived superstacks over $\BA$.
\begin{defenum}
\item The derived superstack $\BCal X$ is $0$-geometric if it is equivalent to an arbitrary (small) coproduct of derived affine superschemes:
\[
\BCal X\simeq\coprod_{i\in I}\Spec\BS_i.
\]
If $\BCal X$ is $0$-geometric and $\BCal Y\simeq\Spec\BT$ is derived affine, we say that $f$ has property $\mathcal{P}$ if for all $i\in I$ the composition $\Spec\BS_i\to\BCal Y$ has property $\mathcal{P}$.
\item  For general $\BCal Y$, the morphism $f$ is $0$-representable if for any morphism $\Spec\BA\to\BCal Y$ the fiber product $\Spec A\times_{ \BCal Y} \BCal X$ is $0$-geometric. In this case, we say that it has property $\mathcal{P}$ if the induced morphism $\Spec\BA\times_{ \BCal Y} \BCal X\to\Spec\BA$ has property $\mathcal{P}$ in the above sense.
\end{defenum}
Assume we have defined $n$-geometric derived superstacks, $n$-representable morphisms, and what does it mean for an $n$-representable morphism to have property $\mathcal{P}$, for $0\leqslant n \leqslant k-1$.
\begin{defenum}
\setcounter{defenumi}{2}
\item The derived superstack $\BCal X$ is $k$-geometric if the diagonal $\BCal X\to\BCal X\times\BCal X$ is $(k-1)$-representable, and there exists a $(k-1)$-representable morphism from a $(k-1)$-geometric derived superstack $\BCal Z\to\BCal X$ that is smooth and surjective (as a $(k-1)$-representable morphism). We call it a \textit{smooth atlas} for $\BCal X$. If $\BCal X$ is $k$-geometric and $\BCal Y$ is derived affine, we say that $f$ has property $\mathcal{P}$ if the above (equivalently, any) smooth atlas $\BCal Z$ has property $\mathcal{P}$.
\item For general $\BCal Y$, the morphism $f$ is $k$-representable if for any morphism $\Spec\BA\to\BCal Y$ the fiber product $\Spec A\times_{\BCal Y}\BCal X$ is $k$-geometric. In this case, we say that it has property $\mathcal{P}$ if the induced morphism $\Spec\BA\times_{\BCal Y}\BCal X\to\Spec\BA$ has property $\mathcal{P}$.
\end{defenum}
\end{definition}

\begin{definition}
\label{def:geometricity}
We say that a derived superstack $\BCal X$ is \emph{geometric} if it is $k$-geometric for some $k\geqslant0$, and we say that it is \emph{locally geometric} if it is covered by a family of geometric derived open substacks.    
\end{definition}

\begin{remark}
By definition, if $\varphi\colon\BCal Z\to\BCal X$ is a smooth atlas for a derived superstack $\BCal X$, then it is also an étale covering. In particular, we have an equivalence of derived superstacks
\[
\BCal X\simeq\colim_{[n]\in\bm{\Delta}^{\mathrm{op}}}\mathrm{\check{C}}(\varphi)_n.
\]
This can be used to prove by induction that, indeed, geometric derived superstacks satisfy descent (hence they deserve the term ``superstack'' rather than ``prestack'').
\end{remark}

\begin{definition}
Let $\BCal X$ be a $1$-geometric derived superstack. We say that $\BCal X$ is a \textit{derived superscheme} if the $0$-geometric atlas $\BCal Z\to\BCal X$ is a disjoint union of Zariski open immersions.
\end{definition}
\begin{remark}
By definition, if $\BCal X$ is a derived superscheme then the classical truncation $\mathfrak{t}_0\BCal X$ is a superscheme and $\mathfrak{t}_0\BCal Z$ is a Zariski atlas. In virtue of the topological invariance of \cref{thm:topological_invariance}, it follows that the converse holds as well: if $\BCal X$ is a geometric derived superstack whose underlying classical superstack $\mathfrak{t}_0\BCal X$ is a superscheme, then the Zariski atlas uniquely lifts to a Zariski atlas of $\BCal X$. More generally, \cref{thm:topological_invariance} implies that for a geometric derived superstack there is an equivalence of $\infty$-groupoids
\[
\mathrm{dSSt}^{\text{Zar}}_{/\BCal X}\simeq \mathrm{\SSt}^{\text{Zar}}_{/\mathfrak{t}_0\BCal X}
\]
between Zariski open derived substacks of $\BCal X$ and Zariski open substacks of $\mathfrak{t}_0\BCal X$. (See also \cite[Prop.~2.1]{Schurg_Toen_Vezzosi_Determinants_of_perfect_complexes} for the analogous statement in derived algebraic geometry.)
\end{remark}
\subsection{Superderived \texorpdfstring{$\infty$}{∞}-categories of derived superstacks}
In \cref{sec:homological_super_algebra}, we developed what essentially amounts to \textit{affine} derived algebraic supergeometry. In this section, we globalize the main definitions introduced in \cref{sec:homological_super_algebra} in order to make notions such as superderived $\infty$-categories, perfect supermodules and cotangent complexes meaningful for arbitrary derived super(pre)stacks. The main technical input is \cref{thm:flat_descent}, which states that the functor $\mathbf{SD}_\infty(-)$ satisfies descent for the flat topology of connective supercommutative algebras. 

Recall the functor \eqref{functor:derived_mod}, given by the assignment $\BA\mapsto\mathbf{SD}_\infty(\BA)$. It is a sheaf of symmetric monoidal presentable and $\BA$-linear $\infty$-categories; considering the right Kan extension along the Yoneda embedding $\mathrm{dSCAlg}_{\BA}^{\leqslant0}\subseteq\mathrm{dSPSt}_{\BA}$ yields a functor
\begin{equation}
\label{functor:derived_qcoh}
\begin{tikzcd}
\mathbf{SD}_\infty(-)\colon\mathrm{dSPSt}_{\BA}\arrow{r} & \smash{\widehat{\mathrm{Cat}}^{\mathrm{L},\otimes}_{\infty}}.
\end{tikzcd}
\end{equation}
Here, $\smash{\widehat{\mathrm{Cat}}^{\mathrm{L},\otimes}_{\infty}}$ denotes the $\infty$-category of (possibly large) $\infty$-categories admitting all colimits and a symmetric monoidal structure compatible with colimits separately in each variable, together with strongly monoidal and colimit-preserving functors. The reason why the target is not $\mathrm{Pr}^{\mathrm{L},\otimes}_{\BA}$ anymore is that, in general, this is not closed inside $\smash{\widehat{\mathrm{Cat}}^{\mathrm{L},\otimes}_{\infty}}$ under \textit{large} limits. 
\begin{definition}
\label{def:derived_qcoh}
Let $\BCal{X}$ be a derived pre-superstack. Its \textit{superderived $\infty$-category} $\mathbf{SD}_\infty(\BCal{X})$ is the image of $\BCal{X}$ under the functor \eqref{functor:derived_qcoh}. Explicitly, this is computed as the limit of presentable symmetric monoidal $\infty$-categories
\[
\mathbf{SD}_\infty(\BCal{X})\defeq\lim_{\Spec\BS\in\mathrm{dSAff}_{/\BCal{X}}^{\mathrm{op}}}\mathbf{SD}_\infty(\BS).
\]
We call an object of $\mathbf{SD}_\infty(\BCal{X})$ a \textit{derived quasicoherent sheaf} over $\BCal{X}$.
\end{definition}
By definition, for any derived pre-superstack $\BCal{X}$ the superderived $\infty$-category $\mathbf{SD}_\infty(\BCal{X})$ is stable, admits all limits and colimits, and it is equipped with a tensor product that commutes with colimits separately in each variable. Limits, colimits and tensor products in $\mathbf{SD}_\infty(\BCal{X})$ are computed locally on all sections over derived affine superschemes $\Spec\BA\to\BCal{X}$. Moreover, since the obvious forgetful functor $\smash{\widehat{\mathrm{Cat}}^{\mathrm{L},\otimes}_{\infty}}\to \smash{\widehat{\mathrm{Cat}}_\infty}$ is conservative and preserves all limits, \cref{def:derived_qcoh} formalizes the intuition that a derived quasicoherent sheaf over a derived pre-superstack $\BCal{X}$ is a family of derived $\BA$-supermodules parametrized by all $\BA$-rational points of $\BCal{X}$, together with a homotopy-coherent system of cocycle conditions allowing us to glue local sections over arbitrary (derived) intersections. 
\begin{remark}
Flat descent for supermodules immediately implies the following facts concerning derived quasicoherent sheaves over derived superstacks.
\begin{remarkenum}
\item If $\BCal{X}$ is a derived pre-superstack with stackification $\BCal{X}^\dagger$, then the canonical functor
\[
\begin{tikzcd}
\mathbf{SD}_\infty(\BCal{X}^\dagger)\arrow{r} & \mathbf{SD}_\infty(\BCal{X})
\end{tikzcd}
\]
is an equivalence. 
\item If $\Spec\BA$ is a derived affine superscheme, then obviously $\mathbf{SD}_\infty(\BA)\simeq\mathbf{SD}_\infty(\Spec\BA)$. More generally, if we have an étale cover $\varphi\colon\coprod_{i\in I}\Spec \BA_i\to\BCal{X}$ then the canonical functor
\[
\begin{tikzcd}
\mathbf{SD}_\infty(\BCal{X})\arrow{r} & \displaystyle\lim_{[n]\in\bm{\Delta}}\mathbf{SD}_\infty(\mathrm{\check{C}}(\varphi)_n)
\end{tikzcd}
\]
is an equivalence.
\item If $\BCal{X}$ is a derived pre-superstack generated under small colimits by the essential image of the Yoneda embedding $\mathrm{dSAff}_{\BA}\subseteq\smash{\widehat{\mathrm{dSSt}}}_{\BA}$, the superderived $\infty$-category $\mathbf{SD}_\infty(\BCal{X})$ is actually a \textit{presentable} $\infty$-category (i.e., the limit in \cref{def:derived_qcoh} is still an object in $\mathrm{Pr}^{\mathrm{L},\otimes}_{\BA}$). This is true for example for all geometric derived superstacks.
\item By definition, $\mathbf{SD}_\infty(-)$ is a functor: therefore, for any morphism of derived pre-superstacks $f\colon\BCal{X}\to\BCal{Y}$ there is an induced pullback functor $f^\ast\colon\mathbf{SD}_\infty(\BCal{Y})\to\mathbf{SD}_\infty(\BCal{X})$ that commutes with all colimits and preserves the monoidal structure. By abstract nonsense, if both derived pre-superstacks are actually geometric superstacks (or, more generally, they are such that the superderived $\infty$-category is presentable), then the functor $f^\ast$ admits a right adjoint $f_\ast$.
\item\label{Dperf} Since the symmetric monoidal structure of $\mathbf{SD}_\infty(\BCal{X})$ is described by its behaviour over all local sections, it follows that fully dualizable objects of $\mathbf{SD}_\infty(\BCal{X})$ are the same as quasicoherent sheaves whose pullback along any morphism $\Spec\BA\to \BCal{X}$ is fully dualizable (i.e., perfect) over $\BA$. We call the full sub-$\infty$-category spanned by fully dualizable objects the \textit{$\infty$-category of perfect complexes} over $\BCal{X}$, and denote it as $\mathrm{SPerf}(\BCal{X})$. Notice that in general a perfect complex over a derived pre-superstack may not be compact (in the sense of \cref{prop:compact}). 
\item\label{free-rank-global}Just like in the affine case, for any derived superstack $\BCal X$ and integers $m,n\geqslant1$, it makes sense to define a \textit{free sheaf of rank $m|n$} over $\BCal X$ as an object $\mathcal{E}$ inside $\mathbf{SD}_\infty(\BCal X)$ such that
\[
\mathcal{E}\simeq\mathscr{O}_{\BCal X}^{\oplus m}\oplus\Pi \mathscr{O}_{\BCal X}^{\oplus n}.
\]
In virtue of the descent for superderived $\infty$-categories and of \cref{prop:rank-descent}, the condition of being free of rank $m|n$ can be tested locally on $\BCal X$. Similarly, we can define \textit{locally free sheaves} of rank $m|n$ over $\BCal X$ as objects $\mathcal{E}\in\mathbf{SD}_\infty(\BCal X)$ such that for every morphism $\varphi\colon\Spec\BA\to\BCal X$ the pullback $\varphi^\ast\mathcal{E}$ is locally free of rank $m|n$ in the sense of \cref{def:free-rank}.
\item For any derived super(pre)stack $\BCal{X}$ for which $\mathbf{SD}_\infty(\BCal{X})$ is presentable, it makes sense to define a $t$-structure in the following way: a quasicoherent sheaf $\mathcal{F}\in\mathbf{SD}_\infty(\BCal{X})$ is connective if and only if for every morphism $x\colon\Spec\BA\to\BCal{X}$ the pullback $x^\ast\mathcal{F}$ is connective for the canonical $t$-structure on $\mathbf{SD}_\infty(\BA)$ of \cref{prop:t-structure}. This condition defines an accessible sub-$\infty$-category $\mathbf{SD}_\infty(\BCal{X})^{\leqslant0}\subseteq\mathbf{SD}_\infty(\BCal{X})$ closed under colimits, hence an accessible $t$-structure. If $\BCal{X}$ admits an étale cover $\varphi\colon\coprod_i\Spec \BA_i\to\BCal{X}$ (for example, if $\BCal{X}$ is geometric), then an object $\mathcal{F}\in\mathbf{SD}_\infty(\BCal{X})$ is coconnective if and only if for all $\varphi_i\colon\Spec \BA_i\to\BCal{X}$ the pullback $\varphi_i^\ast\mathcal{F}$ is coconnective for the $t$-structure on $\mathbf{SD}_\infty(\BA_i)$ of \cref{prop:t-structure}. This follows from the fact that, by definition, if $\eta\colon\Spec\BS\to\BCal{X}$ is flat then the pullback functor $\eta^\ast$ is $t$-exact. Moreover, in this case, the $t$-structure on $\mathbf{SD}_\infty(\BCal{X})$ is both left and right complete, and the heart of the $t$-structure agrees with the usual abelian category of quasicoherent sheaves over the superstack $\mathfrak{t}_0\BCal X$.
\end{remarkenum}
\end{remark}

\begin{remark}
For underived algebraic stacks, the lisse-\'etale site is the most common choice in order to build a well-behaved (triangulated) derived category of quasicoherent sheaves. Yet, it lacks good functoriality properties. That is, given a morphism of algebraic stacks $f\colon\CX\to\CY$ this does not induce a geometric morphism between the corresponding lisse-\'etale topoi. However, in order to glue cotangent supercomplexes of derived supercommutative algebras covering a derived superstack, it is crucial that morphisms of derived superstacks induce geometric morphisms between the corresponding $\infty$-topoi: this is why we define the derived $\infty$-category of derived superstacks using the fpqc/\'etale topology. This is not an issue: as proved in \cite[Ch.~$1$, $\S2.4$]{studyindag1}, when $\CX$ is a classical algebraic stack which is quasi-compact and quasi-separated, then the triangulated homotopy category of the derived stable $\infty$-category of $\CX$ defined via the \'etale site and the usual derived category defined via the lisse-\'etale as defined in \cite{LMB} agree, at least when restricting to eventually coconnective objects. 
\end{remark}
\subsection{Cotangent complexes and perfect obstruction theories}
\label{sec:cotangent_pot}
Recall from \cref{section:cotangentsupercomplex} that for any derived affine superscheme $\Spec \BA$ 
there is a cotangent complex $\BL_\BA$ which classifies infinitesimal deformations of $\Spec \BA$. This is a well-defined object in the superderived $\infty$-category $\mathbf{SD}_\infty(\BA)$. Since the functor $\mathbf{SD}_\infty(-)$ is a sheaf for the flat topology of supermodules (\cref{thm:flat_descent}), using the formalism of \textit{local almost representable functors} (\cite[$\S17.2.1$--$17.2.3$ and Ex.~17.2.4]{sag}), for any derived superstack we can define its cotangent complex as follows.
\begin{definition}
\label{def:cotangent-global}
Let $\BCal X$ be a derived superstack over a base $\BA$. The \textit{cotangent complex of $\BCal X$} is, if it exists, an object in $\mathbf{SD}_\infty(\BCal{X})$ satisfying the following universal properties:
\begin{enumerate} \item for every
morphism $x\colon\Spec\BA\to \BCal X$ there exists a derived $\BA$-module $\BL_{\BCal X,x}$ representing the functor
\[
\begin{tikzcd}
\operatorname{fib}_x(\BCal X(\BA \oplus(-)) \arrow{r} & \BCal X(\BA)) 
\end{tikzcd}
\]
\item for every diagram
\[
\begin{tikzcd}
\Spec \BB\arrow[swap]{dr}{y}\arrow{rr}{\phi} && \Spec \BA\arrow{dl}{x} \\
& \BCal X
\end{tikzcd}
\]
the canonical map $\phi^*\BL_{\BCal X,x}\to \BL_{\BCal X,y}$ is an isomorphism. 
\end{enumerate}
If such an object exists, these conditions allow us to glue the local cotangent complexes to a well-defined quasicoherent sheaf, unique up to a contractible space of choices.
\end{definition}
For any morphism of derived superstacks $f\colon\BCal X\to\BCal Y$ admitting cotangent complexes we have a relative cotangent complex $\BL_{f}$ which is the mapping cone of the morphism $f^\ast\BL_{\BCal Y}\to\BL_{\BCal X}$. Moreover, the results in \cref{section:cotangentsupercomplex} can be globalized in order to obtain deformation-theoretic descriptions and connectivity estimates on the cotangent complexes, see also \cite[$\S17$]{sag}.

This discussion allows us to give the definition of a \emph{perfect obstruction theory} in the setting of supergeometry.
\begin{definition}
\label{def:superPOT}
Let $f\colon \BCal X \to \BCal Y$ be a morphism of superstacks admitting cotangent complexes. An \emph{obstruction theory} relative to $f$ is the datum of an object $\BE\in \mathbf{SD}_\infty(\BCal X)$ along with a morphism $\psi\colon \BE \to \BL_{f}$ such that
\begin{enumerate}
    \item [\mylabel{OT-i}{(i)}] $\BH^0(\psi)$ is an isomorphism of $\OO_{\BCal X}$-modules, and
    \item [\mylabel{OT-ii}{(ii)}] $\BH^{-1}(\psi)$ is a surjection of $\OO_{\BCal X}$-modules.
\end{enumerate}
An obstruction theory is said to be \emph{perfect} if, moreover, $\BE$ is perfect of cohomological amplitude $[-1,0]$. If $\BCal{X}$ is defined over an underived base commutative superalgebra $\BA$, an obstruction theory on the structure morphism $\BCal{X} \to \Spec \BA$ is simply called an obstruction theory on $\BCal X$.

Finally, the \emph{virtual dimension} of a perfect obstruction theory $\BE \to \BL_{\BCal X}$ is defined to be the rank of $\BE$, in the sense of \Cref{free-rank-global}.
\end{definition}

\begin{theorem}
\label{thm:cotg-exists-and-induced-pot}
Let $\BCal X$ be a derived superstack. Then, the following statements hold true. 
\begin{enumerate}
\item [\mylabel{cotg-exists}{\normalfont{(1)}}] If $\BCal{X}$ is locally geometric (cf.~\Cref{def:geometricity}), then $\BCal{X}$ has a global cotangent complex $\BL_{\BCal{X}}$ in the sense of \Cref{def:cotangent-global}.
\item [\mylabel{induced-POT}{\normalfont{(2)}}] Assume $\BCal{X}$ is an algebraic superstack defined over an underived supercommutative algebra $\BA$, and let $\BR\BCal X$ be a locally geometric derived enhancement of $\BCal X$ with closed immersion $\iota \colon \BCal X\into \BR\BCal X$. Then, the canonical morphism $\iota^*\BL_{\BR\BCal X}\to \BL_{\BCal X}$ is an obstruction theory on $\BCal X$.
\end{enumerate}
\end{theorem}

\begin{proof}
As for \ref{cotg-exists}: since $\BCal X$ is covered by geometric derived superstacks, we can consider a surjection $\BCal Z\to \BCal X$ given by (a disjoint union of) affine superschemes. Thus, the statement \ref{cotg-exists} is equivalent to the fact that affine superschemes admit cotangent supercomplexes (which is true in virtue of \cref{section:cotangentsupercomplex}), and that for a morphism of derived supercommutative $\BA$-algebras $\varphi\colon\BS\to\BT$ the object $\BL_{\BS/\BA}\otimes^{\mathrm{super}}_{\BS}\BT$ corepresents the functor
\[
\begin{tikzcd}
\operatorname{fib}_{\varphi}(\mathrm{Map}_{\mathrm{dSCAlg}_{\BA}}(\BS,\BT\oplus(-)) \arrow{r} &\mathrm{Map}_{\mathrm{dSCAlg}_{\BA}}(\BS,\BT) ).
\end{tikzcd}
\]
Equivalently, this is the functor sending a derived $\BT$-supermodule $\BM$ to the space of maps between derived supercommutative $\BA$-algebras augmented over $\BT$:
\[
\mathrm{Map}_{\mathrm{dSCAlg}_{\BA//\BT}}\left(\BS,\BT\oplus(-)\right).
\]
By the tensor-forgetful adjunction and the universal property of the cotangent supercomplex, for any derived $\BT$-supermodule we have a chain of equivalences
\begin{align*}
\mathrm{Map}_{\mathbf{SD}_\infty(\BT)}\left(\BL_{\BS/\BA}\otimes^{\mathrm{super}}_{\BS} \BT,\BM\right)&\simeq\mathrm{Map}_{\mathbf{SD}_\infty(\BS)}{\left(\BL_{\BS/\BA},\BM\right)}\\&
\simeq\mathrm{fib}_{\mathrm{id}}\left( \mathrm{Map}_{\mathrm{dSCAlg}_{\BA}}(\BS,\BS\oplus\BM)\to\mathrm{Map}_{\mathrm{dSCAlg}_{\BA}}(\BS,\BS)\right)
\\&\simeq\mathrm{Map}_{\mathrm{dSCAlg}_{\BA//\BS}}(\BS,\BS\oplus\BM).
\end{align*}
Now we simply observe that the derived supercommutative algebra $\BS\oplus\BM$ is just the image of $\BT\oplus\BM$ under the right adjoint in the adjunction
\[
\adjunction{\varphi\circ(-)}{\mathrm{dSCAlg}_{\BA//\BS}}{\mathrm{dSCAlg}_{\BA//\BT}}{(-)\times_{\BT}\BS,}
\]
therefore
\begin{align*}
\mathrm{Map}_{\mathbf{SD}_\infty(\BT)}\left(\BL_{\BS/\BA}\otimes^{\mathrm{super}}_{\BS} \BT,\BM\right)&\simeq\mathrm{Map}_{\mathrm{dSCAlg}_{\BA//\BS}}(\BS,\BS\oplus\BM)\\
&\simeq\mathrm{Map}_{\mathrm{dSCAlg}_{\BA//\BT}}(\BS,\BT\oplus\BM),
\end{align*}
as desired.

The statement \ref{induced-POT} can be deduced from testing the statement locally, in which case the relevant statement is \cref{thm:obstruction-theory-on-truncation}.
\end{proof}

\begin{definition}
\label{def:quasi-smooth}
We say that a derived superstack $\BCal X$ is \emph{quasi-smooth} if its cotangent complex $\BL_{\BCal X}$ is perfect of Tor amplitude contained in $[-1,0]$.
\end{definition}

We immediately obtain the following corollary of \Cref{thm:cotg-exists-and-induced-pot}\ref{induced-POT}. 

\begin{corollary}
\label{cor:quasi-smooth-pot}
A derived enhancement $\BR\BCal X$ of an algebraic superstack $\BCal X$ is quasi-smooth if and only if the canonical obstruction theory $\iota^*\BL_{\BR\BCal X}\to \BL_{\BCal X}$ is perfect.
\end{corollary}

\section{Moduli of stable supermaps}
\label{sec:stable-supermaps}
\subsection{Prestable and stable supercurves}
In this section we recall the definition of the superstacks
\[
\Mf^{\mathsf{st-}\mathrm{SUSY}}_{g,\nf}\subset \Mf_{g,\nf}^{\mathsf{\pre-}\mathrm{SUSY}},
\]
parametrising stable SUSY curves and prestable SUSY curves, respectively. See \cite[Def.~2.21]{BrHRMa23} for the notion of a Cohen--Macaulay morphism in the super setting. 

\begin{definition}
\label{def:prestablecurve} 
Fix $g \in \BZ_{\geq 0}$ and a pair of non-negative integers $\nf =(n_{\mathrm{NS}},n_{\mathrm{RR}})$. Fix a superscheme $\boldit{S}$.

\begin{enumerate}
\item [\mylabel{supercurve}{(i)}] A \emph{supercurve} of arithmetic genus $g$ over a superscheme $\Sc$ is a proper and Cohen--Macaulay morphism $f\colon \Xcal \to \Sc$ of superschemes of relative dimension $1|1$ whose bosonic fibers have arithmetic genus $g$.  

\item [\mylabel{prestable-supercurve}{(ii)}] A \emph{prestable supercurve} of arithmetic genus $g$ with punctures over a superscheme $\Sc$ is a supercurve $f\colon \Xcal \to \Sc$, together with:
\begin{enumerate}
\item a collection of disjoint closed sub-superschemes $\boldit{W}_i\hookrightarrow \U$ ($i=1,\dots,n_{\mathrm{NS}}$), where $\U$ is the smooth locus of $f$, such that $f \colon \boldit{W}_i\to\Sc$ is an isomorphism for every $i$. These superschemes are called   \emph{Neveu--Schwarz (NS) punctures};
\item a collection of disjoint Cartier divisors $\Zc_j$ of relative degree 1 ($j=1,\dots,n_{\mathrm{RR}}$), contained in $\U$. They are called   \emph{Ramond--Ramond (RR) punctures}. 
\end{enumerate}
These data must satisfy the following condition: if for every bosonic fiber $X_s$ of $f\colon \Xcal \to \Sc$ we write $x_{s,i}=W_i\cap X_s$ and $z_{s,j}=Z_j\cap X_s$,  then the pair $(X_s, D_s)$, where   
\[
D_s= \{x_{s,1}\dots x_{s,n_{\mathrm{NS}}},z_{s,1}\dots, z_{s,n_{\mathrm{RR}}}\},
\]
is a prestable $(n_{\mathrm{NS}}+n_{\mathrm{RR}})$-pointed curve of 
arithmetic genus $g$ (in the sense of \Cref{sec:classical-stable-maps}).
\end{enumerate}
\end{definition}

\begin{definition}[{\cite[Def.~5.8]{BrHRMa23}}]
\label{def:susycurve}
A \emph{prestable SUSY curve} of arithmetic genus $g$ with punctures over a superscheme $\Sc$ is a  prestable supercurve $(f\colon\Xcal\to\Sc,\set{\boldit{X}_i},\set{\boldit{Z}_j})$ of arithmetic genus $g$, as in \Cref{def:prestablecurve}\ref{prestable-supercurve}, together with
an epimorphism $\overline{\delta}\colon \Omega_f\to \omega_f(\Zc)$ of $\Oc_{\Xcal}$-modules, where $\Zc=\sum \Zc_j$. In other words, the composition
\[
\begin{tikzcd}
\ker\overline{\delta} \arrow[hook]{r} &  
\Omega_f \arrow{r}{d} & 
\Omega_f\wedge \Omega_f \arrow{rr}{\overline{\delta}\wedge\overline{\delta}} &&
\omega_f^{\otimes 2}(2\Zc)
\end{tikzcd}
\]
yields an isomorphism
\[
\begin{tikzcd}
\ker\overline{\delta} \arrow{r}{\sim} & \omega_f^{\otimes 2}(\Zc).
\end{tikzcd}
\]
\end{definition}

In the case where the prestable SUSY curve is smooth, the isomorphism $\ker \overline{\delta} \simto \omega_f^{\otimes 2}(\boldit Z)$ is equivalent to a Ramond--Ramond conformal structure, see \cite[Prop.~3.5]{BrHR21}.

We now move to the definition of \emph{stable SUSY curves}, following \cite{Felder-Kazhdan-Polishchuk}. 
If  $f\colon \Xcal\to\Sc$ is a prestable SUSY curve, there is a one-to-one correspondence
between NS punctures and effective Cartier divisors
(as in \cite[Def.~2.13]{Felder-Kazhdan-Polishchuk}
and following lines); if the NS puncture is $P$
and the corresponding Cartier divisor is locally given by a superfunction $f$, then the ideal of $P$ is generated by $(f, D(f))$, where $D$ is a local generator of the superconformal structure. 
So the  NS punctures $\Wc_i\hookrightarrow \Xcal$ will be identified with effective Cartier divisors, which we shall denote by the same symbol.

Moreover, if $\varepsilon\colon \Ucal\hookrightarrow \Xcal$ is the inclusion of the smooth locus of $f\colon \Xcal\to\Sc$, the sheaf
\[
\omega_{\Xcal/\Sc}^2\defeq  \varepsilon_\ast \varepsilon^\ast \omega_{\Xcal/\Sc}^{\otimes 2}=\varepsilon_\ast (\omega_{\Ucal/\Sc}^{\otimes 2})\,,
\]
where $\omega_{\Xcal/\Sc}=\omega_f$ is the relative dualizing sheaf, is locally free, i.e., it  is a  {line bundle} \cite[Thm.~4.1]{Felder-Kazhdan-Polishchuk}.

\begin{definition}
\label{def:stablecurve} 
A \emph{stable SUSY curve with punctures} is a prestable SUSY curve $f\colon \Xcal\to\Sc$ with punctures, such that the line bundle
\[
\omega_{\Xcal/\Sc}^2(\Wc+\Zc)\,,
\]
where $\Wc=\sum_i\Wc_i$ and $\Zc=\sum_j\Zc_j$, is relatively strongly ample\footnote{See \cite[Sec.~2.5]{BrHRPo20} or \cite[Def.~4.2]{Felder-Kazhdan-Polishchuk} for such a notion of relative ampleness.} with respect to $f\colon\Xcal\to\Sc$.
\end{definition}

Contrary to what happens for stable SUSY curves, a prestable SUSY curve $f\colon \Xcal\to\Sc$ may not be (super)projective. We will see, however, that the prestable SUSY curve supporting a stable supermap is always (super)projective. For this reason, we restrict ourselves to projective prestable SUSY curves.

It is straightforward to check that \emph{projective} prestable SUSY curves of genus $g$ with fixed number of punctures form a groupoid fibration over $\SSchs_{\et}$, that we denote by $\preSUSY$.   

\begin{prop}
\label{prop:prestable-are-superstack}
The category fibered in groupoids  $\preSUSY\to \SSchs_{\et}$ is a  superstack.
\end{prop}
\begin{proof}Isomorphisms between projective prestable SUSY curves are an \'etale sheaf because they are representable by superschemes (see \cite[Corol.\! 4.27]{BrHRPo20} and note that the noetherian hypothesis is not necessary). Since descent data for polarized superschemes are effective \cite[Prop.\! B.63]{BHRSS25}, it is easily seen that descent data for projective prestable SUSY curves are effective as well.   
\end{proof}

It is, once again, straightforward to observe that stable SUSY curves of genus $g$ with fixed number of punctures form a groupoid fibration over $\SSchs_{\et}$, that we denote by $\Mf^{\mathsf{st-}\mathrm{SUSY}}_{g,\nf}$. In fact, since stability is an open condition,\footnote{A SUSY curve is stable if and only if  the underlying bosonic curve with
$n_{\mathrm{NS}}+n_{\mathrm{RR}}$ punctures is stable (see \cite[Prop.\! A.2]{Felder-Kazhdan-Polishchuk}). Since for ordinary curves stability is an open condition, the same is true for SUSY curves.} $\Mf^{\mathsf{st-}\mathrm{SUSY}}_{g,\nf}$ is easily seen to be an open sub-superstack of $\Mf^{\mathsf{pre-}\mathrm{SUSY}}_{g,\nf}$, that is there is an open immersion 
\[
\begin{tikzcd}
\Mf^{\mathsf{st-}\mathrm{SUSY}}_{g,\nf}
\arrow[hook]{r}{} & \Mf^{\mathsf{pre-}\mathrm{SUSY}}_{g,\nf}. 
\end{tikzcd}
\]

\begin{prop}
[{\cite[Thm.~A]{Felder-Kazhdan-Polishchuk}}]
\label{prop:moduli}  
The category fibered in groupoids $p\colon \Mf^{\mathsf{st-}\mathrm{SUSY}}_{g,\nf}\to \SSchs_{\et}$ 
of stable SUSY curves is a  proper and smooth Deligne--Mumford superstack.
\qed
\end{prop}

\begin{prop}
\label{prop:prestable-are-smooth}
The superstack  $\preSUSY\to \SSchs_{\et}$ of the projective prestable SUSY curves is  smooth and algebraic.
\end{prop}

\begin{proof} 
The technique is similar to the one employed in the classical case, see \cite[Lemma 2.1]{Bae-Schmidt} and \cite[Prop.~2]{Behrend_GW}.

Consider the universal projective prestable SUSY curve $\pi: \mathcal{C}^{\mathsf{pre-}\mathrm{SUSY}}_{g,\nf} \to \preSUSY$. 
The morphism $\pi$ is flat, projective, and superschematic. Let $\mathcal{C}^{\text{sm}} \subset \mathcal{C}^{\mathsf{pre-}\mathrm{SUSY}}_{g,\nf} $ be the locus where $\pi$ is smooth. Since the singular locus of the fibers is closed, $\mathcal{C}^{\text{sm}}$ is an open substack. Consequently, the restricted morphism $\pi^{\text{sm}}: \mathcal{C}^{\text{sm}} \to  {\preSUSY}$ is superschematic, smooth and surjective.

For any integer $m \ge 1$ consider the $m$-fold fiber product
\[ \mathcal{C}^{\text{sm}, m} = \mathcal{C}^{\text{sm}} \times_\preSUSY \dots \times_\preSUSY \mathcal{C}^{\text{sm}} \]
Let $\mathcal{U}_m \subset \mathcal{C}^{\text{sm}, m} $ be the open substack defined by excluding the pairwise diagonals.\footnote{As the universal curve is a representable morphism (actually superschematic), the relative diagonal morphisms are injective \cite[Prop.\! 2.20]{BHRSS25}. Moreover, the universal curve is separated, and then the relative diagonal morphisms are proper \cite[Def.\! 4.3]{BHRSS25}. It follows that they are closed immersions, so that the pairwise diagonals are closed sub-superstacks, and their complement $\mathcal{U}_m$  is an open substack of $\mathcal{C}^{\text{sm}, m}$. }
For every superscheme $S$, the $S$-points of $\mathcal{U}_m$ corresponds to projective prestable SUSY curves over $S$ equipped with $n_{\mathrm{NS}}$ NS punctures, $n_{\mathrm{RR}}$ RR punctures 
and $m$ disjoint sections lying in the smooth locus.
The   projection map $p_m\colon \mathcal{U}_m \to \preSUSY$, which  forgets the  $m$ additional sections, is a  superschematic smooth morphism, as shown in the subsequent Lemma \ref{lem:proj-is-smooth}.

Now we regard the $m$ sections as NS punctures.
Let $\mathcal{V}_m \subset \mathcal{U}_m$ be the locus where the SUSY curves, equipped with $n_{\mathrm{NS}}+m$ NS punctures and
$n_{\mathrm{RR}}$ RR punctures, are stable. Since stability is an open condition, $\mathcal{V}_m$ is an open sub-superstack of $\mathcal{U}_m$. Thus, the restricted projection $p_m|_{\mathcal{V}_m}\colon  \mathcal{V}_m \to \preSUSY$ is superschematic and smooth. Note that  $\mathcal{V}_m$ is an open sub-superstack of  a superstack of stable SUSY curves, hence it is algebraic and smooth by \cite{Felder-Kazhdan-Polishchuk}.
 
To conclude, note that the family 
\[
\left\{\begin{tikzcd}
p_m|_{\mathcal{V}_m} \colon \mathcal{V}_m \arrow{r} & \preSUSY
\end{tikzcd}\right\}_{m \ge 1}
\]
covers $\preSUSY$.
Let  $(\boldit{C}, p_1, \dots, p_{n_{\mathrm{NS}}}, r_1,\dots,r_{n_{\mathrm{RR}}})$ be a geometric point in $\preSUSY$.  By choosing $m$ sufficiently large and placing $m$ distinct NS punctures $q_1, \dots, q_m$ in the smooth loci of the components that violate the stability condition, the SUSY curve 
\[
(\boldit{C}, p_1, \dots, p_{n_{\mathrm{NS}}},q_1, \dots, q_m, r_1,\dots,r_{n_{\mathrm{RR}}})
\]
becomes stable.

Since $\preSUSY$ is covered by the images of the smooth morphisms   originating from algebraic smooth superstacks, it is algebraic and smooth.
 \end{proof}

 Later we shall need to use the sheaf $\mathcal A$ of infinitesimal deformations of a projective prestable punctured SUSY curve. A close analysis of the theory developed in \cite{Felder-Kazhdan-Polishchuk} shows that this sheaf has the same expression as the one for stable SUSY curves, the difference between the two situations being of global nature. Denoting by
 $\mathcal A_+$ and $\mathcal A_-$ its even and odd parts, we have
\begin{equation}\label{eq:sheafA} \mathcal A_+ = T_C\left(-\sum_iZ_{\mathrm{NS},i}-\sum_jZ_{\mathrm{RR},j}\right),
 \qquad \mathcal A_- = T_C\otimes L \left(-\sum_iZ_{\mathrm{NS},i}-\sum_jZ_{\mathrm{RR},j}\right),
\end{equation}
 where the line bundle $L$ is the spin structure given by $L^2=\omega_C\big(\sum_jZ_{RR,j}\big)$ see \cite[Cor.~3.10]{Felder-Kazhdan-Polishchuk}).

\begin{lemma} 
\label{lem:proj-is-smooth}
The projection map $p_m\colon \mathcal{U}_m \to \preSUSY$, which  forgets the $m$ additional sections, is a superschematic smooth morphism.
\end{lemma}

\begin{proof} The projection $p_m$ is clearly superschematic. 
Its smoothness  is a direct consequence of the   stability of smoothness under base change and under composition, and the fact that open immersions are smooth.   Recall that  $\pi^{\text{sm}} \colon \mathcal{C}^{\text{sm}} \to  {\preSUSY}$ is   smooth by definition.

 We now consider the case $m=2$. The projection map $\mathrm{pr}_1 \colon \mathcal{C}^{\text{sm}} \times_\preSUSY  \mathcal{C}^{\text{sm}} \to \mathcal{C}^{\text{sm}}$ is the base change of the smooth morphism $\pi^{\text{sm}} \colon \mathcal{C}^{\text{sm}} \to\preSUSY$ along $\pi^{\text{sm}}$ itself, hence it  is smooth. The morphism $\mathcal{C}^{\text{sm}, 2}\to \preSUSY $ is the composition of the smooth morphisms  $\mathrm{pr}_1$ with $\pi^{\text{sm}}$, hence it is smooth. By   induction on $m$ the   projection $\mathcal{C}^{\text{sm}, m}_{\mathfrak{M}_{g,n}} \to \preSUSY$ is smooth.

Finally one uses the fact that open immersions are smooth.
\end{proof}

\subsection{Stable supermaps}
In this Subsection we describe the moduli superstack of stable supermaps, following \cite[Sec.~2.1]{BHRSM25}.

\subsubsection{Chow groups of superschemes}
As in the case of stable maps, the definition of stable supermap involves fixing a 1-cycle class in the target. A theory of cycles for superschemes has been developed in \cite{BrHRMa23}. We briefly review the basic theory in this subsection.

Let us write $\mathbb{Z}^2$ as $\mathbb{Z}\oplus\Pi\mathbb{Z}$ (here $\Pi$ is the parity change operator), so that it has a superring structure with product:
\[
(m-\Pi n)(m'-\Pi n') = (mm'+nn' - \Pi (mn'+m'n)).
\]
An $h$-supercycle on a superscheme $\Xcal$ is a finite sum
\[
\alpha={\textstyle \sum_i} (m_i -  {\Pi }n_i) [Y_i]
\]
where $m_i-\Pi n_i \in\mathbb{Z}^2$ and the $Y_i$ are closed $h$-dimensional subvarieties of the bosonic reduction $X$ of $\Xcal$. The group of supercycles of $\Xcal$ is the bigraded $\Z^2$-module
\[
Z_\bullet (\Xcal)=\bigoplus_{h=0}^m Z_h(\Xcal) = Z_\bullet (X) \oplus \Pi Z_\bullet (X),
\]
where $m$ is the dimension of $X$.
There is a natural notion of linear equivalence for supersycles, so that we can define the Chow group 
\[
A_h(\Xcal)= Z_h(\Xcal)/W_h(\Xcal)
\]
where $W_h(\Xcal)\subset Z_h(\Xcal)$ is the  graded $\mathbb{Z}^2$-submodule formed by the 
 $h$-supercycles rationally equivalent to zero.

One can prove \cite[Def.~4.17]{BrHRMa23} that pushforward morphisms $f_\ast\colon A_\bullet(\Xcal) \to A_\bullet(\Ycal)$ exist for every proper superscheme morphism  $f\colon \Xcal \to \Ycal$.

\subsubsection{The moduli superstack of stable supermaps}
We now fix nonnegative integers $g$, $n_{\mathrm{NS}}$ and $n_{\mathrm{RR}}$, a smooth projective superscheme $\boldit{Y}$ with a polarization $\Lcl_{\Ycal}$, and a rational equivalence class of 1-supercycles $\beta\in A_1(\Ycal)$.

\begin{definition}
\label{def:stablemap2}  
A stable supermap $F$ to $\Ycal$ over a superscheme $\Sc$ of class $\beta$ with $n_{\mathrm{NS}}$ NS punctures and $n_{\mathrm{RR}}$ RR punctures is given by the following set of data:
\begin{enumerate}
\item [\mylabel{st1}{(1)}] A prestable SUSY curve $(f\colon\Xcal\to\Sc, \{\boldit{W}_i\},\{\Zc_j\}, \overline{\delta})$ over $\Sc$ with $n_{\mathrm{NS}}$ Neveu--Schwarz punctures and $n_{\mathrm{RR}}$ Ramond--Ramond punctures, as in \Cref{def:susycurve}.
\item [\mylabel{st2}{(2)}] A superscheme morphism $\phi\colon\Xcal\to \Ycal$ such that $\phi_\ast[\Xcal_s]=\beta$ for every closed point $s\in \boldit S$.
\end{enumerate}
These data are subject to the following stability condition:
\begin{enumerate}
\item [\mylabel{st3}{(3)}] the even line bundle
\[
\Lcl_F\defeq \omega_{\Xcal/\Sc}^2\left(\sum_{1\le i\le n_{\mathrm{NS}}} \boldit X_i + \sum_{1\le j\le n_{\mathrm{RR}}} \Zc_j\right)\otimes \phi^\ast \Lcl_{\Ycal}^{\otimes 3}
\]
is  strongly relatively ample with respect to $f\colon\Xcal\to\Sc$.\footnote{See \cite[Subsect.~2.5]{BrHRPo20} or \cite[Def.~4.2]{Felder-Kazhdan-Polishchuk} for the definition of strongly relatively ample.}
\end{enumerate}
\end{definition}

When $\Ycal$ reduces to a point, we recover the definition of stable SUSY curve with punctures. 

\begin{definition}[{\cite[Def.\! 6.5]{BrHRMa23}, \cite[Def.\! 2.8]{BHRSM25}}]
\label{def:CFGsmaps}  
The category fibered in groupoids 
\[
\begin{tikzcd}
\Mf^{\mathsf{st}-\mathrm{SUSY}}_{g,\nf}(\Ycal,\beta) \arrow{r}{p} & \SSchs_{\et}
\end{tikzcd}
\]
of $\beta$-valued stable supermaps of arithmetic genus $g$ with $n_{\mathrm{NS}}$ Neveu--Schwarz  punctures and $n_{\mathrm{RR}}$ Ramond--Ramond punctures into $\Ycal$ 
is given by the following data: 
\begin{itemize}
\item Objects are $\beta$-valued stable supermaps $\Xf=((f\colon\Xcal\to\Sc, \{\boldit{W}_i\},\{\Zc_j\}, \overline{\delta}), \phi)$ over a superscheme $\Sc$ into $\Ycal$, of arithmetic genus $g$, with $n_{\mathrm{NS}}$ Neveu--Schwarz punctures and $n_{\mathrm{RR}}$ Ramond--Ramond punctures,\footnote{We use a simplified notation here; some of the relevant  data are left implicit.} just as in \Cref{def:stablemap2}. 
\item Morphisms are cartesian diagrams
\begin{equation}
\label{eq:cartesian}
\begin{tikzcd}[row sep=large,column sep=large]
\boldit{X}'\MySymb{dr}\arrow{r}{\Xi}\arrow{d}{g} & \boldit{X}\arrow{d}{f} \\
\boldit{T}\arrow{r}{\xi} & \boldit{S}
\end{tikzcd}
\end{equation}
that is, diagrams inducing an isomorphism $\Xcal'\simto \xi^\ast \Xcal\defeq \Xcal\times_{\Sc}\Tc$ of superschemes over $\Tc$, compatible with the punctures and the morphisms $\overline{\delta}$, and such that $\phi'=\phi\circ\Xi$. 
\end{itemize}
The functor $p$ maps a stable supermap to the base superscheme and the morphism $\Xi$ to the base morphism $\xi$. The pullback $\xi^\ast\boldit{X}\to\Tc$ is given by the fiber product $f_{\Tc}\colon\Xcal\times_{\Sc}\Tc\to\Tc$ and the fiber products of the data $\{\boldit{W}_i\},\{\Zc_j\}, \overline{\delta}$ (thus providing  a natural ``cleavage'').\footnote{The fiber product of $\overline{\delta}$ is a derivation of the same kind due to \cite[Prop.\! 2.26]{BrHRMa23}.} 
\end{definition}

\begin{theorem}[{\cite[Thm.\! 2.17]{BHRSM25}}]
\label{thm:DMstack}
The category fibered in grupoids of stable supermaps
\[
\begin{tikzcd}
\Mf^{\mathsf{st}-\mathrm{SUSY}}_{g,\nf}(\Ycal,\beta) \arrow{r}{p} & \SSchs_{\et}
\end{tikzcd}
\]
is a Deligne--Mumford superstack.
\end{theorem}

One can also define an algebraic superstack
\[
\Mf_{g,\nf}^{\pre-\mathrm{SUSY}}(\Ycal,\beta)
\]
whose groupoid of $\boldit S$-objects, for a superscheme $\boldit S$, is given by conditions \ref{st1}--\ref{st2} in \Cref{def:stablemap2}, but ignoring the stability condition \ref{st3}. This stack is not of Deligne--Mumford type. There is a diagram of superstacks
\[
\begin{tikzcd}
\Mf^{\mathsf{st-}\mathrm{SUSY}}_{g,\nf}(\Ycal,\beta)\arrow[hook]{r}{\mathrm{open}}\arrow[swap]{dr}{\boldit{\tau}} &\Mf_{g,\nf}^{\pre-\mathrm{SUSY}}(\Ycal,\beta)\arrow{d} \\
&\Mf_{g,\nf}^{\mathsf{\pre-}\mathrm{SUSY}}
\end{tikzcd}
\]
where the vertical morphism forgets the datum of the map $\boldit{X} \to \boldit{Y}$, and the target is smooth by \Cref{prop:prestable-are-smooth}. We view the diagonal map $\boldit{\tau}$ as the (super) analog of the forgetful map $\tau \colon \Mbar \to \mathfrak M_{g,n}$ recalled in \Cref{sec:classical-stable-maps}.

We now move to the derived setup.
Let $\BCal{C}^{\pre-\mathrm{SUSY}}_{g,\nf} \to \BCal{M}^{\pre-\mathrm{SUSY}}_{g,\nf}$ be the universal curve. Since this morphism is flat, by \Cref{flat-implies-Hom} the derived superstack
\[
\BR\Mf_{g,\nf}^{\pre-\mathrm{SUSY}}(\Ycal) \defeq \BR\underline{\Hom}_{/\BCal{M}^{\pre-\mathrm{SUSY}}_{g,\nf}}(\BCal{C}^{\pre-\mathrm{SUSY}}_{g,\nf},\BCal{M}^{\pre-\mathrm{SUSY}}_{g,\nf} \times \boldit{Y})
\]
defines a derived enhancement of 
\[
\Mf_{g,\nf}^{\pre-\mathrm{SUSY}}(\Ycal)\defeq \coprod_{\beta} \Mf_{g,\nf}^{\pre-\mathrm{SUSY}}(\Ycal,\beta),
\]
where $\beta$ ranges over all possible choices of 1-supercycles of the form $(1-\Pi)\beta_0$, where $\beta_0\in A_1(\boldit{Y}_{\bos})$.
For an arbitrary derived affine superscheme $\boldit{S}$, its $\infty$-groupoid of $\boldit{S}$-valued points is described by the mapping space
\[
\mathrm{Map}_{/\BCal{M}^{\pre-\mathrm{SUSY}}_{g,\nf}\times\boldit{S}}(\BCal{C}^{\pre-\mathrm{SUSY}}_{g,\nf} \times \boldit{S},\BCal{M}^{\pre-\mathrm{SUSY}}_{g,\nf} \times \boldit{Y}\times\boldit{S}).
\]
Here, by $\times$ we mean a \textit{derived} pullback of derived superstacks. Being a relative derived hom-superstack, it comes with a structure morphism
\begin{equation}
\label{eq:structuremap}
\begin{tikzcd}
\pi\colon \BR\Mf_{g,\nf}^{\pre-\mathrm{SUSY}}(\Ycal) \arrow{r} & 
\Mf_{g,\nf}^{\pre-\mathrm{SUSY}}.
\end{tikzcd}
\end{equation}
In virtue of the topological invariance of \Cref{thm:topological_invariance}, the Zariski open substack 
\[
\Mf_{g,\nf}^{\mathsf{st-}\mathrm{SUSY}}(\Ycal) \subset \Mf_{g,\nf}^{\mathsf{pre-}\mathrm{SUSY}}(\Ycal)
\]
lifts to a Zariksi open immersion of derived superstacks 
\[
\BR\Mf_{g,\nf}^{\mathsf{st-}\mathrm{SUSY}}(\Ycal) \subset \BR\Mf_{g,\nf}^{\mathsf{pre-}\mathrm{SUSY}}(\Ycal).
\]
By abuse of notation we shall still denote by $\pi$ the natural structure morphism 
\[
\begin{tikzcd}
\BR\Mf_{g,\nf}^{\mathsf{st-}\mathrm{SUSY}}(\Ycal) \arrow{r} &
\Mf_{g,\nf}^{\pre-\mathrm{SUSY}}.
\end{tikzcd}
\]

\subsection{\texorpdfstring{Perfect obstruction theory for moduli of stable supermaps}{}}
The goal of this section is to prove the following result.

\begin{theorem}
\label{thm:pot-on-maps}
Let $\boldit Y$ be a smooth projective superscheme. Fix $g\geq 0$, $\beta \in A_1(\boldit Y)$ and a pair of non-negative integers $\nf = (n_{\mathrm{NS}},n_{\mathrm{RR}})$. Then, the Deligne--Mumford superstack $\Mf^{\mathsf{st-}\mathrm{SUSY}}_{g,\nf}(\Ycal,\beta)$ admits a canonical perfect obstruction theory.
\end{theorem}

\begin{proof}
We will first compute the cotangent complex of its derived enhancement $\BR\Mf^{\mathsf{st-}\mathrm{SUSY}}_{g,\nf}(\Ycal,\beta)$ and then show that it enjoys the desired finiteness conditions of \Cref{def:superPOT}.

In order to compute this cotangent complex, we will study the local cotangent complex at an $\boldit{S}$-point using the map \eqref{eq:structuremap}. Let $x\colon\Spec \mathbb{A} \to \BR\Mf_{g,\nf}^{\mathrm{SUSY}}(\Ycal)$ classify a stable supermap $(\BCal C, \boldit{Z}_{\mathrm{NS}},\boldit{Z}_{\mathrm{RR}},\delta,\varphi)$. Pre-composing with $\pi$ amounts to forgetting the stable map $\varphi$;
we shall denote by $\widetilde{\BCal{C}}$
the corresponding set of data.
By the transitivity triangle for the cotangent complex we have an exact triangle
\[
\begin{tikzcd}
\pi(x) ^\ast \BL_{\preSUSY} \arrow{r} &  x^\ast\BL_{\BR\Mf^{\mathsf{st-}\mathrm{SUSY}}_{g,\nf}(\Ycal)} \arrow{r} &  x^\ast \BL_\pi.
\end{tikzcd}
\]
The cotangent complex of $\Mf_{g,\nf}^{\pre-\mathrm{SUSY}}$ at the point $\pi(x)$ is computed by the dual of the sheaf $\mathcal{A}$ of infinitesimal deformations of the prestable SUSY curve
$\widetilde{\BCal{C}}$ (see \Cref{eq:sheafA})
 shifted by $-1$, i.e., $\pi(x)^\ast \BL_{\preSUSY}\simeq \Gamma_\ast (\mathcal A)^\vee[-1]$.

Let us compute the cotangent complex of $\pi$ at the point $x$. The fiber of the morphism $\pi$ at the point $\pi(x)$ is the derived hom-superstack
$\BR\underline{\Hom}_{\SSt}(\boldit{C},\boldit{Y})$, whose cotangent complex at the point $x$, since the target $\boldit Y$ is smooth, is given by $\Gamma_\ast(\varphi^\ast T_{\boldit{Y}})^\vee$: this follows from adapting the argument in \cite[Prop.~4.4.2]{DerivedHilbert} to the super-setting. Thus, the cotangent complex $x^\ast \BL_{\BR\Mf^{\mathsf{st-}\mathrm{SUSY}}_{g,\nf}(\Ycal)}$ is computed by the cone
\begin{equation}
\label{eq:cone}
\begin{tikzcd}
\mathrm{Cone}(
\Gamma_\ast(\varphi^\ast T_{\boldit{Y}})^\vee \arrow{r} & 
\Gamma_\ast (\mathcal A)^\vee) {[}-1{]}
\end{tikzcd}
\end{equation}
whose map is induced by the sequence
\[
\begin{tikzcd}
\mathcal A \arrow[hook]{r} &   T_{\boldit C} \arrow{r} & \varphi^\ast T_{\boldit Y}.
\end{tikzcd}
\]
We need to check that this defines a perfect obstruction theory. Note that, as $\boldit Y$ is smooth and projective over $\BA$,
the derived $\BA$-supermodules $\Gamma_\ast(\varphi^\ast T_{\boldit{Y}})^\vee$
and $\Gamma_\ast (\mathcal A)^\vee$ 
are perfect supermodules concentrated in degrees $-1$ and 0, so that the same applies to the cone in \Cref {eq:cone}. Note in particular that   $\Gamma_\ast (\mathcal A)^\vee[-1]$ is concentrated in degree 0 as its $\HH^0(\mathcal A)$ computes the Lie   algebra  of the automorphism group, that for a stable supermap is finite and reduced --- see \cite{BHRSM25}.

By \Cref{thm:cotg-exists-and-induced-pot}\ref{induced-POT}, the closed immersion
\[
\begin{tikzcd}
\Mf^{\mathsf{st-}\mathrm{SUSY}}_{g,\nf}(\Ycal)\arrow[hook]{r}{\iota} & \BR\Mf^{\mathsf{st-}\mathrm{SUSY}}_{g,\nf}(\Ycal)
\end{tikzcd}
\]
yields a perfect obstruction theory
$\iota^\ast \BL_{\BR\Mf^{\mathsf{st-}\mathrm{SUSY}}_{g,\nf}(\Ycal)}\to 
\BL_{\Mf^{\mathsf{st-}\mathrm{SUSY}}_{g,\nf}(\Ycal)}$.
\end{proof}
 
\subsection{The virtual dimension of \texorpdfstring{$\SUSY$}{Mstgn(X,)}}
By definition, the virtual dimension of the moduli superstack $\SUSY$ is the Euler characteristic of the cone in \Cref{eq:cone}. It is computed using the super Grothendieck--Riemann--Roch theorem proved by Manin, Penkov and Voronov in \cite{VMP} (see also \cite{BHRSM25} for a streamlined exposition). 

This computation was made in \cite{BHRSM25}. The computation of the cone in \Cref{eq:cone} made in the present paper, therefore, supports the formula given  in \textit{loc. cit.} which was only conjectural, as it was not based on a suitable deformation theory. 

We report here the resulting formula (\Cref{eq:dimmod}); details are available in \cite{BHRSM25}.  The formula we are going to give involves a sheaf $\mathcal F_{\boldit Y}$ on the scheme $Y = \boldit{Y}_{\mathrm{bos}}$; this is defined as $\mathcal F_{\boldit Y} = \Pi \mathcal J/\mathcal J^2$, where $\mathcal J$ is the ideal sheaf in $\OO_{\boldit Y}$ generated by the odd elements of $\OO_{\boldit Y}$ (note that as
$\mathcal F_{\boldit Y}$ is annihilated by $\mathcal J$, it may be regarded as a sheaf on $Y$). Actually, up to parity, $\mathcal F_{\boldit Y}$ is the odd part of the restriction of the tangent bundle of $\boldit Y$ to $Y$.

\begin{theorem}
\label{thm:vdim-maps}
The virtual  dimension of the moduli superstack   $\SUSY$ is 
 \begin{multline}
 \label{eq:dimmod}
\mathrm{vdim } \SUSY 
  = (r-3)(1-g) + n_{\mathrm{NS}}+ n_{\mathrm{RR}} (1+ s/2) \\  - \Pi \left[ (1-g)(s-2)+ n_{\mathrm{NS}}+ (n_{\mathrm{RR}}/2) (r+1)\right]  
  + (1-\Pi)\int_{\beta_0} [ \ch_1(T_Y)-\ch_1(\mathcal F_{\boldit Y})],
\end{multline}
where $r\vert s = \dim \boldit Y$. Separating the even and odd dimensions for clarity,
   \begin{align*}
   \mathrm{even\ vdim \ } \SUSY &= (r-3)(1-g) + n_{\mathrm{NS}}+ n_{\mathrm{RR}} (1+ s/2)
   + \int_{\beta_0} [ \ch_1(T_Y)-\ch_1(\mathcal F_{\boldit Y})], \\
   \mathrm{odd\ vdim\  } \SUSY &=  (1-g)(s-2)+ n_{\mathrm{NS}}+ (n_{\mathrm{RR}}/2) (r+1)
   + \int_{\beta_0} [ \ch_1(T_Y)-\ch_1(\mathcal F_{\boldit Y})].
   \end{align*}
\end{theorem}

\begin{remark} 
When the target $\boldit Y$ is bosonic (an ordinary scheme $Y$) then $s = \ch_1(\mathcal F_{\boldit Y}) =0$ and the previous formula reduces to
 \begin{multline*}  \operatorname{vdim } \SUSY    =   (1-g)(r-3) - \Pi(2g-2) + n_{\mathrm{NS}}+ n_{\mathrm{RR}} (1+ s/2) \\  - \Pi ( n_{\mathrm{NS}}+ (n_{\mathrm{RR}}/2) (r+1))
  + (1-\Pi) \int_{\beta_0}  \ch_1(T_Y), 
  \end{multline*}
which agrees with the formula that in \cite[Thm.~4.3.8]{KSY-2022} was 
conjectured in the case $g=n_{\mathrm{RR}}=0$, i.e.,
\[
\mathrm{vdim\ } \Mf^{\mathsf{st-}\hbox{\rm\tiny SUSY}}_{0,(n,0)}( Y,\beta_0) = r - 3 + n - \Pi (n-2) + (1-\Pi) \int_{\beta_0}  \ch_1(T_Y).
\]
 Moreover, further specializing to the ordinary moduli stack of stable maps into a scheme $Y$, so that
$n_{\mathrm{RR}}=0$ and $n_{\mathrm{NS}}$ is the number $n$ of punctures, one gets the usual formula for the virtual dimension 
\[
(r-3)(1-g) + n + \int_{\beta_0}  \ch_1(T_Y)  .
\]
\end{remark}

\appendix
\section{Proofs left behind}

\subsection{Proof of \cref{prop:t-structure}}
\label{sec:t-structure_proof}
Since the functor \eqref{functor:underlyingmodule} commutes with all colimits and limits, it is clear that $\mathbf{SD}_{\infty}(\BA)^{\leqslant0}$ is closed under colimits and the inclusion functor is accessible, so this defines the connective part of an accessible $t$-structure on $\mathbf{SD}_{\infty}(\BA)$. So we are only left to prove that the $\infty$-category $\mathbf{SD}_{\infty}(\BA)^{\geqslant0}$ is precisely the right orthogonal to $\mathbf{SD}_{\infty}(\BA)^{\leqslant0}$: in other words, we need to check that given a derived $\BA$-supermodule $\BN$, the underlying derived $R$-module of the derived $\BA$-supermodule $\BR\BH\mathrm{om}_{\BA}(\BM,\BN)$ is coconnective for all connective derived $\BA$-supermodules $\BM$ if and only if $\BN$ is coconnective in our sense. We first observe that, being $\BA$ connective, $\mathbf{SD}_{\infty}(\BA)^{\leqslant0}$ can be characterized as the smallest full subcategory of $\mathbf{SD}_{\infty}(\BA)^{\leqslant0}$ containing both $\BA$ and $\Pi\BA$ and closed under small colimits. Therefore, writing any connective derived $\BA$-supermodule $\BM$ as a colimit of positive shifts of $\BA$ and $\Pi\BA$, using the fact that the functor \eqref{functor:underlyingmodule} commutes with all limits and colimits, we can write
\begin{align*}
\BR\BH\mathrm{om}_{\BA}(\BM,\BN)&\simeq\BR\BH\mathrm{om}_{\BA}(\colim_{r,s,n_{rs}}\BA^{r|s}[n_{rs}],\BN)\\&\simeq\lim_{r,s,n_{rs}}\BR\BH\mathrm{om}_{\BA}(\BA^{r|s}[n_{rs}],\BN)^{\mathrm{even}}\\&\simeq\lim_{r,s,n_{rs}}\BR\BH\mathrm{om}_R(R^{r|s}[n_{rs}],\BN)^{\mathrm{even}}\\&\simeq\lim_{r,s,n_{rs}} (\BN^{\oplus r})^{\mathrm{even}}\oplus(\BN^{\oplus s})^{\mathrm{odd}}[-n_{rs}]
\end{align*}
where $r,s,n_{rs}\geqslant0$. Since taking cohomology commutes with all direct sums, such derived $R$-module is coconnective for all $r,s,n_{rs}\geqslant0$ if and only if this is true for $r=s=1$ and $n_{rs}=0$. But for this choice of integers this is precisely the underlying derived $R$-supermodule of $\BN$. The claims on the left and right completion of the $t$-structure and on the description of $\BH^0$ are then straight-forward. 

For the statement concerning the heart of the $t$-structure, we first notice that the statement is true when $\BA=R$ is the (ordinary) ground ring $R$: in this setting, $\mathbf{SD}_{\infty}(R)$ is by definition the derived $\infty$-category of the abelian category $\mathrm{SMod}_R$, and so the claim follows from \cite[Prop.~1.3.2.19]{ha}. For a general derived $R$-superalgebra, we observe that the functor $\BH^0\colon \mathbf{SD}_{\infty}(R)\to\mathrm{SMod}_R$ sending a derived $R$-supermodule $\BM$ to $\mathrm{H}^0(\BM^{\mathrm{even}})\oplus \mathrm{H}^0(\BM^{\mathrm{odd}})$ preserves algebra objects, and thus it defines a functor $\BH^0\colon\mathbf{SD}_{\infty}(\BA)\to\mathrm{SMod}_{\BH^0(\BA)}$. Restricting such functor to the heart of the $t$-structure, we obtain a well-defined functor $\BH^0\colon \mathbf{SD}_{\infty}(\BA)^{\heartsuit}\to\mathrm{SMod}_{\BH^0(\BA)}$. This is clearly essentially surjective. In order to prove that it is fully faithful, we observe that for any ordinary $\BH^0(\BA)$-supermodule (seen as a discrete derived $\BA$-supermodule), the subcategory of $\mathbf{SD}_{\infty}(\BA)$ consisting of derived $\BA$-supermodules $\BM$ for which the map
\[
\begin{tikzcd}
\tau^{\leqslant0}\BR\BH\mathrm{om}_{\BA}(\BM,\BN)\arrow{r} & \Hom_{\BH^0(\BA)}(\BH^0(\BM),\BN)
\end{tikzcd}
\]
is bijective is closed under colimits and contains $\BA$. Therefore, it contains the whole $\mathbf{SD}_{\infty}(\BA)^{\leqslant0}$, and in particular $\mathbf{SD}_{\infty}(\BA)^{\heartsuit}$.

We finally prove the claim concerning the projective rigidity of $\mathbf{SD}_\infty(\BA)$: since for any connective derived supercommutative $R$-algebra $\BA$ this is a $\infty$-category of modules for an algebra in $\mathbf{SD}_\infty(R)$, we can reduce ourselves to check the claim when $\BA=R$ is an ordinary commutative ring thanks to \cite[Prop.~3.17]{liang_higher_algebra}.
\begin{enumerate}
    \item The subcategory $\mathbf{SD}_\infty(R)^{\leqslant0}\simeq\mathbf{D}_\infty(R)^{\leqslant0}\times\mathbf{D}_\infty(R)^{\leqslant0}$ is accessible and contains the monoidal unit $R^{1|0}$. Moreover, since the monoidal structure of $\mathbf{SD}_\infty(R)$ is described in terms of direct sums of tensor products in $\mathbf{D}_\infty(R)$, connective objects are readily seen to be stable under tensor products.
    \item The $\infty$-category $\mathbf{SD}_\infty(R)^{\leqslant0}$ is projectively generated: one can check that the projective object $R\oplus \Pi R$ generates everything under colimits, using \cite[Prop.~7.1.4.12]{ha}.
    \item Dualizable and compact projective objects in $\mathbf{SD}_\infty(R)^{\leqslant0}$ coincide. This is an easy consequence of \cref{prop:superf} (implication \ref{prop:def} $\Leftrightarrow$ \ref{prop:compact}): see \cref{remark:last_part_thm_tstr}.
\end{enumerate}
\subsection{Proof of \cref{prop:superf}}
\label{sec:perf_proof}
We first show that \ref{prop:def} implies \ref{prop:compact}. Let $\BM$ be a fully dualizable object: then the functor
\[
\begin{tikzcd}
\mathbf{R}\mathbb{H}\mathrm{om}_{\BA}(\BM,-)\colon\mathbf{SD}_\infty(\BA)\arrow{r} & \mathbf{SD}_\infty(\BA)
\end{tikzcd}
\]
is equivalent to the functor $\BM^{\vee}\otimes^{\mathrm{super}}_{\BA}(-)$, and thus preserves all colimits. Forgetting the $\BA$-module structure preserves again all colimits, so we have a colimit-preserving functor
\[
\begin{tikzcd}
\mathbf{R}\mathbb{H}\mathrm{om}_{\BA}(\BM,-)^{\mathrm{even}}\oplus\mathbf{R}\mathbb{H}\mathrm{om}_{\BA}(\BM,-)^{\mathrm{odd}}\colon\mathbf{SD}_\infty(\BA)\arrow{r} & \mathbf{SD}_\infty(R)\simeq\mathbf{D}_\infty(R)\times\mathbf{D}_\infty(R).
\end{tikzcd}
\]
Let us remark that the above equivalence is an equivalence of plain (\textit{not monoidal}) stable $\infty$-categories. The functor $\mathrm{Map}_{\BA}(\BM,-)^{\mathrm{even}}$ is then obtained by projecting onto the first copy of $\mathbf{SD}_\infty(R)$, and finally post-composing with the ``forgetful'' functor
\[
\begin{tikzcd}
\Omega^\infty\colon\mathbf{D}_\infty(R)\arrow{r} & \mathrm{Spaces}.
\end{tikzcd}
\]
Since the functor $\Omega^\infty$ commutes with sifted colimits (in particular, with filtered ones), it follows that the functor $\mathrm{Map}_{\BA}(\BM,-)^{\mathrm{even}}$ commutes with filtered colimits.

Let us now prove that \ref{prop:compact} implies \ref{prop:def} (this is where we use the connectivity assumption on $\BA$). We first show that if the functor $\mathrm{Map}_{\BA}(\BM,-)^{\mathrm{even}}$ commutes with filtered colimits, then also the functor
\[
\begin{tikzcd}
\mathbf{R}\mathbb{H}\mathrm{om}_{\BA}(\BM,-)\colon\mathbf{SD}_\infty(\BA)\arrow{r} & \mathbf{SD}_\infty(\BA)
\end{tikzcd}
\]
must preserve them as well. Let $\left\{\BN_i\right\}_{i\in I}$ be a filtered system of derived $\BA$-supermodules; since the $t$-structure on $\mathbf{SD}_\infty(\BA)$ of \cref{prop:t-structure} is both left and right complete, it is sufficient to check that for all integers $n\in\BZ$ the map
\begin{equation}
    \label{eq:comparison}
    \begin{tikzcd}
    \BH^n\left(\colim_{i\in I}\mathbf{R}\mathbb{H}\mathrm{om}_{\BA}(\BM,\BN_i)\right)\arrow{r} & \BH^n\left(\mathbf{R}\mathbb{H}\mathrm{om}_{\BA}\left(\BM,\colim_{i\in I}\BN_i\right)\right)
    \end{tikzcd}
\end{equation}
is an isomorphism of $R$-supermodules. Note that, for all integers $n\in\BZ$ and for all derived $R$-supermodules, we have
\[
\BH^n(\BM)^{\mathrm{even}}\cong\pi_{0}\Omega^\infty(\BM[-n]) \quad\text{ and }\quad\BH^n(\BM)^{\mathrm{odd}}\cong\pi_{-n}\Omega^\infty(\Pi\BM[-n]).
\]
Since by assumption $\mathrm{Map}_{\BA}(\BM,-)^{\mathrm{even}}\simeq\Omega^\infty \mathbf{R}\mathbb{H}\mathrm{om}_{\BA}(\BM,-)^{\mathrm{even}}$ commutes with filtered colimits, and homotopy groups of spaces commute with filtered colimits, it follows that the map \eqref{eq:comparison} is an isomorphism for all integers $n\geqslant0$. In particular, the functor $\mathbf{R}\mathbb{H}\mathrm{om}_{\BA}(\BM,-)$ commutes with arbitrary small colimits (since it always commutes with finite colimits) and it preserves the derived supertensor product: indeed, writing two arbitrary derived $\BA$-supermodules as $\BN\simeq\colim_{\BN}\BA$ and $\BP\simeq\colim_{\BP}\BA$ and using the compatibility of the supertensor product with colimits, it follows that
\[
\mathbf{R}\mathbb{H}\mathrm{om}_{\BA}(\BM,\BN\otimes^{\mathrm{super}}_{\BA}\BP)\simeq\colim_{\BN}\colim_{\BP}\mathbf{R}\mathbb{H}\mathrm{om}_{\BA}(\BM,\BA)\simeq\mathbf{R}\mathbb{H}\mathrm{om}_{\BA}(\BM,\BN)\otimes^{\mathrm{super}}_{\BA}\mathbf{R}\mathbb{H}\mathrm{om}_{\BA}(\BM,\BP).
\]
Using the fact that colimit-preserving and $\mathbf{SD}_\infty(\BA)$-linear endofunctors of $\mathbf{SD}_\infty(\BA)$ are always given by tensoring with a derived $\BA$-supermodule \cite[Thm.~4.8.4.1]{ha} it follows that $\mathbf{R}\mathbb{H}\mathrm{om}_{\BA}(\BM,\BN)\simeq\mathbf{R}\mathbb{H}\mathrm{om}_{\BA}(\BM,\BA)\otimes^{\mathrm{super}}\BN$ for any derived supermodule $\BA$. Therefore the obvious evaluation and the coevaluation
\[
\begin{tikzcd}
\mathrm{coev}\colon\BA\arrow{r} & \mathbf{R}\mathbb{H}\mathrm{om}_{\BA}(\BM,\BA)\otimes^{\mathrm{super}}_{\BA}\BM\simeq\mathbf{R}\mathbb{H}\mathrm{om}_{\BA}(\BM,\BM)
\end{tikzcd}
\]
selecting the identity morphism of $\BM$ yield the desired dualizability data for $\BM$.

The fact that \ref{prop:thick} implies \ref{prop:compact} is clear, since $\BA$ and $\Pi\BA$ are both clearly dualizable (and actually self dual) and hence compact, so the smallest stable sub-$\infty$-category they generate under finite colimits and retracts only consists of compact objects. We now prove that all compact objects are contained in such $\infty$-category: let $\mathrm{Thick}(\BA,\Pi\BA)$ be the $\infty$-category described in \cref{prop:thick}. The inclusion $\mathrm{Thick}(\BA,\Pi\BA)\subseteq\mathbf{SD}_\infty(\BA)$ induces a fully faithful embedding $\mathrm{Ind}(\mathrm{Thick}(\BA,\Pi\BA))\subseteq\mathbf{SD}_\infty(\BA)$. If this were not essentially surjective, $\mathrm{Thick}(\BA,\Pi\BA)$ would produce a presentable and stable sub-$\infty$-category of $\mathbf{SD}_\infty(\BA)$ that is closed under colimits; therefore, it would define the connective part of a $t$-structure. If an object $\BN$ is coconnective for such $t$-structure, it follows that $\mathrm{Map}_{\BA}(\BM,\BN)\simeq\left\{\ast\right\}$ for all $\BM$ in the essential image of $\mathrm{Thick}(\BA,\Pi\BA)$. But this contains the objects $\BA[n]$ and $\Pi\BA[n]$ for all integers $n$, thus we obtain that $\BH^n(\BN)^{\mathrm{even}}$ and $\BH^n(\BN)^\mathrm{odd}$ are trivial for all integers $n$, so $\BN$ is actually the trivial derived supermodule. Since $\mathbf{SD}_\infty(\BA)$ is generated by its compact objects, it follows that $\mathrm{Thick}(\BA,\Pi\BA)\simeq\mathrm{SPerf}(\BA)$.

Finally, if $\BA$ is purely bosonic, then we have an equivalence
\[
\mathbf{R}\mathbb{H}\mathrm{om}_{\BA}(\BM,-)^{\mathrm{even}}\simeq\mathbf{R}\Hom_{\BA}(\BM^{\mathrm{even}},(-)^{\mathrm{even}})\oplus\mathbf{R}\Hom_{\BA}(\BM^{\mathrm{odd}},(-)^{\mathrm{odd}}).
\]
In particular, $\BM$ is compact in $\mathbf{SD}_\infty(\BA)$ if and only if both $\BM^{\mathrm{even}}$ and $\BM^{\mathrm{odd}}$ are compact in $\mathbf{D}_\infty(\BA)$, proving the equivalence between \ref{prop:compact} and \ref{prop:degreewise}.
\begin{remark}
\label{remark:last_part_thm_tstr}
Using \cite[Prop.~7.2.2.6]{ha}, the equivalence \cref{prop:def} $\Leftrightarrow$ \ref{prop:compact} implies that all dualizable objects in $\mathbf{SD}_\infty(\BA)^{\leqslant0}$ are compact \textit{and projective}. This concludes the last step in the proof of \cref{prop:t-structure}. Notice that no circularity is involved in our argument, since in order to prove the above equivalence we only use the existence of the left and right complete $t$-structure from \cref{prop:t-structure} (which is proved independently).
\end{remark}
\subsection{Proof of \cref{thm:flat_descent}}
\label{sec:flat_descent_proof}
Let $\left\{\BS\to\BT_i\right\}_i$ be a faithfully flat covering of connective derived supercommutative $\BA$-algebras. The functor $\mathbf{SD}_{\infty}(-)$ is a functor of the form $\mathrm{Mod}_{(-)}(\mathscr{C})$ for a stable presentable symmetric monoidal $\infty$-category $\mathscr{C}$; thus, we can easily adapt the proof of \cite[Lemma D.3.5.5]{sag} to the super setting and deduce that $\mathbf{SD}_{\infty}(-)$ preserves finite products. So, if $\left\{i_1,\ldots,i_n\right\}$ denotes the finite subset of indices yielding the desired faithfully flat morphism, we can replace the $\BT_i$'s with $\prod_{j=1}^n\BT_{i_j}$. Therefore, we are reduced to prove the following: let $\varphi\colon \BS\to\BT$ be a morphism of derived supercommutative $\BA$-algebras, and let $\check{\mathrm{C}}(\varphi)^{\bullet}$ denote its \v{C}ech nerve in the $\infty$-category $\mathrm{dSCAlg}_{\BA}$. In other words, this is the cosimplicial derived $\BA$-superalgebra whose $n$-th term is described by the $n$-fold tensor product $\BT\otimes_{\BS}^{\mathrm{super}}\cdots\otimes_{\BS}^{\mathrm{super}}\BT$. Then the functor
\[
\begin{tikzcd}
\mathbf{SD}_{\infty}(\BT)\arrow{r} & \displaystyle\lim_{[n]\in\bm{\Delta}^{\mathrm{op}}}\mathbf{SD}_{\infty}(\check{\mathrm{C}}(\varphi)_n)
\end{tikzcd}
\]
is an equivalence. Using the dual statement of \cite[Cor.~4.7.5.3]{ha}, we are reduced to prove the following.
\begin{enumerate}
    \item For any morphism $\alpha\colon[m]\to[n]$ in $\bm{\Delta}$, the natural diagram 
\[
\begin{tikzcd}
  \mathbf{SD}_{\infty}(\check{\mathrm{C}}(\varphi)_m)
    \arrow[r, "d^0"]
    \arrow[d, "\alpha^\ast"]
  &
  \mathbf{SD}_{\infty}(\check{\mathrm{C}}(\varphi)_{m+1})
    \arrow[d, "\alpha^\ast"]
  \\
  \mathbf{SD}_{\infty}(\check{\mathrm{C}}(\varphi)_n)
    \arrow[r, "d^0"]
  &
  \mathbf{SD}_{\infty}(\check{\mathrm{C}}(\varphi)_{n+1})
\end{tikzcd}
\]
is horizontally right adjointable. This is clear because the forgetful functor $\mathbf{SD}_{\infty}(\check{\mathrm{C}}(\varphi)_{m+1})\to\mathbf{SD}_{\infty}(\check{\mathrm{C}}(\varphi)_m)$ does the trick. 
\item The functor $\mathbf{SD}_{\infty}(\BS)\to\mathbf{SD}_{\infty}(\check{\mathrm{C}}(\varphi)_0)\simeq\mathbf{SD}_{\infty}(\BT)$ is comonadic. 
\end{enumerate}
By virtue of the Barr--Beck--Lurie monadicity theorem \cite[Thm.~4.7.3.5]{ha}, proving the comonadicity of the functor $(-)\otimes^{\mathrm{super}}_{\BS}\BT$ is equivalent to proving that it admits a right adjoint (which is trivially given by the forgetful functor along $\BS\to\BT$), that it reflects equivalences, and that it preserves totalizations of cosimplicial diagrams which become split\footnote{The general definition of split cosimplicial objects can be found in \cite[Definition 4.7.2.2]{ha}, but for the scope of our arguments we only need to consider split cosimplicial objects in abelian categories, which are equivalent to non-negatively graded cochain complexes that are acyclic resolutions of $\ker(d^0)$.} after base change to $\BT$. 
\begin{enumerate}
    \item We start by proving that the functor $(-)\otimes^{\mathrm{super}}_{\BS}\BT$ reflects equivalences. Let $f\colon\BM\to\BN$ be a morphism of derived $\BS$-supermodules, and assume that $f\otimes\mathrm{id}_{\BT}\colon\BM\otimes_{\BS}^{\mathrm{super}}\BT\to\BN\otimes_{\BS}^{\mathrm{super}}\BT$ is an equivalence. This is equivalent to saying that the mapping cone of $f$ is annihilated by tensoring with $\BT$. Consider the unit morphism $\eta_{\mathrm{cone}(f)}\colon\mathrm{cone}(f)\to \mathrm{cone}(f)\otimes^{\mathrm{super}}_{\BS}\BT$, and complete it to a triangle. By the compatibility of the derived tensor product with mapping cones, the mapping cone $\mathrm{cone}(\eta)$ is equivalent to $\mathrm{
    cone}(f)\otimes_{\BS}^{\mathrm{super}}\mathrm{cone}(\eta_{\BS})$. Moreover, $\eta_{\BS}$ naturally identifies with $\varphi$ itself. In other words, we obtain a triangle
    \[
    \begin{tikzcd}
    \mathrm{cone}(f)\arrow{r} & \mathrm{cone}(f)\otimes^{\mathrm{super}}_{\BS}\BT\arrow{r} & \mathrm{cone}(f)\otimes^{\mathrm{super}}_{\BS}\mathrm{cone}(\varphi).
    \end{tikzcd}
    \]
    In virtue of the naturality of the long exact sequence of cohomology supergroups, for all integers $n\in\BZ$ we obtain a commutative diagram
\[
\begin{tikzcd}[row sep=large, column sep=large]
  \BH^{n-1}\left(
    \mathrm{cone}(\varphi)
    \otimes^{\mathrm{super}}_{\BS}
    \mathrm{cone}(f)
  \right)
    \arrow[r, "\delta_n"]
    \arrow[d]
  &
  \BH^n\left(\mathrm{cone}(\varphi)\right)
    \arrow[d]
  \\
  \BH^{n-1}\left(
    \tau^{\geqslant n}\left(\mathrm{cone}(\varphi)\right)
    \otimes^{\mathrm{super}}_{\BS}
    \mathrm{cone}(f)
  \right)
    \arrow[r]
  &
  \BH^n\left(\tau^{\geqslant n}\mathrm{cone}(\varphi)\right),
\end{tikzcd}
\]
where $\tau^{\geqslant n}$ is understood as the truncation functor for the $t$-structure of \cref{prop:t-structure}. We claim that $\mathrm{cone}(\varphi)$ is itself a flat discrete $\BS$-superalgebra: if this was true, then the bottom left corner of the above diagram would be zero, the right vertical arrow would be an isomorphism, and thus every connecting morphism $\delta_n$ would be zero. So the long exact sequence in cohomology would actually split in short exact sequences
    \[
    \begin{tikzcd}[column sep=small]
    0\arrow{r} & \BH^{n}(\mathrm{cone}(f))\arrow{r} & \BH^n(\mathrm{cone}(f)\otimes_{\BS}^{\mathrm{super}}\BT)\arrow{r} & \BH^n(\mathrm{cone}(f)\otimes_{\BS}^{\mathrm{super}}\mathrm{cone}(\varphi))\arrow{r} & 0.
    \end{tikzcd}
    \]
    Since $f\otimes\mathrm{id}_{\BT}$ is an equivalence, the middle term is zero for any integer $n$, so the left hand side is zero as well. But this is equivalent to saying that $f$ was already an equivalence.    
    
    So, let us prove that $\mathrm{cone}(\varphi)$ is indeed $\BS$-flat. Using \cref{prop:flat3}, we need to prove that for a discrete $\BS$-supermodule $\BM$ the tensor product $\BM\otimes^{\mathrm{super}}_{\BA}\mathrm{cone}(\varphi)$ is again discrete. This is the mapping cone of the map $\BM\to\BM\otimes_{\BS}^{\mathrm{super}}\BT$, where both source and target are discrete because $\BT$ is flat over $\BS$. Thus, the only obstruction to $\mathrm{cone}(\varphi)$ being $\BS$-flat is
    \[
\begin{aligned}
\BH^{-1}\bigl(\mathrm{cone}(\varphi)
  \otimes_{\BS}^{\mathrm{super}}\BM\bigr)
&\cong \mathrm{ker}\left(
  \begin{tikzcd}[ampersand replacement=\&, cramped, column sep=small]
  \BM \arrow[r]
  \& \BM\otimes_{\BS}^{\mathrm{super}}\BT
  \end{tikzcd}
  \right)\\
&\cong \mathrm{ker}\left(
  \begin{tikzcd}[ampersand replacement=\&, cramped, column sep=small]
  \BM \arrow[r]
  \& \BM\otimes_{\BH^0(\BS)}^{\mathrm{super}}\BH^0(\BT)
  \end{tikzcd}
  \right),
\end{aligned}
\]
where the second isomorphism follows from flatness of $\BT$. But $\BH^0(\BT)$ is a flat $\BH^0(\BS)$-supermodule, so the map $\BM\to\BM\otimes^{\mathrm{super}}_{\BH^0(\BS)}\BH^0(\BT)$ is injective: indeed, it is the unit morphism for the adjunction
    \[
    \adjunction{(-)\otimes^{\mathrm{super}}_{\BH^0(\BS)}\BH^0(\BT)}{\mathrm{SMod}_{\BH^0(\BS)}}{\mathrm{SMod}_{\BH^0(\BT)}}{\mathrm{forget}},
    \]
    and since $\BH^0(\BT)$ is faithfully flat the left adjoint is faithful. 
    \item We now prove that the functor $(-)\otimes^{\mathrm{super}}_{\BS}\BT$ preserves totalizations of simplicial objects which become split in $\mathbf{SD}_{\infty}(\BT)$. We first observe that, since $BT$ is faithfully flat over $\BS$, the adjunction $(-)\otimes^{\mathrm{super}}_{\BS}\BT\dashv\mathrm{forget}$ induces an adjunction at the level of the heart
    \[
    \adjunction{(-)\otimes^{\mathrm{super}}_{\BH^0(\BS)}\BH^0(\BT)}{\mathrm{SMod}_{\BH^0(\BS)}}{\mathrm{SMod}_{\BH^0(\BT)}}{\mathrm{forget}}.
    \]
    Since the left adjoint preserves all equalizers and reflects equivalences (because $\BH^0(\BT)$ is faithfully flat over $\BH^0(\BS)$) it follows that it is monadic. So let $\BK^{\bullet}\colon\mathbf{\Delta}^{\mathrm{op}}_+\to\mathbf{SD}_{\infty}(\BS)$ be an augmented simplicial diagram such that $\BK^{\bullet}\otimes^{\mathrm{super}}_{\BS}\BT$ is split. Post-composing with the functors $\BH^n(-)$, this yields a split cosimplicial $\BH^0(\BS)$-supermodule $\BK^{\bullet}(n)$ which, under the Dold--Kan correspondence \cite[Thm.~1.2.3.7]{ha}, can be understood as a cochain complex of $\BH^0(\BT)$-supermodules
    \[
    \begin{tikzcd}
    \BK^{-1}(n)\cong\mathrm{ker}(d^0)\arrow{r} &  \BK^0(n)\arrow{r}{d^0} & \BK^1(n)\arrow{r} & \cdots.
    \end{tikzcd}
    \]
    Moreover, after base change over $\BH^0(\BT)$ this cochain complex becomes a split exact resolution of $\BK^{-1}(n)\otimes^{\mathrm{super}}_{\BH^0(\BS)}\BT$. But since $\BH^0(\BT)$ is faithfully flat over $\BH^0(\BS)$, it follows that $\BK^{\bullet}(n)$ was already a split exact resolution of $\BK^{-1}(n)$ for all integers $n\in\BZ$. Applying (the dual of) \cite[Cor.~$1.2.4.12$]{ha} we obtain that the limit of the diagram $\BK^{\bullet}$ in $\mathbf{SD}_{\infty}(\BA)$ is hence characterized by the fact that
\[
\BH^n\left(\lim_{\mathbf{\Delta}}\BK^{\bullet}\right)\cong \BK^{-1}(n)
\]
for all $n\in\BZ$, while the limit of the diagram $\BK^{\bullet}\otimes^{\mathrm{super}}_{\BS}\BT$ is characterized by the fact that
\[
\BH^n\left(\lim_{\mathbf{\Delta}}\left(\BK^{\bullet}\otimes^{\mathrm{super}}_{\BS}\BT\right)\right)\cong\BK^{-1}(n)\otimes^{\mathrm{super}}_{\BH^0(\BS)}\BH^0(\BT)
\]
for all $n\in\BZ$. In particular, because of faithfully flatness of $\BT$, the natural comparison map
\[
\begin{tikzcd}
\left(\displaystyle\lim_{\mathbf{\Delta}}\BK^{\bullet}\right)\otimes^{\mathrm{super}}_{\BS}\BT\arrow{r} & \displaystyle\lim_{\mathbf{\Delta}}\left(\BK^{\bullet}\otimes^{\mathrm{super}}_{\BS}\BT\right)
\end{tikzcd}
\]
induces an isomorphism on all cohomology supergroups, hence it is an equivalence.
\end{enumerate}
\begin{remark}
Let $\BA$ be an underived supercommutative $R$-algebra, and let $\varphi\colon \BS\to\BT$ be a faithfully flat cover between \textit{underived} supercommutative $\BA$-algebras. Recall that \cref{prop:t-structure} equips $\mathbf{SD}_\infty(\BA)$ with a $t$-structure whose heart is equivalent to the abelian category $\mathrm{SMod}_{\BA}$ of super $\BA$-modules. On the other hand, the stable $\infty$-category
\[  
\lim_{[n]\in\mathbf{\Delta}}\mathbf{SD}_\infty(\check{\mathrm{C}}(\varphi)_n)
\]
admits a $t$-structure where the full sub-$\infty$-category of connective objects is described as
\[
\left(\lim_{[n]\in\mathbf{\Delta}}\mathbf{SD}_\infty(\check{\mathrm{C}}(\varphi)_n)\right)^{\leqslant0}=\lim_{[n]\in\mathbf{\Delta}}\mathbf{SD}_\infty(\check{\mathrm{C}}(\varphi)_n)^{\leqslant0}\simeq\lim_{[n]\in\mathbf{\Delta}}\mathbf{SD}_\infty(\BT\otimes_{\BS}^{\mathrm{super}}\cdots\otimes_{\BS}^{\mathrm{super}}\BT)^{\leqslant0}.
\]
In the above formula, the super-tensor product is underived because $\BT$ is flat over $\BS$.  Flatness of $\BT$ over $\BS$ also implies that the coconnective objects are described as
\[
\left(\lim_{[n]\in\mathbf{\Delta}}\mathbf{SD}_\infty(\check{\mathrm{C}}(\varphi)_n)\right)^{\geqslant0}=\lim_{[n]\in\mathbf{\Delta}}\mathbf{SD}_\infty(\check{\mathrm{C}}(\varphi)_n)^{\geqslant0}\simeq\lim_{[n]\in\mathbf{\Delta}}\mathbf{SD}_\infty(\BT\otimes_{\BS}^{\mathrm{super}}\cdots\otimes_{\BS}^{\mathrm{super}}\BT)^{\geqslant0}.
\]
It follows that
\[
\left(\lim_{[n]\in\mathbf{\Delta}}\mathbf{SD}_\infty(\check{\mathrm{C}}(\varphi)_n)\right)^{\heartsuit}\simeq\lim_{[n]\in\mathbf{\Delta}}\mathbf{SD}_\infty(\BT\otimes_{\BS}^{\mathrm{super}}\cdots\otimes_{\BS}^{\mathrm{super}}\BT)^{\heartsuit}\simeq\lim_{[n]\in\mathbf{\Delta}}\mathrm{SMod}_{\BT\otimes_{\BS}^{\mathrm{super}}\cdots\otimes_{\BS}^{\mathrm{super}}\BT}.
\]
This is now a limit over a cosimplicial diagram computed in the $2$-category of abelian categories. This implies that we can equivalently compute it as the $2$-limit of the truncated diagram
\[
\begin{tikzcd}[column sep=large]
\mathrm{SMod}_{\BT}
  \arrow[r]
&
\mathrm{SMod}_{\BT\otimes^{\mathrm{super}}_{\BS}\BT}
  \arrow[r, shift left=0.6ex]
  \arrow[r, shift right=0.6ex]
&
\mathrm{SMod}_{\BT\otimes^{\mathrm{super}}_{\BS}\BT
  \otimes^{\mathrm{super}}_{\BS}\BT},
\end{tikzcd}
\]
which is precisely the category of descent data for $\varphi$. Thus, passing to the hearts in the statement of \cref{thm:flat_descent}, we recover the ordinary faithfully flat descent statement for abelian categories of supermodules.
\end{remark}
\bibliographystyle{siam}
\bibliography{bib}

\bigskip
\noindent
{\small Ugo Bruzzo \\
\address{Instituto de Ci\^encias Exatas, Universidade Federal de Minas Gerais, Av.~Ant\^onio Carlos 6627,  Belo Horizonte,
MG 30123‑970 (Brazil);
 IGAP (Institute for Geometry and Physics), Trieste (Italy)  } \\
\href{mailto:ubruzzo@ufmg.br}{\texttt{ubruzzo@ufmg.br}}}

\bigskip
\noindent
{\small Daniel Hern\'andez Ruip\'erez \\
\address{Departamento de Matem\'aticas and IUFFYM, Universidad de Salamanca, Plaza de la Merced 1--4, 37008, Salamanca (Spain)} \\
\href{mailto:ruiperez@usal.es}{\texttt{ruiperez@usal.es}}}

\bigskip
\noindent
{\small Emanuele Pavia \\
\address{Maison du Nombre, University of Luxembourg, 6, avenue de la Fonte L-4364, Esch-sur-Alzette (Luxembourg)} \\
\href{mailto:emanuele.pavia@uni.lu}{\texttt{emanuele.pavia@uni.lu}}}

\bigskip
\noindent
{\small Andrea T. Ricolfi \\
\address{SISSA, Via Bonomea 265, 34136, Trieste (Italy)} \\
\href{mailto:aricolfi@sissa.it}{\texttt{aricolfi@sissa.it}}}
\end{document}